\documentclass[reqno]{amsart}
\usepackage[utf8]{inputenc}
\usepackage{hyperref}
\usepackage{amsthm,amsmath,amssymb}
\usepackage{mathrsfs}
\usepackage{amsfonts}
\usepackage{tikz}
\usepackage{tkz-euclide}
\usepackage{amsthm}
\usepackage{amsmath,amscd}
\usepackage[all]{xy} 
\usepackage{graphicx}

\title[Laminations of Cubic Polynomials]{Real Laminations of Cubic Polynomials on Boundaries of Hyperbolic Components}

\author{Yueyang Wang} \thanks{Supported by the Fundamental Research Funds for the Central Universities} \address{School of Mathematics, Shanghai University of Finance and Economics, Shanghai, 200433, China} \email{yueyangwang@icloud.com}

\date{\today}

\newtheorem{thm}{Theorem}[section]
\newtheorem{cor}[thm]{Corollary}
\newtheorem{prop}[thm]{Proposition}
\newtheorem{lem}[thm]{Lemma}
\newtheorem{fact}[thm]{Fact}

\newtheorem{asp}{Assumption}[section]

\theoremstyle{remark}
\newtheorem{clm}{Claim}
\newtheorem{rmk}{Remark}[section]

\theoremstyle{definition}
\newtheorem{defi}{Definition}

\def\a{{\mathcal a}}

\def\S{\mathcal{S}}

\def\RZ{\mathbb{S}}
\def\Q{\mathbb{Q}}
\def\QZ{\Q/\Z}

\def\P{{\mathcal{P}_{3}}}
\def\CCC{{\mathcal{C}_{3}}}
\def\H{{\mathcal{H}(3)}}

\def\A{{\rm (A)}}
\def\B{{\rm (B)}}
\def\CCC{{\rm (C)}}
\def\C{\widehat{\mathbb C}}
\def\D{\mathbb{D}}

\def\L{\lambda}

\def\R{\mathbb{R}}
\def\C{\mathbb{C}}
\def\Z{\mathbb{Z}}

\def\AA{\mathcal{A}}
\def\p{\partial_{\mathrm{t}}}
\def\pp{\partial_{\mathrm{w}}}
\def\cc{{\mathrm{Crit}}}
\def\d{\mathrm{dist.}}
\def\RR{ L }
\def\g{\tilde{g}}
\def\O{\tilde{O}}
\def\i{\underline{\iota}}

\def\CC{{\widehat{\mathbb{C}}}}
\def\a{{\bf a}}

\def\M{{\rm Modu}}
\def\S{\mathcal{S}_p}

\def\LL{\Delta}

\def\H{\mathcal{H}}
\def\RRR{\mathrm{R}}
\def\M{\mathbf{M}}
\def\vt{\vartheta}

\begin{document}

\maketitle

\begin{abstract} 
    Milnor divides all bounded hyperbolic components of cubic polynomials into 4 types (A), (B), (C) and (D).
    In this article, we characterize the real laminations of cubic polynomials on the tame boundary of all bounded hyperbolic components of type (A), (B), or (C).
    For such maps, we prove that the real lamination is the smallest lamination which contains the real lamination of maps in the hyperbolic component and an equivalence relation generated by one characteristic equivalence class.
    As an application, we show that every hyperbolic cubic polynomial except type (D) is not combinatorially rigid.
\end{abstract}

\section{Background and Main Result}

\subsection{Polynomials and Laminations}
Dynamics of cubic polynomial maps have been studied intensively over the past 40 years.
Following the celebrated papers \cite{branner1988iteration,branner1992iteration} by Branner and Hubbard, we study the space of monic and centered cubic polynomials with a marked critical point in the Branner-Hubbard normal form
$$
    f_\a(z)=z^{3}-3c^{2}z+(2c^3+v),\qquad \a=(c,v) \in \C^2.
$$
They form the space $\P\simeq \C^2$.
Every cubic polynomial is conjugated to one in the Branner-Hubbard normal form by an affine map.

Every $f_\a \in \P$ has marked critical points $\pm c(\a)=\pm c$.
Let
$K_\a:=\{ z \in \mathbb{C}: \{f^n_\a(z) \} \text{ is bounded}  \}$
be its \emph{filled-in Julia set} and $J_\a:=\partial K_\a$ be its \emph{Julia set}.
The complement $\Omega_\a:=\C \setminus K_\a$ of the filled-in Julia set is the \emph{basin of infinity}.

According  Douady and Hubbard \cite{douady1984etude}, a useful tool in the study of polynomial dynamics is the B\"ottcher map and \emph{external rays} in $\Omega_\a$.
Let $\varphi_\a$ be the unique B\"ottcher map defined near $\infty$ with $\varphi_\a(\infty)=\infty$ and the property
$$\varphi_\a(f_\a(z))=(\varphi_\a(z))^3,\qquad \lim_{z \to \infty}\frac{\varphi_\a(z)}{z}=1.$$ 
If $K_\a$ is connected, then the inverse map $\psi_\a:=\varphi_\a^{-1}$ can be extended to a conformal map $\psi_\a : \C \setminus \overline{\D} \to \C \setminus K_\a$.
For $\theta \in \mathbb{R}/\mathbb{Z}$, the \emph{external ray} with angle $\theta$ is defined as $$R_\a(\theta):=\psi_\a((1,+\infty)e^{2\pi i \theta}).$$
For $\theta \in \mathbb{R}/\mathbb{Z}$, if the external ray $R_\a(\theta)$ is defined and $\lim_{s\to 1}\psi_\a(se^{2\pi i\theta})=z$, then we say that the external ray $R_\a(\theta)$ \emph{lands} at $z$.
It is well-known that every external ray $R_\a(\theta)$ with rational angle $\theta \in \Q/\Z$ lands at a repelling or parabolic preperiodic point if $K_\a$ is connected

Following McMullen \cite{McMullen1995TheCO}, we consider the equivalence relation $\L_\Q(\a)$ which encodes the landing pattern of these external rays with rational angles.
That is, $(\theta,\theta')\in \L_\Q(\a)$ if and only if $R_\a(\theta)$ and $R_\a(\theta')$ land at a common point.
This equivalence relation $\L_\Q(\a)$ is called the \emph{rational lamination} of $f_\a$.

Furthermore, if $J_a$ locally connected, then by the Carath\'eodory Theorem,  
$\psi_\a$ can be extended to a continuous map $\partial \D \to J_\a$.
Hence every external ray lands at a point in $J_\a$.
Thus, we can define the \emph{real lamination} $\L_\R(\a) \subset \RZ \times \RZ$ of $f_\a$ such that  $(\theta,\theta')\in \L_\R(\a)$ if and only if $R_\a(\theta)$ and $R_\a(\theta')$ land at a common point.
In this case, the Julia set has a topological model $J_\a \simeq \RZ/\L_\R(\a)$.
Moreover, we have
\begin{equation*}
    \xymatrix{\ar @{} [dr]
   \RZ \ar[d]^{\psi_\a} \ar[r]^{\tau_d} & \RZ \ar[d]^{\psi_{\a}} \\
   J_\a \ar[r]^{f_\a}        & J_{\a} .     }
\end{equation*}
Thus, the combinatorial and topological aspects of the dynamics on $J_\a$ is encoded in its real lamination.
The rational and real laminations of polynomials of degree $d \geq 2$ are studied by Kiwi \cite{kiwi2001rational,wanderingorbitportrait,KIWI2004207,CombinatorialContinuity}.




In the parameter space $\C^2$, all $\a \in \C^2$ with $K_\a$ connected form the \emph{connected locus} $\mathcal{C}_3$.
The combinatorial structure of cubic connectedness locus and related topics are studied by the group of Blokh in \cite{blokh2005necessary,blokh2013cubic,blokh2013laminations,blokh2016quadratic,blokh2017combinatorial,blokh2017parameter,blokh2018perfect}.

In this article, we study the lamination of cubic polynomials on the boundaries of hyperbolic components.
A map $f_\a$ is called \emph{hyperbolic} if the orbits of $\pm c(\a)$ converge to attracting cycles.
Parameters $\a$ with $f_\a$ hyperbolic forms the \emph{hyperbolic locus}.
Its connected components are called \emph{hyperbolic components}.
hyperbolic components in connectedness locus are called \emph{bounded hyperbolic components}.
Milnor \cite{milnor2008cubic} divides them into $4$ types based on the orbit of the two finite critical points.
\begin{itemize}
    \item[(A)] Both finite critical points belong to the same Fatou component.
    \item[(B)] Two finite critical points belong to different Fatou components of the same attracting cycle.
    \item[(C)] Only one finite critical point belongs to the immediate basin of an attracting cycle. The other finite critical point belongs to the basin but not the immediate basin of this attracting cycle.
    \item[(D)] Two finite critical points are attracted by distinct attracting cycles.
\end{itemize}
Among them, hyperbolic components of type (D) are exceptional.
For $\a$ in a hyperbolic component of type (D), the dynamics of $f_\a$ can be decomposed into two ``disjoint'' quadratic polynomials.
In the following discussion, we exclude this case.

Let $\mathcal{H}$ be a bounded hyperbolic component of type \A, \B~or \CCC.
It is well-known that $J_\a$ is locally connected for $\a \in \mathcal{H}$.
According to Man\'e-Sud-Sullivan \cite{mane1983dynamics}, the real lamination $\L_\R$ takes a constant $\L_\R(\mathcal{H})$ on $\mathcal{H}$.

The dynamics of the maps on the boundary is more difficult to understand.
According to the classification of Peterson-Tan \cite{petersen2009analytic}, the boundary of a bounded hyperbolic component divides into two parts: the \emph{tame} part $\p \mathcal{H}$, consisting of the boundary points which possess an attracting cycle, and its complement the \emph{wild} part $\pp \mathcal{H}$.
According to Roesch \cite{roesch2007hyperbolic} and X.Wang \cite{wang2017hyperbolic}, the dynamics of $f_\a$ for $\a \in \p \mathcal{H}$ is now not so mysterious anymore.
Without loss of generality, we assume that the critical point $c(\a)$ is in the immediate attracting basin $B_\a$.
The attracting basin is denoted by $\tilde{B}_\a$.
According to \cite{wang2017hyperbolic,wang2023} (Theorem \ref{wang} and Proposition \ref{multiplier}), $\a \in \p \H$ can be classified into two cases:
\begin{itemize}
    \item {\bf parabolic case}: There is a parabolic periodic orbit attached on $\partial B_\a$;
    \item {\bf non-parabolic case}: $-c(\a)$ is on the boundary of some component in $\tilde{B}_\a$.
\end{itemize}
In particular, $J_\a$ is locally connected in both cases.
Hence the real lamination $\L_\R(\a)$ makes sense for every $\a \in \overline{\H}$.
(This is another reason why we exclude the type (D) case since boundary maps in a type (D) hyperbolic components may not have locally connected Julia set.)

Pictures of various examples of boundary maps are listed in the following.
Figure \ref{fig1} is a critically infinite (non-parabolic) example for $\a \in \p \H_0$ where $\H_0$ is the \emph{central hyperbolic component} containing $z \mapsto z^3$.
Figure \ref{fig2} and Figure \ref{fig3} are non-parabolic examples of the boundary points of type (C) components with an attracting fixed point.
Figure \ref{fig2} is a critically infinite example, while Figure \ref{fig3} is a critically finite example.
Figure \ref{fig4} is critically finite example of an intersection point of a hyperbolic component of type (A) and a hyperbolic component of type (C).
Figure \ref{fig5} and Figure \ref{fig6} are parabolic examples.
Figure \ref{fig5} represents a boundary point of $\H_0$ which is not on the boundary of any other hyperbolic components, while Figure \ref{fig6} represents an intersection point of a hyperbolic component of type (A) and a hyperbolic component of type (B).

 \begin{figure}[h]
  \begin{center}
   \vspace{2mm}
   \begin{minipage}{.31\linewidth}
    \includegraphics[width=\linewidth]{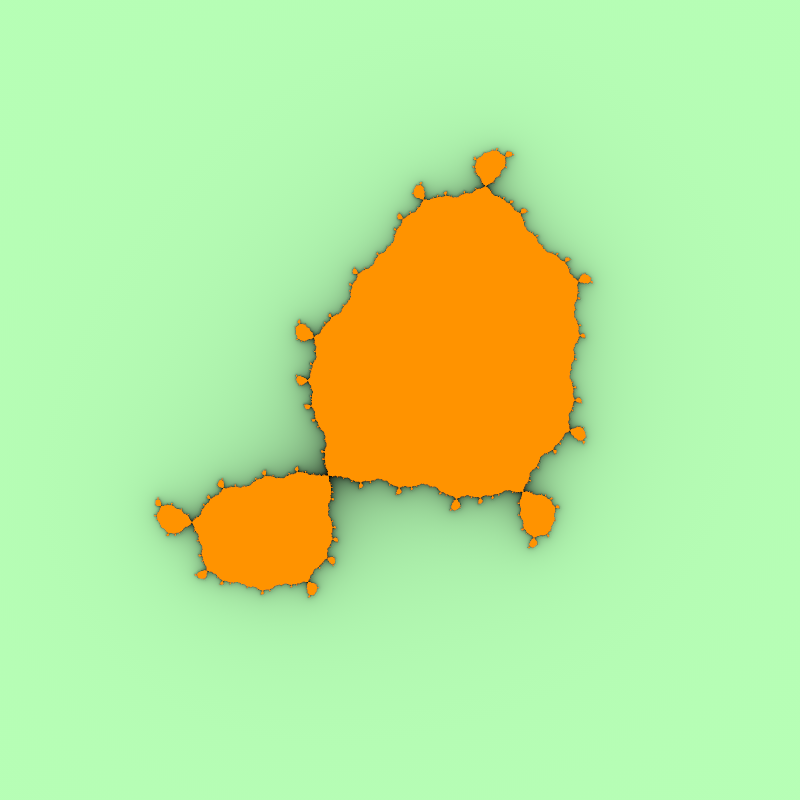}
    \caption{}\label{fig1}
  \end{minipage}
  \hspace{1mm}
  \begin{minipage}{.31\linewidth}
    \includegraphics[width=\linewidth]{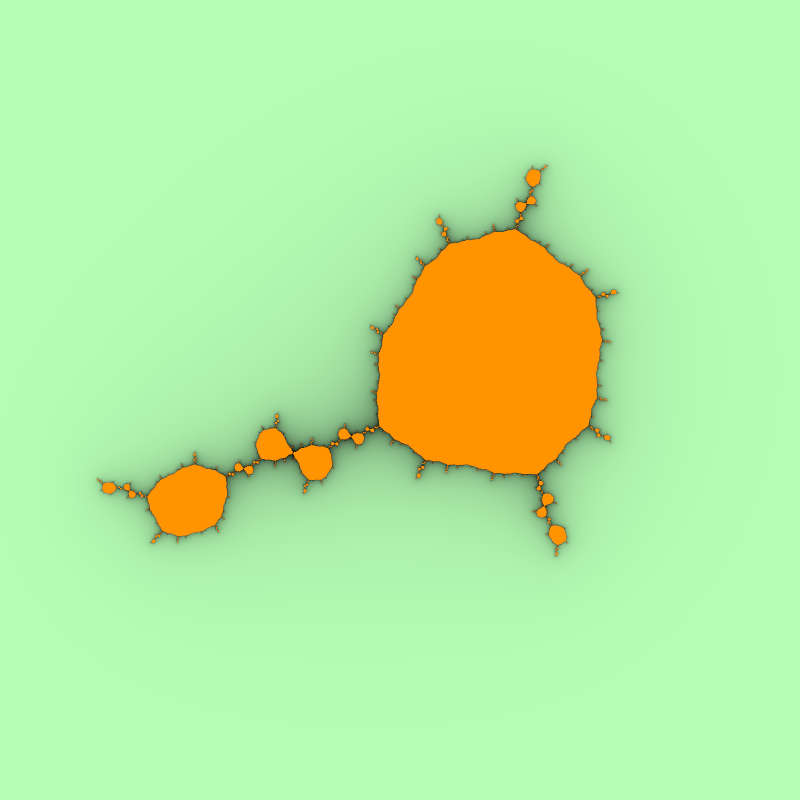}
    \caption{}\label{fig2}
    \end{minipage}
   \hspace{1mm}
   \begin{minipage}{.31\linewidth}
    \includegraphics[width=\linewidth]{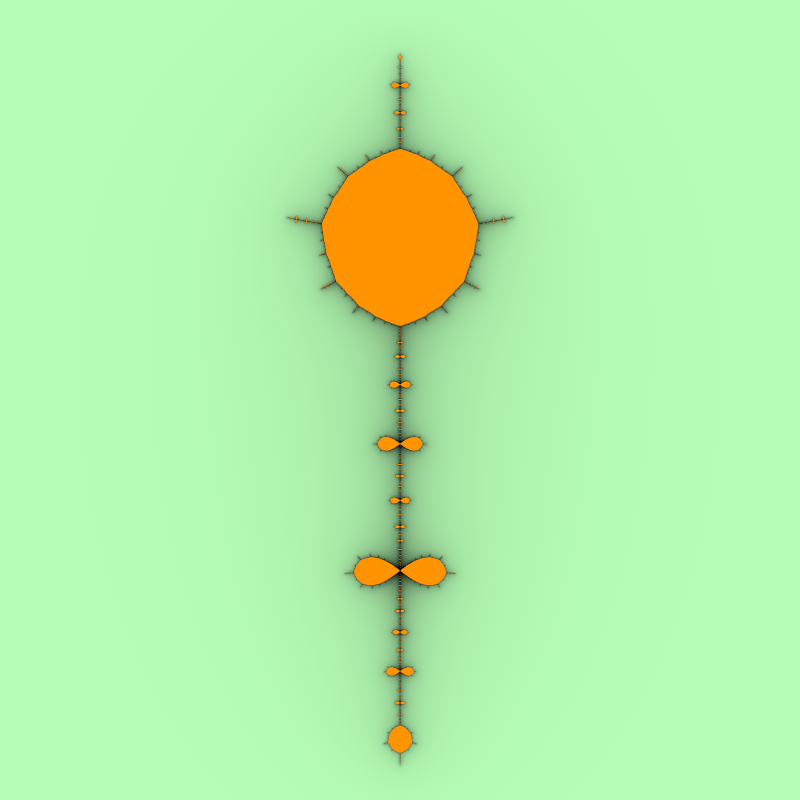}
    \caption{} \label{fig3}
\end{minipage}
  
  \vspace{6mm}
  \begin{minipage}{.31\linewidth}
    \includegraphics[width=\linewidth]{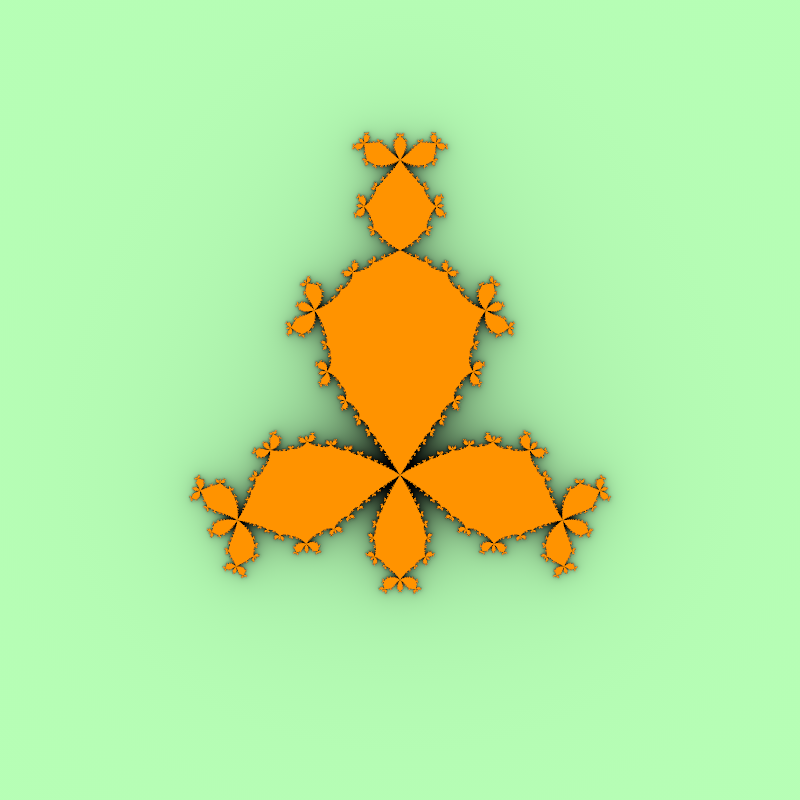}
    \caption{} \label{fig4}  \end{minipage}
  \hspace{1mm}
  \begin{minipage}{.31\linewidth}
    \includegraphics[width=\linewidth]{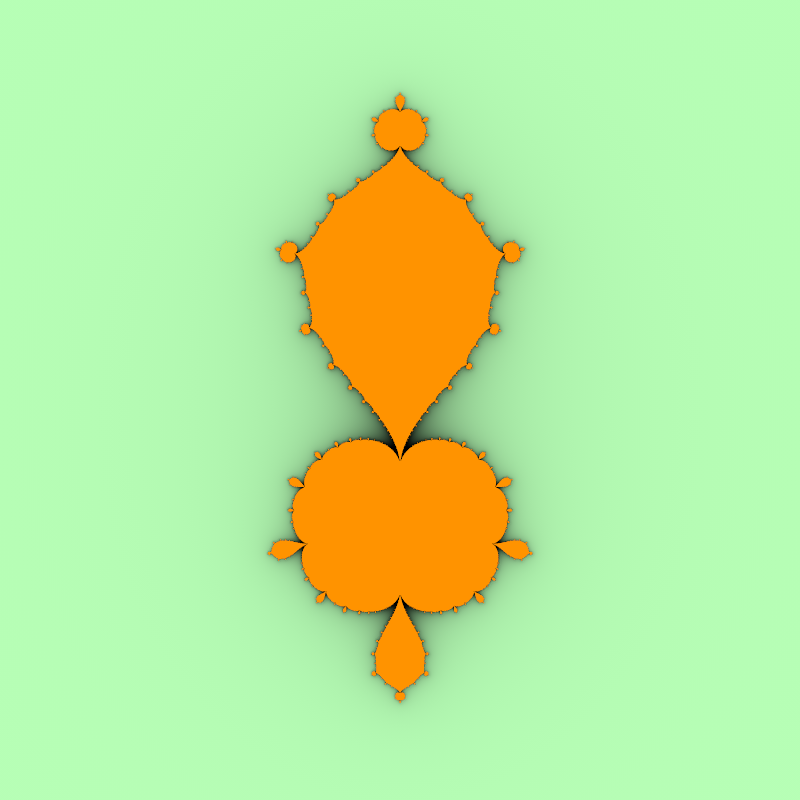}
    \caption{} \label{fig5}  \end{minipage}
   \hspace{1mm}
   \begin{minipage}{.31\linewidth}
    \includegraphics[width=\linewidth]{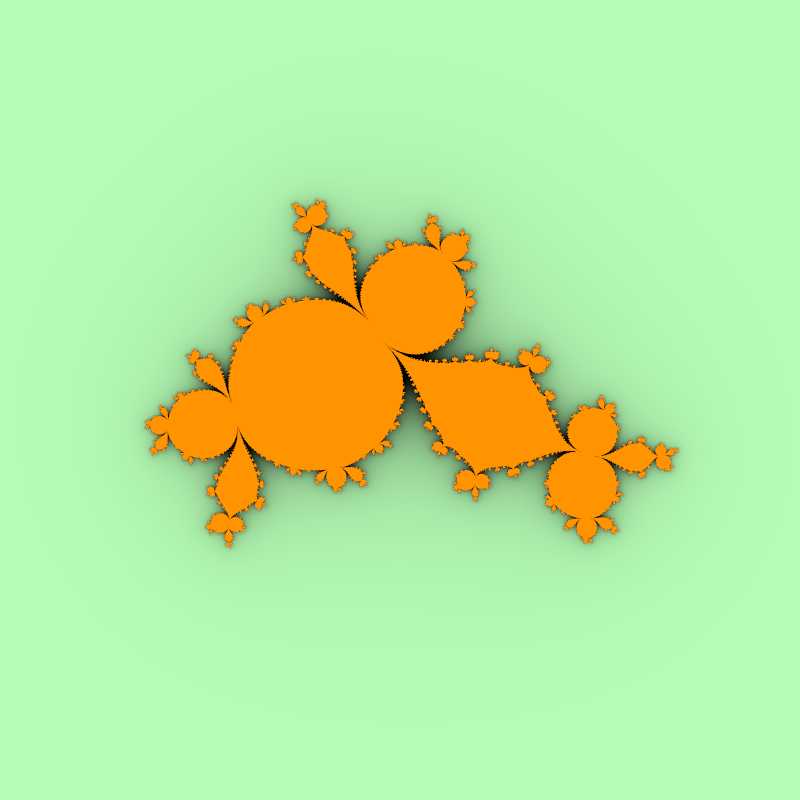}
    \caption{} \label{fig6}
\end{minipage} \end{center}
\end{figure}

Beyond the dynamical system, the topological and geometric structure of the closures of bounded hyperbolic components themselves are also a major topic and deeply studied in \cite{milnor2012hyperbolic,petersen200910,roesch2007hyperbolic,wang2017hyperbolic,cao2022boundaries,cao2025boundaryI,cao2025boundaryII,luo2023geometrically,gao2025boundaries}.

\subsection{Main Results}
This article aims to characterize the relations between real laminations $\L_\R(\H)$ of cubic polynomials in a hyperbolic component $\H$ and the real laminations of maps on its tame boundary.
We will show that the real lamination $\L_\R(\a)$ of any $\a \in \p \H$ contains $\L_\R(\H)$, and additional part is very ``small".
To precisely describe the the phrase ``small", we need the following terminologies.


An equivalence relation $\Lambda \subset \RZ \times \RZ$ is called \emph{$\tau_3$-invariant} if $\tau_3$ maps an equivalence class of $\Lambda$ to an equivalence class $\Lambda$.
A $\tau_3$-invariant equivalence relation $\Lambda$ is called \emph{minimal}, if there exists a equivalence class $E^*$ of $\Lambda$ such that every non-singleton equivalence class $E$ is eventually mapped to $E^*$ under $\tau_3$.
Here $E^*$ is called the \emph{generator} of $\Lambda$.

For $L \subset \RZ \times \RZ$, let $\langle L \rangle$ denote the smallest equivalence relation on $\RZ$ containing $L$.
The main result of this article is the following.

\begin{thm}[lamination of maps on tame boundary]\label{thmmain2}
    Let $\mathcal{H}$ be a bounded hyperbolic component of type \A, \B~ or \CCC.
    For any $\a \in \p \mathcal{H}$, there exists a minimal $\tau_3$-invariant equivalence relation $\Lambda(\a) \subset \RZ \times \RZ$ such that
    $$\L_\R(\a) = \langle \Lambda(\a) \cup \L_\R(\mathcal{H}) \rangle.$$
    In particular, we have $\L_\R(\mathcal{H})\subsetneq \L_\R(\a)$.
\end{thm}


Let $\a^* \in \mathcal{H}$ be the unique post critically finite map in $\mathcal{H}$. 
By Theorem \ref{thmmain2}, we deduced the following semi-conjugacy statement.

\begin{cor}[semi-conjugacy]\label{semiconjugacy}
    For any $\a \in \p \mathcal{H}$, there exists a non-injective continuous semi-conjugacy $\pi_\a : J_{\a^*} \to J_\a$ such that  
\begin{equation*}
    \xymatrix{\ar @{} [dr]
 \partial \D \ar[d]^{\tau_3} \ar[r]_{\psi_{\a^*}} \ar@/^1pc/[rr]^{\psi_\a} &  J_{\a^*} \ar[d]^{f_{\a^*}} \ar[r]_{\pi_\a} & J_\a \ar[d]^{f_{\a}} \\
 \partial \D  \ar[r]^{\psi_{\a^*}} \ar@/_1pc/[rr]_{\psi_\a}  &   J_{\a^*} \ar[r]^{\pi_\a}        & J_\a .     }
\end{equation*}
\end{cor}

To the best knowledge of the author, it is not known whether the statement $\L_\R(\mathcal{H})\subset \L_\R(\a)$ holds for $\a \in \pp \mathcal{H}$ for $\H$ of type \A, \B, or \CCC.
While, counterexamples are constructed for $\a \in \pp \mathcal{H}$ for $\H$ of type (D).
See \cite[Example 8.5, Example 8.6]{inou2012combinatorics} for details.
Conjecturally, such counterexamples can also be constructed for wild boundaries of hyperbolic components of type (A), (B) and (C).

Since we focus on the tame boundary map $\a \in \p \H$, we may assume that the attracting cycle that $f_\a$ possess has period $p$.
By changing the multiplier of the $p$ attracting cycle using surgeries (Proposition \ref{multiplier}), we may assume that the $p$ attracting cycle is super-attracting.
This allows us to restrict our discussion in Milnor's periodic $p$ super-attracting slices 
$$\mathcal{S}_{p}:=\left \{ \a \in \mathbb{C}^2 : c(\a) \text{ is a $p$-periodic point of $f_\a$} \right \}.$$
It is first introduced by Milnor \cite{milnor2008cubic}, and is studied by many people, see \cite{roesch2007hyperbolic, dujardin2008distribution, Bonifant_2010, arfeux2016irreducibility, wang2017hyperbolic}.
The pictures of $\S$ for $p=1,2,3$ are shown form Figure \ref{S1} to Figure \ref{S3}.


Let $\mathcal{H} $ be a hyperbolic component of type \A, \B~or \CCC, then $\mathcal{H} \cap \S$ forms a bounded hyperbolic component in $\S$.
Theorem \ref{thmmain2} is obtained through a careful study of the real laminations of cubic polynomials in $\partial \H \cap \S$.
 \begin{figure}[h]
  \begin{center}
   \vspace{2mm}
   \begin{minipage}{.31\linewidth}
    \includegraphics[width=\linewidth]{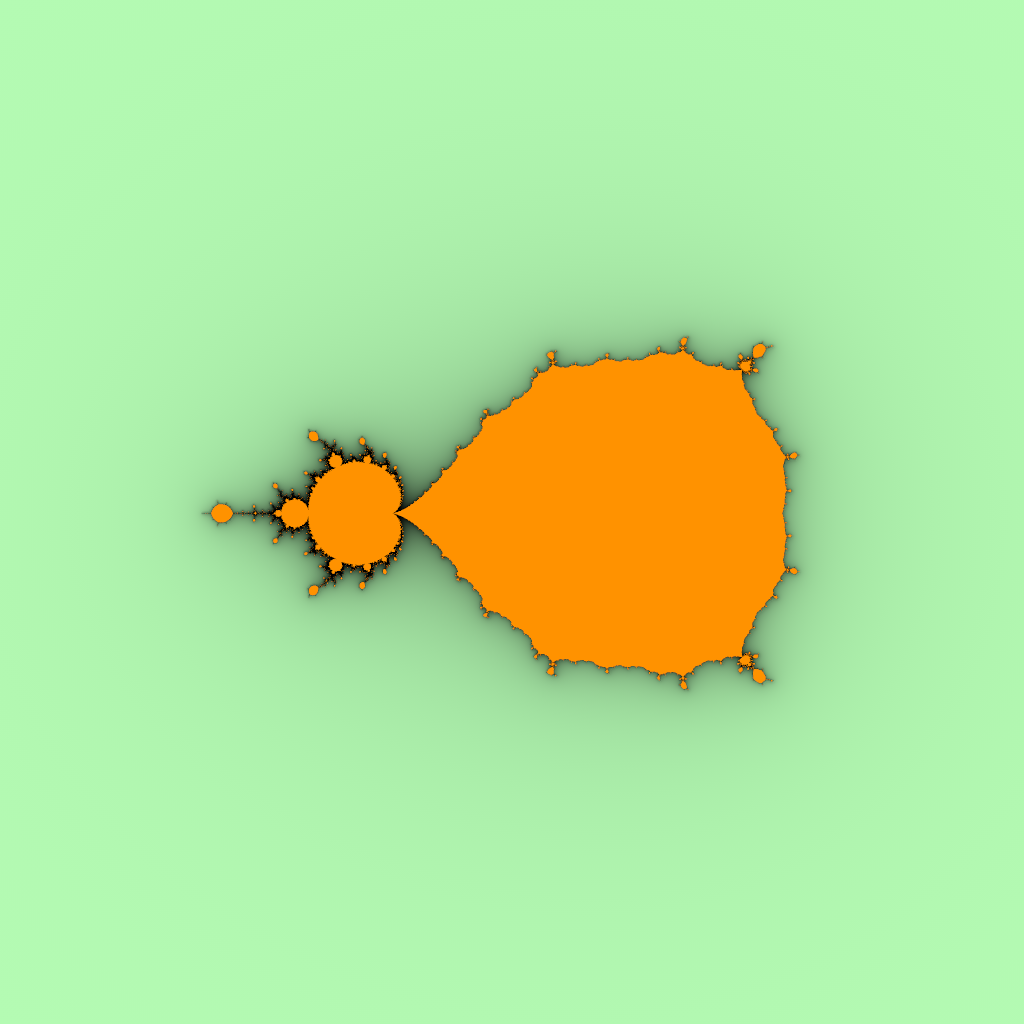}
    \caption{$\mathcal{S}_1$}\label{S1}
  \end{minipage}
  \hspace{1mm}
  \begin{minipage}{.31\linewidth}
    \includegraphics[width=\linewidth]{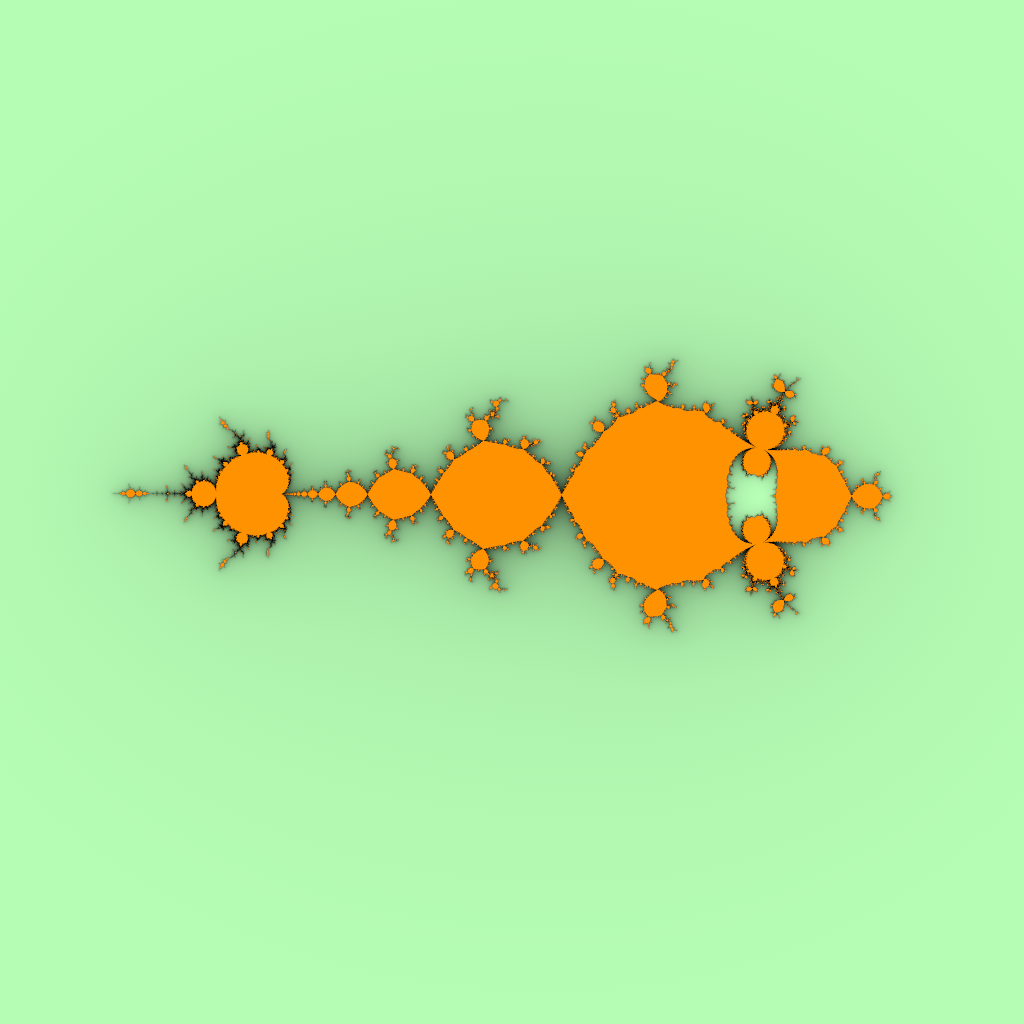}
    \caption{$\mathcal{S}_2$}\label{S2}
  \end{minipage}
   \hspace{1mm}
   \begin{minipage}{.31\linewidth}
    \includegraphics[width=\linewidth]{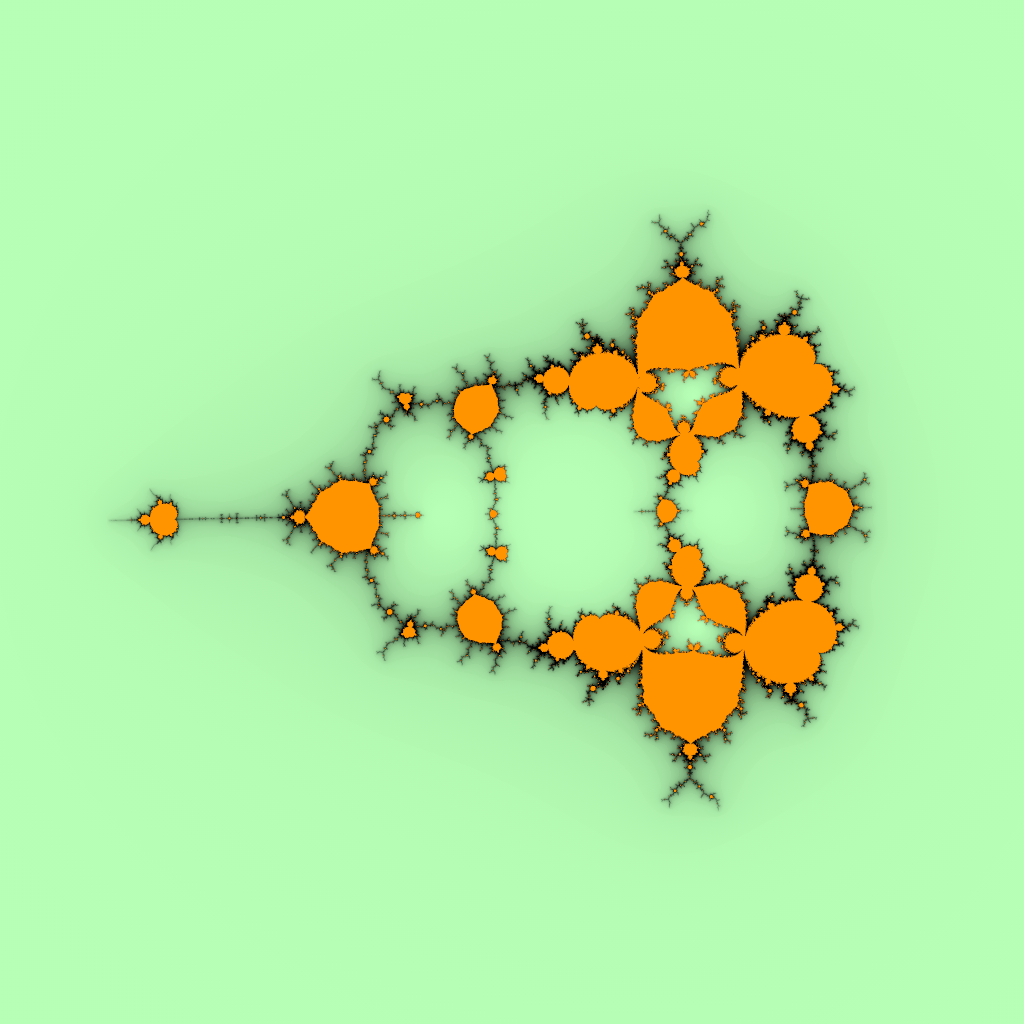}
    \caption{$\mathcal{S}_3$}\label{S3}
  \end{minipage}
 \end{center}
\end{figure}

As an application, we use our result to study the combinatorial rigidity of cubic polynomials.
Combinatorial rigidity of polynomials is studied by many people, see \cite{McMullen1995TheCO,Henriksen2003TheCR,avila2009combinatorial}.
A polynomial $f_{\a_0}$ of degree $d$ with no indifferent cycles is said to be \emph{combinatorially rigid} if for every polynomial $f_\a$ with no indifferent cycles and $\L_\Q(\a)=\L_\Q(\a_0)$, the composition $\psi_\a\circ\varphi_{\a_0}$ of the B\"ottcher maps extends to a quasiconformal map on $\CC$.

For each degree $d\geq 2$, it was conjectured that every polynomial with non indifferent cycles is combinatorially rigid.
This is called the Combinatorial Rigidity Conjecture.
In degree $d=2$, the Combinatorial Rigidity Conjecture implies the Density of Hyperbolicity Conjecture.
It is known that the combinatorial rigidity holds for every non infinite renormaliziable quadratic polynomial.
But in general, it is still open.

In degree $d=3$, Henriksen \cite{Henriksen2003TheCR} constructed a counterexample and showed the Combinatorial Rigidity Conjecture is false in degree $3$.
His counterexample involves infinite renormaliziable polynomials with non-locally connected Julia set.
Using Theorem \ref{thmmain2}, we obtain the following result.

\begin{thm}[combinatorial rigidity for cubic polynomials]\label{degree3}
    Every hyperbolic cubic polynomial which is not type {\rm (D)} is not combinatorially rigid.
\end{thm}


\subsection{Sructure of the Article and Strategy of the Proofs}
The strategy of the proof of Theorem \ref{thmmain2} is the following.
The major part of our discussion lies in the super-attracting slice $\S$.
Our idea is trying to observe how and why the real lamination changes when the parameter approaches the boundary.

In Section \ref{sec3}, we develop a new theory of non-smooth generalized internal rays over the model space $\O \times \RZ$ for a hyperbolic map in a hyperbolic component $\H$ of type (A), (B) or (C).
Here $\O$ denote the grand orbit $\O$ of the $p$ super-attracting cycle.
It is well-known that given each $(v,t) \in \O \times \RZ$, we can define the canonical B\"ottcher internal arc until it hits an iterative preimage $\omega(\a)$ of the critical point $-c(\a)$.
What we do is to pass through $\omega(\a)$ and continue the canonical extension to reach the boundary.
The Jordan arc we obtained is called a \emph{generalized internal ray}.
Each time (either $1$ or infinite times) when we pass through a pre-critical point, there are two directions $\iota \in \{ L,R \}$ (which represent the turning direction: left or right).
Thus, we can define various generalized internal rays based on how it turns at each pre-critical point.
There are two particular kinds of generalized internal rays which are called the \emph{left and right internal ray}s. They turn a particular direction (left or right) at every pre-critical points. 


The landing properties of the left and right internal rays help us to guess the lamination of the boundary maps
As we can observe in the computer generated pictures, the external rays which are joined by a pair of left and right internal rays are stretched closer and closer as the parameter approaches $\a_0 \in \partial \H$ along a certain canonical parameter ray.
Thus, a new but simple invariant equivalence relation $\Lambda(\a_0)$ appears in our vision.
For this reason, by adding $\Lambda(\a_0)$ to $\L_\R(\H)$, we call 
$\Lambda_\R(\a_0):=\langle \Lambda(\a_0) \cup \L_\R(\H) \rangle$
the \emph{visual lamination} of $\a_0$.
This definition is pure combinatorial.
To this end, surprisingly, we are able to prove that the visual lamination is indeed a lamination only by using the combinatorial definition without showing that it is the real lamination of some cubic polynomial (Proposition \ref{visuallamination}).

From Subsection \ref{ssec4.3}, show that the visual lamination $\Lambda_\R(\a_0)$ is indeed the real lamination $\L_\R(\a_0)$ of $\a_0$.
The verification is a difficulty of this article.
We employee the puzzle techniques as well as the its perturbation theories.
Another key ingredient in the proof is Peterson and Zakeri's result \cite{hausdorfflimit} on the Huasdorff limit of external rays when it approaches a parabolic parameter. 
This is carefully discussed in Subsection \ref{ssec2.4}.

Finally, quasiconformal surgeries allows us to transfer the result in $\S$ to the tame boundary of $\H$.
This is done in Subsection \ref{ssec4.6} where the main theorems are proved.





\section{External Rays and Laminations}


\subsection{Notations}
Let $\mathbb{C}$ denote the complex plane.
The \emph{Riemann's sphere} $\CC:=\mathbb{C} \cup \{\infty\}$ is the one point compactification of $\mathbb{C}$. 
Denote $\mathbb{D}:=\{z \in \C: |z|<1 \}$.

Denote $\RZ:=\R/\Z$
Let $\pi: \mathbb{R} \to \RZ$ be the natural projection map.
$\theta_{1},\theta_{2}, \cdots \theta_{n} \in \RZ$ are said to be in \emph{positive cyclic order} if there exist $\tilde{\theta}_{1}, \tilde{\theta}_{2}, \cdots, \tilde{\theta}_{n} \in \mathbb{R}$ such that $\pi(\tilde{\theta}_{k})=\theta_{k}$ for $1 \leq k \leq n$ and
\begin{equation*}
    \tilde{\theta}_{1}<\tilde{\theta}_{2}< \cdots <\tilde{\theta}_{n}<\tilde{\theta}_{1}+1.
\end{equation*}
For any $\theta_{1}, \theta_{2} \in \mathbb{R}/\mathbb{Z}$, the \emph{positive open interval} is defined to be
$$
    (\theta_{1},\theta_{2})_{+}:=\{\theta \in \mathbb{R}/\mathbb{Z}: \theta_{1},\theta,\theta_{2}~\mbox{are in positive cyclic order}\}.
$$
The \emph{length} of the interval $(\theta_{1},\theta_{2})_{+}$ is defined to be
$
    \LL (\theta_{1},\theta_{2})_{+}:=\tilde{\theta}_{2}-\tilde{\theta}_{1},
$
where $\tilde{\theta}_{1}, \tilde{\theta}_{2}$ are representatives of $\theta_{1},\theta_{2}$ such that $\tilde{\theta}_{1}<\tilde{\theta}_{2}<\tilde{\theta}_{1}+1$.
Similarly, one can define other kinds of intervals on $\RZ$.
The length of these intervals are all defined to be $ \LL (\theta_{1},\theta_{2})_{+}$.
Let $E \subset \RZ$ be a finite set, define 
$$\Delta E:=\min \{ \Delta (\theta,\theta')_+ :  \theta,\theta' \in E  \}.$$

For $d \geq 2$, let $\tau_d: \RZ \to \RZ$ be the induced map of the map $\theta \mapsto d \theta : \mathbb{R} \to \mathbb{R}$ on $\RZ$.

\subsection{External rays of Polynomials}\label{polynomialbasin}
Let $\AA$ be complex manifold.
A holomorphic map $f: \AA \times \C \to \C, ~ (\a,z) \mapsto f_\a(z)$ is called a \emph{holomorphic family of polynomials} of degree $d \geq 2$ if for any $\a \in \AA$, $f_\a$ is a polynomial map of degree $d$.
It is usually denoted by $\mathcal{F}=\{ f_\a\}_{\a \in \AA}$.
Let $\cc(\a)$ denote the set of critical points of $f_\a$.

For $\a \in \AA$,
$\infty$ is a fixed critical point of $f_\a$ with multiplicity $d$.
Let 
$
    K_\a:=\{z \in \mathbb{C} : \{ f_\a^n(z) \} \mbox{ is bounded} \},
$
denote the filled-in Julia set, $J_\a:=\partial K_\a$ be the Julia set.
Then $\Omega_\a:= \C \setminus K_\a$ is the \emph{attracting basin} of $\infty$.
The Green function $G_\a : \C \to [0,+\infty)$ of $f_\a$ is defined to be 
$$
    G_\a(z):=\lim_{n \to \infty} \frac{1}{d^n} \log_+ |f_\a^n(z)|
$$
where $\log_+:=\max \{\log , 0 \}$.
For any $s>1$, define $\Omega_{\a}(s):=\{ z \in \Omega_\a : \exp \circ G_\a (z) > s \}$.
Let 
$
    s_{\a}:=\max \{\exp \circ G_{\a}(c): c \in \Omega_\a \cap \cc(\a) \}.
$
Let $\mathscr{B}:=\{(\a,z)\in \AA \times \CC: z \in \Omega_{\a}(s_{\a}) \}$.
There exists a holomorphic map $\varphi:  \mathscr{B} \to \CC \setminus \overline{\mathbb{D}},~(\a,z) \mapsto \varphi_\a(z)$ such that for any $\a \in \AA$ and $z \in \Omega_{\a}(s_{\a})$, we have 
\begin{equation}\label{boteq3}
    \varphi_\a(f_\a(z))=(\varphi_\a(z))^d,\qquad \lim_{z \to \infty}\frac{\varphi_\a(z)}{z}=1. 
\end{equation}
Fix $\a\in \AA$, the map $\varphi_\a$ is called the \emph{B\"ottcher map} of $f_\a$ with respect to $\infty$.  
Denote $\psi_{\a}$ to be the inverse of $\varphi_{\a}$.
By \eqref{boteq3}, we have 
\begin{equation}\label{boteq4}
   f_\a \circ \psi_{\a} (w)=\psi_{\a} (w^d),
\end{equation}
According to \cite{douady1984etude}, there exists a unique $s_{\a}(\theta) \in [1, {s_{\a}}]$ such that $\psi_\a$ admits a canonical (means that \eqref{boteq4} still holds) smooth extension to $(s_{\a}(\theta),+\infty)e^{2\pi i \theta}$ without hitting any iterative preimages of critical points.
If $s_{\a}(\theta)=1$, we define the \emph{external ray}  with angle $\theta$ is defined to be
$$
    R_{\a}(\theta):=\psi_{\a}((1,+\infty)e^{2\pi i \theta}).
$$
If $\lim_{s \to 1} \psi_{\a}(se^{2\pi i \theta})$ exists, we say that the external ray $ R_{\a}(\theta)$ \emph{land}s at $z_\theta(\a):=\psi_{\a}(e^{2\pi i \theta})$.
If $\theta \in \mathbb{Q}/\mathbb{Z}$ and $s_{\a}(\theta)=1$, then $R_{\a}(\theta)$ lands at an eventually repelling or parabolic periodic point.
If $K_\a$ is connected, then for any eventually repelling or parabolic periodic point $z$, there exist at least one, and at most finitely many $\theta \in \mathbb{Q}/\mathbb{Z}$ such that $R_\a(\theta)$ lands at $z$.

Suppose that $R_\a(\theta)$ and $R_\a(\theta')$ land  at a common point, define the \emph{corresponding sector} $S_\a(\theta,\theta')_+$ of $(\theta,\theta')_+$ to be the connected component of  $\C \setminus  (\overline{R_\a(\theta)} \cup \overline{R_\a(\theta')})$.
Then the followings are equivalent.
\begin{enumerate}
    \item The angular width $\LL(\theta,\theta')_+<1/d$, and $\LL (\tau_d(\theta),\tau_d(\theta'))_+ = d \LL(\theta,\theta')_+ $.
    \item The sector $S_\a(\theta,\theta')_+$ does not contain any critical point of $f_\a$.
    \item The map $f_\a : S_\a(\theta,\theta')_+ \to S_\a(\tau_d(\theta),\tau_d(\theta'))_+$ is a conformal map.
\end{enumerate}

We also need Roesch-Yin's characterization of limbs structure of the filled-in Julia set with respect to the bounded Fatou components.

\begin{thm}[Roesch-Yin \cite{roesch2008boundary}]\label{roeschyin}
    Let $f_\a$ be a polynomial of degree $d$ with connected Julia set, and $B$ be a bounded Fatou component of $f_\a$.
    For each $x \in \partial B$, there exists a compact connected set $L_x \subset K_\a \setminus B$ such that $L_x \cap \overline{B}=\{ x\}$ and
    $$K_\a=B \cup \left(\bigsqcup_{x \in \partial B} L_x \right).$$
    Moreover, the following holds.
   \begin{enumerate}
       \item If $L_x=\{ x\}$, then there exists a unique external ray which lands at $x$.
       \item If $L_x\neq \{ x \}$, then there exist two external rays $R_\a(\theta^-)$, $R_\a(\theta^+)$ land at $x$ such that they together with their common landing point separate $L_x$ away from $B$.
       If $B$ is a fixed Fatou component of $f_\a$, then there exists $n$ such that $L_{x_n}$ contains a critical point of $f_\a$ where $x_n:=f_\a^n(x)$.
   \end{enumerate}
\end{thm}

\subsection{Lamination}
In this part, we briefly survey the theory of lamination.
One may refer to \cite{kiwi2001rational,KIWI2004207} for a complete introduction.

Following Kiwi, for $X \subset \RZ$ which is invariant under $\tau_d$,  a  \emph{$\tau_d$ invariant lamination} $\L\subset X \times X$ on $X$ is an equivalence relation such that the followings hold.
\begin{enumerate}
        \item[(R1)]\label{lr1} $\L$ is closed in $X \times X$.
        \item[(R2)]\label{lr2} Each $\L$-equivalence class is a finite set.
        \item[(R3)]\label{lr3} If $E$ is a $\L$-equivalence class, then $\tau_d(E)$ is also a $\L$-equivalence class.
        \item[(R4)]\label{lr4} $\tau_d: E \to \tau_d(E)$ is \emph{consecutive-preserving}, i.e. for any component $(\theta,\theta')_+$ of $\RZ\setminus E$, $(\tau_d(\theta),\tau_d(\theta'))_+$ is a  component of $\RZ \setminus \tau_d(E)$.
        \item[(R5)]\label{lr5} $\L$-equivalence classes are \emph{pairwise unlinked}, i.e. for any two equivalence class $E \neq E'$, $E$ is contained in a single complementary arc of $\RZ \setminus E'$.
\end{enumerate}
These properties are called the \emph{axioms of lamination}.
In the following, a $\tau_d$-invariant lamination is briefly referred as lamination.

Recall that if $f_\a$ is a polynomial of degree $d$ with connected Julia set.
Then the \emph{rational lamination}  is defined by 
$$
    \L_\Q(\a):=\{ (\theta,\theta') \in (\Q/\Z)^2: R_\a(\theta), R_\a(\theta') \mbox{ land at a common point} \}.
$$
Furthermore, if $J_\a$ is locally connected, then the \emph{real lamination}  is defined by
$$
    \L_\R(\a):=\{ (\theta,\theta') \in \RZ \times \RZ: R_\a(\theta), R_\a(\theta') \mbox{ land at a common point} \}.
$$
One may curious whether the rational lamination and the real lamination is indeed  a ($\tau_d$-invariant) lamination?
Kiwi obtain the following result.

\begin{thm}[Kiwi \cite{kiwi2001rational,KIWI2004207}]\label{kiwicontain}
    Let $f_\a$ be a polynomial with connected an locally connected Julia set.
    The following holds.
    \begin{itemize}
        \item The rational lamination $\L_\Q(\a)$ is a lamination.
        \item If $f_\a$ do not have irrationally neutral cycle, then $\L_\R(\a)$ is a lamination. 
        \item If no critical point of $f_\a$ with infinite orbit lies on the boundaries of Fatou components, then $\L_\R(\a)$ is the smallest closed equivalence relation containing $\L_\Q(\a)$.
    \end{itemize}
    
\end{thm}


\subsection{Huasdorff Limit of External Rays}\label{ssec2.4}
In this part, we briefly introduce Peterson-Zakeri's theory \cite{hausdorfflimit} on the Haudorff limit of external rays.

Let $\mathcal{K}$ be the metric space of all non-empty compact subset of $\CC$ equipped with the \emph{Huasdorff distance} which is defined by
\begin{align*}
    \d(K_1,K_2):= \inf \{ \rho>0 : K_1 \subset U_\rho(K_2) \mbox{ and } K_2 \subset U_\rho(K_1) \}.
\end{align*}
where $U_\rho(K):=\{ z \in \CC: d(z,K)<\rho \}$ is the $\rho$-neighborhood of $K$.
It is well-known that $\mathcal{K}$ is a compact metric space.

We say $\{ K_n \} \subset \mathcal{K}$ converges to $K \in \mathscr{K}$ in Hausdorff distance if $\d(K_n,K) \to 0$ as $n\to \infty$.
$K$ is called the \emph{Hausdorff limit} of $\{ K_n\}$.
One may verify that $K_n \to K$ in Hausdorff limit if and only if 
for every $\epsilon>0$, there exists $n_0$ such that for $n>n_0$, we have $K_n \subset U_\epsilon(K)$ and $K \subset U_\epsilon(K_n)$. 


Peterson and Zakeri study the Hausdorff limit of periodic external rays of polynomials.
Let $\mathcal{F}=\{ f_\a\}_{\a \in \AA}$ be a family of polynomials of degree $d\geq 2$.
Suppose that $\{ \a_n \} \subset  \mathcal{C}(\mathcal{F})$ converges to $ \a_0 \in \mathcal{C}(\mathcal{F})$, and
$\theta \in \RZ$ has period $p$ under $\tau_d$.
By taking subsequence, we assume that $\{ z_\theta(\a_n) \}$ converges to $z_\infty$ (which may not equal to $z_\theta(\a_0)$), and $R_{\a_n}(\theta)$ converges to $\RR(\theta)$ in the Hausdorff distance.
\begin{lem}
   The Hausdorff limit set $\RR(\theta)$ is a compact and connected invariant set of $f_{\a_0}^p$ containing $\overline{R_{\a_0}(\theta)}$ and $z_\infty$.
\end{lem}


According to \cite{douady1984etude}, we have the stability of the perturbation of external rays which land at repelling periodic orbits. 

\begin{lem}[continuity of external rays landing on repelling orbit]\label{stablerepelling}
    Suppose that $R_{\a_0}(\theta)$ lands at a repelling periodic or preperiodic point $z_0$. 
   \begin{enumerate}
       \item There exists a neighborhood $\mathcal{U}$ of $\a_{0}$ such that for any $\a \in \mathcal{U}\cap \mathcal{C}$, $s_{\a}(\theta)=1$ and $z_\theta(\a)$ is also eventually repelling.
    The map $\xi: (\mathcal{U} \cap \mathcal{C}) \times [1,+\infty) \to \C,~(\a,s)\mapsto \psi_\a(se^{2\pi i \theta})$ is continuous.
    \item If the forward orbit of $z_0$ does not contain any critical point, then there exist a a neighborhood $\mathcal{U}$ of $\a_{0}$, $s_{\a}(\theta)=1$ and $z_\theta(\a)$ is also eventually repelling.
    The map $\xi: \mathcal{U}  \times [1,+\infty) \to \C,~(\a,s)\mapsto \psi_\a(se^{2\pi i \theta})$ is continuous. 
    Fix $s \geq 1$, the map $\a \mapsto \psi_\a(se^{2\pi i \theta}) $ is holomorphic on $\mathcal{U}$.
    \end{enumerate}
    
\end{lem}

As a immediate corollary, if $\{\a_n \} \subset \mathcal{C}$, $\a_n \to \a_0$, and $R_{\a_0}(\theta)$ lands at a repelling preperiodic point $z_0$, then the Hausdorff limit of $R_{\a_n}(\theta)$ is precisely $R_{\a_0}(\theta)$. 
For the case that $z_0$ is parabolic, we shall introduce the following result \cite{hausdorfflimit} by Peterson-Zakeri which will play an important role in proving Proposition \ref{keylemma}.

\begin{thm}[Peterson--Zakeri \cite{hausdorfflimit}]\label{pz}
    The Hausdorff limit set $L(\theta)$ is path-connected.
    Let $z \in \RR(\theta)$.
    If $z$ is in the Julia set, then it is a fixed point of $f_\a^p$.
    If $z$ is in the Fatou set, then $z$ is in a parabolic basin.
\end{thm}

Let $\mathcal{F}=\{ f_\a\}_{\a \in \AA}$ be a family of polynomials of degree $d \geq 2$.
Let $\mathcal{C}_0 \subset \mathcal{C}$ is the set of $\a$ such that $f_\a$ has only one parabolic cycle, and its immediate basin contains only one critical points of $f_\a$.

\begin{lem}\label{ring}
    Let $f_\a$ be a polynomial of degree $d\geq 2$, and $U,V$ be a two periodic Fatou components in the same cycle with a common boundary point $z_0$, then $z_0$ is a periodic point of $f_\a$.
    If $f_\a^q(U)=V$, then $f_\a^q(z_0)=z_0$.
\end{lem}

\begin{proof}
    Let $p$ denote the period of the cycle containing $U,V$.
    Since $f_\a^p(U)=U$ and $f_\a^p(V)=V$, we see that $f_\a^p(z_0)$ is also a common boundary point of $U$ and $V$.
    It follows that $f_\a^p(z_0)=z_0$.
    Hence $z_0$ is periodic.

     \begin{figure}[h]
  \begin{center}
    \includegraphics[width=0.6\linewidth]{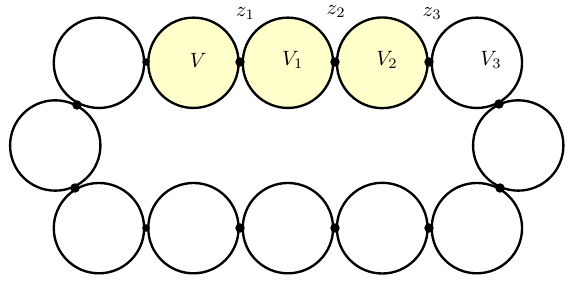}
    \caption{}
 \end{center}
\end{figure}

    Now assume that $f_\a^q(U)=V$, we show $f_\a^q(z_0)=z_0$ by contradiction.
    Suppose that $z_1=f_\a^q(z_0)\neq z_0$, then $V_1:=f_\a^q(V) \neq V$ and $\partial V \cap \partial V_1=\{ z_1 \}$.
    Applying $f_\a^q$ recursively, we obtain a sequence of Fatou components $V_n:=f_\a^{nq}(V)$ and $z_n:=f_\a^n(z)$ such that $\partial V_n \cap \partial V_{n+1}=\{ z_{n+1}\}$.
    Since $V$ is periodic, there exists $m\leq p$ such that $V_m=V$.
    It follows that we get finitely many Fatou components attaching one by one and forms a circle in $K_\a$.
    This is impossible.

\end{proof}

\begin{lem}\label{presevecyclic}
    Let $f_\a$ be a polynomial of degree $d\geq 2$, and $\theta \in \QZ$ has periodic $p$ under $\tau_d$, and $R_\a(\theta)$ lands at a parabolic point $z$.
    Let $B$ be a parabolic Fatou component of $z$.
    Then $B$ is a periodic Fatou component with exact period $p$.
\end{lem}

\begin{proof}
    Since $f_\a$ is locally orientational-preserving near every parabolic point, the period $q$ of $B$ divides $p$.
    On the other hand, $f_\a^q(z)$ is a parabolic periodic point on $\partial B$.
    Thus $B=f_\a^q(B)$ is both the parabolic basion of $z$ and $f_\a^q(z)$.
    It follows that $f_\a^q(z)=z$.
    Again by the orientational-preserving properties, we see that $f_\a^q(R_\a(\theta))=R_\a(\theta)$.
    Thus, $p$ also divides $p$, i.e. $p=q$.

\end{proof}

\begin{lem}\label{1fixedpoint}
    Let $B$ be a fixed attracting or parabolic component
    If $f_\a|_B$ has degree $2$, then there exists a unique fixed point of $f_\a$ on $\partial B$.
\end{lem}

\begin{proof}
    For the attracting case, by quasiconformal surgery, we may assume that $B$ is  super-attracting with a unique critical point $c$ of $f_\a$ in $B$.
    Thus, the B\"ottcher map $\phi$ near $c$ can be extended to a conformal map $\phi :B \to \D$.
    Since $\partial B$ is a Jordan curve, $\phi$ extends to $\partial B$ and forms a conjugacy connecting $f_\a|_{\partial B}$ and $z\mapsto z^2$ on $\partial \D$.
    Hence there exists a unique fixed point of $f_\a$ on $\partial B$.

    For the parabolic case, let $c$ be the unique critical point of $f_\a$ in $B$.
    According to \cite{Petersen01052010}, there exists a Riemann mapping $\phi: B \to \D$ such that $\phi(c)=0$ and
    $$
       g:=\phi \circ f_\a \circ \phi^{-1}(z) = \frac{3z^2+1}{z^2+3}.
    $$
    According to \cite[Remark 1]{Petersen01052010}, there exists a continuous map $\psi:\partial \D \to \partial \D$ such that $\psi\circ g\circ \psi^{-1}(z)=(\psi\circ \phi)\circ f\circ(\psi \circ \phi)^{-1}(z)=z^2$.
    Therefore, we also conclude that there exists a unique fixed point of $f_\a$ on $\partial B$.    
\end{proof}

\begin{prop}[countinuity of the landing points]\label{keylemma}
    Let $\{ \a_n \} \subset \mathcal{C}$, and $\a_n \to \a_0  \in \mathcal{C}_0$.
    Then for any $\theta \in \QZ$ which is periodic under $\tau_d$, we have $z_\theta(\a_n) \to z_\theta(\a_0)$.
\end{prop}

\begin{proof}[Proof of Proposition \ref{keylemma}]
    Let $\theta \in \QZ$, then the landing point $z_{\theta}(\a_0)$ of $R_{\a_0}(\theta)$ is either eventually repelling or parabolic. 
    Hence, there are two cases. 

    {\bf repelling case:} Assume that $z_{\theta}(\a_0)$ is a repelling periodic point.
    By Lemma \ref{stablerepelling}, there exists a neighborhood $\mathcal{U} \subset \AA$ such that $\xi: \mathcal{U} \times [1,+\infty) \to \C$ is continuous.
    It follows that $z_{\theta}(\a_n)=\xi(\a_n,1)$ and $z_\theta(\a_0)=z_\infty$ for $n$ large.
    Hence the conclusion of Proposition \ref{keylemma} holds.

    {\bf parabolic case:}
    To simplify the notations, denote $z_0:=z_{\theta}(\a_0)$, $z_n:=z_\theta(\a_n)$.
    Assume $z_0$ is a parabolic periodic point.
    Let $B_1,B_2,\dots,B_q$ be all parabolic components whose boundary contains $z_0$.
    Since $\a_0 \in \mathcal{C}_0$, these $B_k$'s belong to the same periodic cycle.
    We claim that for every $B_k$, $z_0$ is the parabolic point of $B_k$.
    Notice at least one of $B_k$'s must be a parabolic basin of $z_0$, without loss of generality, we assume that it is $B_1$.
    For each $B_k$, there exists $n_k$ such that $f_\a^{n_k}(B_1)=B_k$.
    It follows that $ f_\a^{n_k}(z_0) $ is the parabolic point of $B_k$.
    By Lemma \ref{ring}, we have $f_\a^{n_k}(z_0)=z_0$.
    $z_0$ is the parabolic point of $B_k$. 
    By Lemma \ref{presevecyclic}, each $B_k$ has period $p$ under $f_\a$.
    Since $\a_0 \in \mathcal{C}_0$, then the forward orbit of $B_k$ contains exactly one critical point.
    It follows that $f_\a^p|_{B_k}$ is a proper map of degree $2$.
    By Lemma \ref{1fixedpoint}, $z_0$ is the unique fixed point of $f_\a^p$ on $\partial B_k$.

    Now we show that $z_0=z_\infty$ by contradiction.
    Notice that $z_0$ and $z_\infty$ are both fixed point of $f_{\a_0}^p$.
    Let $L(\theta)$ be the Hausdorff limit of $R_{\a_n}(\theta)$.
    By Theorem \ref{pz}, we can find path $\gamma \subset L(\theta)$ joining $z_0$ and $z_\infty$.
    This path $\gamma$ cannot entirely contained in the Julia set since every point in $J_{\a_0}\cap L(\theta)$ is a fixed point of $f_{\a_0}^p$.
    It follows that $\gamma \cap B_k \neq \emptyset$ for some $k$. 
    We claim that $\gamma \subset B_k \cup \{ z_0 \}$.
    Otherwise, we can find another fixed point of $f_{\a_0}^p$ on $\gamma \cap \partial B$.
    This contradicts the fact that $z_0$ is the unique fixed point of $f_\a^p$ on $\partial B_k$.
    Thus, we have $z_n \to z_0$ as $n\to \infty$.
    
\end{proof}

As a immediate Corollary, we obtain a parabolic version of Lemma \ref{stablerepelling}.

\begin{fact}[external rays landing on parabolic points]\label{stableparabolic}
    Suppose that $\a_0 \in \mathcal{C}_0$, and $z_0$ is a parabolic fixed point which is the common boundary point of exactly $p$ parabolic components in one cycle.   
    For $\a$ sufficiently close to $\a_0$ such that all cycles repelling, the parabolic fixed point $z_0$ bifurcates into a repelling fixed point $z_0(\a)$ and a $p$ periodic repelling cycle.
    If $R_{\a_0}(\theta)$ lands at $z_0$, then $R_\a(\theta)$ either land at $z_0(\a)$ or a point in the $p$ cycle.

\end{fact}

\begin{proof}
    Since $z_0$ is a parabolic fixed point with $p$ parabolic components in one cycle attached at $z_0$.
    It follows that $z_0$ must satisfy that 
    $$f_{\a_0}'(z_0)=e^{2\pi i \frac{p}{q}},\qquad f_{\a_0}^p(z)-z_0=z+a_2z^{p+1}+o(z^{p+1}).$$
    For $\a$ sufficiently close to $\a_0$ such that all cycles repelling, since $z_0$ is a fixed point of $f_{\a_0}^p$ with multiplicity $p$, it follows that $z_0$ bifurcates into $p+1$ repelling fixed points of $f_\a^p$.
    Since $z_0$ is a fixed point, one of these $p+1$ is a fixed point $z_0(\a)$ for $f_\a$.
    Noticing that $f_{\a_0}'(z_0)=e^{2\pi i \frac{p}{q}}$, we see that the remaining $p$ points must belong to a $p$ periodic cycle.

\end{proof}

Let $\Q_p$ be the set of all periodic $\theta \in \Q$ which is periodic under $\tau_d$.
This set is invariant under $\tau_d$.
Thus, we can define the lamination $\L_{\Q_p}(\a)$ on $\Q_p$ for $\a \in \mathcal{C}$.
As a Corollary of Proposition \ref{keylemma}, we deduce the following result on the semi-continuity of laminations on $\Q_p$.

\begin{cor}[continuity of periodic laminations]\label{Qp}
    Suppose that $\{\a_n \} \subset \mathcal{C}$ converges to $\a_0 \in \mathcal{C}_0$, then we have $\limsup_n \L_{\Q_p}(\a_n) \subset \L_{\Q_p}(\a_0)$.
\end{cor}

\begin{proof}
    Suppose that $(\theta,\theta') \in \limsup_n \L_{\Q_p}(\a_n) $, then there exists a subsequence $\a_{n_k}$, such that $(\theta,\theta') \in \L_{\Q_p}(\a_{n_k})$.
    Hence $R_{\a_k}(\theta)$ and $R_{\a_k}(\theta')$ land at a common point, i.e. $z_\theta(\a_{n_k})=z_{\theta'}(\a_{n_k})$ for every $k$.
    By Proposition \ref{keylemma}, since $\a_{n_k} \to \a_0$ as $k \to \infty$, we have $z_\theta(\a_0)=z_{\theta'}(\a_0)$.
    Hence $(\theta,\theta') \in \L_{\Q_p}(\a_0)$.

\end{proof}

\section{Cubic Polynomials and Generalized Internal Rays}\label{sec3}

\subsection{Mapping Scheme and B\"ottcher Maps}

Recall that we parameterize cubic polynomials by the following Branner-Hubbard normal form
$$
    f_\a(z)=z^{3}-3c^{2}z+(2c^3+v),\qquad \a=(c,v) \in \C^2.
$$
Then $\cc(\a)=\{ \pm c(\a) \}$.
Recall that the bounded hyperbolic components of cubic polynomials are divided into $4$ types, namely \A, \B, \CCC~ and (D).
Let $\mathcal{H}$ be such a hyperbolic component.
Let $B_\a(c)$ denote the attracting Fatou component containing $c(\a)$, and $p$ be the period of $B_\a(c)$.
Let $B_\a:=\bigcup_n f_\a^n (B_\a(c))$ be its \emph{immediate attracting basin}, and $\tilde{B}_\a :=\bigcup_n f_\a^{-n}(B_\a(c))$ be its whole \emph{attracting basin}.
Up to exchange $c(\a)$ and $-c(\a)$, types of bounded hyperbolic components and their tame part of the boundary can be summarized as follows.
\begin{align*}
    \A : -c(\a) \in B_\a(c),\quad \B: -c(\a) \in B_\a \setminus B_\a(c), \quad \CCC: -c(\a) \in \tilde{B}_\a \setminus B_\a, 
\end{align*}
and (D): $-c(\a) \notin \tilde{B}_\a $. 
The tame boundary is defined to be
$$\p \mathcal{H}=\{ \a \in \partial \mathcal{H} : c(\a) \mbox{ tends to an attracting cylce} \}.$$
Therefore, for $\a \in \mathcal{H} \cup \p \mathcal{H}$, $B_\a(c)$, $B_\a$ and $\tilde{B}_\a$ are always defined.

First, we may assume that $B_\a$ is a super-attracting basin, i.e. $c(\a)$ is a super-attracting periodic point of $f_\a$ with period $p$.
In fact, using the quasiconformal surgeries (Proposition \ref{multiplier}), we will see that it suffices to consider this case.
Hence we may restrict the discussion in Milnor's periodic $p$ super-attracting slices $\S$ where
$$\mathcal{S}_{p}:=\left \{ \a \in \mathbb{C}^2 : c(\a) \text{ is a $p$-periodic point of $f_\a$} \right \}.$$

Fix $\a \in \S$,
denote $O=O(\a):=\{ c(\a), c_1(\a),\dots,c_{p-1}(\a) \}$ to be the super-attracting $p$ periodic orbit of $c(\a)$ where $c_k(\a):=f_\a^k(c(\a))$ for $0\leq k \leq p-1$, and
$\tilde{O}=\tilde{O}(\a):=\bigcup_n O_n$ be the grand orbit of $c(\a)$.

In the following, we discuss the corresponding B\"ottcher maps of $x \in \tilde{O}$.
In our discussion, we always assume that $\a \in \S$ is \emph{generic}, that is, $-c(\a) \notin \tilde{O}(\a)$.
We can define the corresponding B\"ottcher map $\phi_\a^c$ near $c(\a)$ which conjugates $f_\a$ and $z\mapsto z^2$ for $c(\a)\neq 0$.
Define the \emph{exponential Green function} $G_\a^c : \C \to [0,1]$ with respect to $c(\a)$ by
\begin{align*}
    G_\a^c (z) = \lim_{n\to\infty} \sqrt[n]{|f^{np}_\a(z)-c(\a)|}.
\end{align*}
For $0<s\leq 1$, denote $B_\a^c(s):=(G_\a^c)^{-1}(\{ z : |z|<s \})$.
Let 
$$
    s_{\a}^{c}:=\min \{ G_\a^c(z): z\in (B_\a(c) \setminus \{ c(\a_0)\}) \cap \cc(f_\a^p) \}.
$$
Thus, the B\"ottcher map $\phi_\a^c$ is defined at least on $B_\a^c(s_\a^c)$ to the disk with radius $s_\a^c$ centered at the origin.
Its inverse is denoted by $\psi_\a^c$.

For generic $\a$, and each $v \in \tilde{O}\setminus \{ c(\a)\}$, the map $f_\a^n$ is conformal on a neighborhood of $v$ to a neighborhood of $c(\a)$ where $\ell_v\geq 0$ is the smallest integer such that $f_\a^{\ell_v}(v)=c(\a)$.
Thus, we can pull-back the B\"ottcher map $\phi_\a^c$ and the Green function $G_\a^c$ to get the B\"ottcher map $\phi_\a^v:=\phi_\a^c \circ f_\a^{\ell_v}$ and the Green function $G_\a^v:=G_\a^c \circ f_\a^{\ell_v}$ of $v$.
Its inverse is denoted by $\psi_\a^v$, it is defined on a small disk $D(0,\rho_\a)$ centered at the origin.
One may observe that $\rho_\a$ only depends on $\a$ not on $v$.

We shall carefully discuss the extension of $\psi_\a^v$ for $v \in \tilde{O}$ since it plays an important role in our study of laminations.
To be more systematic, we introduce the \emph{mapping scheme} $(\O,\sigma,\delta)$ and model space $\M$ over $\tilde{O}$.

\begin{defi}[mapping scheme and model space]
    For generic $\a \in \S$, the \emph{mapping scheme} of $\a$ is a triplet $(\O,\sigma,\delta)$ where 
    \begin{itemize}
        \item $\O=\O(\a)$ is the grand orbit of $c(\a)$;
        \item $\sigma : \tilde{O} \to \tilde{O}$ is defined by $\sigma(v)=f_\a(v)$;
        \item $\delta : \tilde{O} \to \mathbb{N}$ is defined by $\delta(c(\a))=2$ and $\delta(v)=1$ for $v\neq c(\a)$.
    \end{itemize}
The \emph{model space} over $\tilde{O}$ is defined to be
$\M=\M(\tilde{O}):=\tilde{O} \times \D$.
If we identify $\partial \D$ and $\RZ$, then $\partial \M$ is identified with $\O \times \RZ$.
The \emph{model map}s $g$ and $\g$ are defined by
$$g: \M \to \M,\quad (v,z)\mapsto (\sigma(v),z^{\delta(v)}).$$
and
    $\tilde{g}: \partial \M \to \partial \M$ by $ (v,t)\mapsto (\sigma(v),\tau_{\delta(v)}(t)).$
\end{defi}

Under the notation above, we conclude that there exists a map $\Psi_\a: \O \times D(0,\rho_\a) \to \tilde{B}_\a$ given by $(v,w)\mapsto \psi_\a^v(w)$ such that
\begin{equation}\label{diaeq1}
     f_\a \circ \Psi_\a  = \Psi_\a \circ g,\qquad \forall (v,w) \in \O \times D(0,\rho_\a).
\end{equation}

Now we try to extend $\Psi_\a$ to maximal domains such that \eqref{diaeq1} still holds.
There is a simple case, if $-c(\a) \notin \tilde{B}_\a$, then for each $v \in \tilde{O}$, $\phi_\a^v$ can be extended to a conformal map from $B_\a(v) $ to $ \D$.
Thus, their inverse $\psi_\a^v$ are all defined and conformal on $\D$.
Thus, the map $\Psi_\a$ can be extended to 
$\Psi_\a : \M \to \tilde{B}_\a$ such that \eqref{diaeq1} still holds, i.e. the following diagram commutes.
\begin{equation}\label{diapsi}
    \xymatrix{\ar @{} [dr]
   \M \ar[d]^{g} \ar[r]^{\Psi_{\a}} & \tilde{B}_\a \ar[d]^{f_{\a}} \\
   \M \ar[r]^{\Psi_\a}        & \tilde{B}_\a     }
\end{equation}
The extended map $\Psi_\a$ is \emph{fiberwise conformal}, i.e. fix any $v \in \tilde{O}$, the restriction map $\Psi_\a: \{ v\} \times \D \to B_\a(v)$ is a conformal map.

The opposite case $-c(\a) \in \tilde{B}_\a$ is more complicated, i,e. the parameter $\a$ belongs to a bounded hyperbolic component $\H$ of type \A~or \B, or \CCC.
Remember we excluded the case that $-c(\a) \in O(\a)$, hence $\a$ cannot be the center of the hyperbolic component $\H$. 
In the following, such maps are called \emph{ABC-hyperbolic}.

Now suppose that $\a$ is ABC-hyperbolic, and $(v,t) \in \O \times \RZ$.
Let $S(v,t)$ be the set of $0<r<1$ such that there exists a Jordan arc $\gamma_r^{v,t}: [0,r] \to B_\a(v) $ starting at $v$ such that the following properties (P1) and (P2) hold.
\begin{enumerate}
    \item[(P1)] {\bf smoothness:} $\gamma_r^{v,t}(s)$ is not an iterative preimage of $-c(\a)$ for $0\leq s < r$.
    \item[(P2)] {\bf dynamical compatibility:} Given any $0<s<r$, for any $n$ large with $g^n(v,se^{2\pi i t})\in O \times D(0,s_\a^c)$, we have 
    $$
        f_\a^n(\gamma_r^{v,t}(s))=\Psi_\a(g^n(v,se^{2\pi i t})).
    $$
    
\end{enumerate}
The property (P1) implies that $\gamma_r^{v,t}$ is a smooth Jordan arc.
Property (P2) states that $\gamma_r^{v,t}$ is canonical and competitive with the dynamics.
By (P2), we see that $\gamma_r^{v,t}$ is indeed parameterized by the exponential Green function $G_\a^v$.
According to \eqref{diaeq1}, to verify (P2), it suffices to show that there exists $n$ with $g^n(v,se^{2\pi i t})\in O \times \D(0,s_\a^c)$ such that $f_\a^n(\gamma_r^{v,t}(s))=\Psi_\a(g^n(v,se^{2\pi i t}))$ holds. 
By the uniqueness of lifting, fix $r \in S(v,t) $, the curve $\gamma:[0,r] \to B_\a(v)$ satisfying (P2) is unique and hence must be $\gamma_r^{v,t}$.
As a consequence, fix $r<r' \in S(v,t)$, we have  $\gamma_r^{v,t}=\gamma_{r'}^{v,t}$ on $[0,r]$.
In the following when $(v,t)$ is clear, $\gamma_r^{v,t}$ is sometimes briefly written as $\gamma_r$ or simply $\gamma$.


\begin{lem}[radial extension]\label{lem31}
  For generic $\a \in \S$, and $(v,t) \in \O \times \RZ$, define $$s_\a^v(t):=\sup S(v,t).$$ 
  If $s_\a^v(t)<1$, then $s_\a^v(t) \in S(v,t)$, and $\gamma_{s_\a^v(t)}^{v,t}(s_\a^v(t))$ is an iterative preimage of $-c(\a)$.
\end{lem}

\begin{proof}

    Assume that $s_\a^v(t)<1$.
    By the definition of $s_\a^v(t)$ and the fact that $s_\a^v(t)\geq \rho_\a>0$, 
    we can find a sequence $\{ r_k\} \subset (\rho_\a,s_\a^v(t))$ monotonely converging to $s_\a^v(t)$ and a sequence of curves $\gamma_{r_k}:[0,r_k] \to B_\a(v)$ satisfying (P1) and (P2).
    By the uniqueness of each $\gamma_{r_k}$, we can glue them together to form a Jordan arc $\gamma: [0,s_\a^v(t)) \to B_\a(v)$ which still satisfies (P1), and (P2).    

    Now fix $n$ such that $g^n(v,se^{2\pi i t})\in \{c \} \times D(0,s_\a^c)$.
    By (P2), any limit point of $\gamma(s)$ as $s \to s_\a^v(t)$ is mapped to $\Psi_\a(g^n(v, s_\a^v(t) e^{2\pi it}))$ by $f_\a^n$.
    It follows that the limit $\gamma(s)$ exists as $s \to s_\a^v(t)$.
    The limit point is denoted by $\omega$.
    Extend $\gamma$ to $\gamma:[0,s_\a^v(t)] \to B_\a(v)$ by defining $\gamma(s_\a^v(t))=\omega$.
    The extended $\gamma$ still satisfies (P1) and (P2).
    Now we show that $\omega$ must be a critical point of $f_\a^n$.
    We prove it by contradiction.
    If not, then there exists a neighborhood $V$ of $z_0$ such that $f_\a^n$ maps $V$ univalently into $B_\a^c(s_\a^c)$.    
    Pick $s'>s_\a^v(t)$ such that $\Psi_\a (g^n(v,[s_\a^v(t),s']e^{2\pi i t}))\subset f_\a^n(V)$.
    Hence we can extend the the Jordan arc $\gamma$ on $[0,s']$ by $\gamma^{v,t}(s):=(f_\a^n|_V)^{-1}\circ \Psi_\a (g^n(v,se^{2\pi i t}))$ for $s \in [s_\a^v(t),s']$.
    Clearly, the extended arc still satisfies (P1) and (P2).
    Hence $s' \in S(v,t)$.
    This contradicts the fact that $s_\a^v(t)=\sup S(v,t) $.
    Therefore, we finish the proof of the lemma.
    
\end{proof}


Notice the two notations $\psi_\a^v$ and $\Psi_\a(v,\cdot)$ are the same.
In the following these two notations will be used both.

For generic $\a \in \S$, define the \emph{reduced model space} $\M_\a \subset \M$ by
$$
    \M_\a:=\bigsqcup_{v \in \tilde{O}} (\{ v\} \times \{ re^{2\pi i t} : 0\leq r < s_\a^v(t) \}).
$$
(If $-c(\a) \notin \tilde{B}_\a$, then $\M_\a:=\M$.)
Define $\M_\a':=\{ (v,t) \in \O \times \RZ : s_\a^v(t)<1 \}$.
From Lemma \ref{lem31}, we deduce the following.

\begin{prop}[reduced model space]\label{reduced}
The reduced model space $\M_\a$ is invariant under $g$.
The map $\Psi_\a$ can be extended to $ \Psi_\a : \M_\a \to \tilde{B}_\a $ by
$$
    \Psi_\a (v,w)=\psi_\a^v(w) := \gamma_{s_\a^v(t)}^{v,t}(r),\qquad  w=se^{2\pi i t},~0\leq s < s_\a^v(t).$$
is a continuous injection such that the following diagram commutes.
\begin{equation*}
    \xymatrix{\ar @{} [dr]
   \M_\a \ar[d]^{g} \ar[r]^{\Psi_{\a}} & \tilde{B}_\a \ar[d]^{f_{\a}} \\
   \M_\a \ar[r]^{\Psi_\a}        & \tilde{B}_\a .    }
\end{equation*}
\end{prop}

\begin{proof}
    According to Lemma \ref{lem31} and (P2), it suffices to show that $g(\M_\a)\subset \M_\a$.
    For any $(v,se^{2\pi it})\in \M_\a$, denote $g(v, se^{2\pi i t}) = (v',s'e^{2\pi i t'})$.
    Since $s \leq s_\a^v(t)$, the Jordan arc $f_\a(\Psi_\a(v,[0,s]e^{2\pi i t}))$ is a curve starting at $v'=\sigma(v)$ and satisfying (P1) and (P2). 
    Parameterizing by the exponential Green function, we see that $\gamma^{v',t'}$ is defined on $[0, s']$.
    Hence, we have $s'\leq s_\a^{v'}(t')'$ by the definition of $s_\a^{v'}(t')$.  
    Thus, we conclude that $(v',s'e^{2\pi i t'}) \in \M_\a$.
    
\end{proof}

The smooth arc $\Psi_\a(v,[0,s_\a^v(t))e^{2\pi i t})=\gamma^{v,t}([0,s_\a^v(t)))$ is called the \emph{smooth internal arc} for $(v,t) \in \O\times \RZ$.
By Proposition \ref{reduced}, we see that $f_\a$ maps a smooth internal arc into a smooth internal arc, i.e. $f_\a(\Psi_\a(v,[0,s_\a^v(t))e^{2\pi i t}))\subset \Psi_\a(v,[0,s_\a^{v'}(t')e^{2\pi i t'}))$ for every $(v,t)\in \O \times \RZ$ with $\g(v,t)=(v't')$.

If $s_\a^v(t)=1$, the smooth internal arc 
$$R_\a^{v}(t):=\Psi_\a(v, [0,1)e^{2\pi i t})$$ is also called the \emph{smooth internal ray} for $(v,t)$.

If $s_\a(t)<1$, then $\Psi_\a(v,s_\a^v(t)e^{2\pi i t})$ can be defined to be its limit.
Moreover, $\omega=\Psi_\a(v,s_\a^v(t)e^{2\pi i t})$ is an iterative preimage of $-c(\a)$.
We say the {smooth internal arc} $\Psi_\a(v,[0,s_\a^v(t))e^{2\pi i t})$ \emph{terminates} at $\omega$.
Conversely, we have the following.

\begin{lem}[critical points and internal arcs]\label{lem32}
    Let $f_\a$ be an ABC-hyperbolic map.
    There exist $(x,t_\a)$ and $(y,t'_\a)$ in $\M'_\a$ such that $\g(x,t_\a)=\g(y,t'_\a)$ and 
\begin{equation}\label{critical}
    -c(\a)=\psi_\a^x(r_\a e^{2\pi i t_\a})=\psi_\a^y(r'_\a e^{2\pi i t'_\a}),
\end{equation}
where $r_\a=s_\a^x(t_\a)$, $r_\a'=s_\a^y(t_\a')$.
For any $n$-th iterative preimage of $\omega$ of $-c(\a)$, there exist at least one $(v,t) \in \M_\a'$ such that $\g(v,t)=(x,t_\a)$ or $\g(v,t)=(y,t_\a')$ and
\begin{equation}\label{precritical}
    \omega=\psi_\a^v(s_\a^v(t)e^{2\pi i t}).
\end{equation}

Conversely, we have $(v,t) \in \M_\a'$ if and only if the forward orbit of $(v,t)$ under $\g$ contains $(x,t_\a)$ or $(y,t_\a')$.
\end{lem}

\begin{proof}
    
  First, we show that there are two internal arc terminating at the critical point $-c(\a)$.
  Let $m \geq 1$ be the smallest positive integer such that $f_\a^m(-c(\a)) \in B_\a(c)$.
  One may verify that $G_\a^c ( f_\a^m(-c(\a)) ) <s_\a^c$.
  Hence, there exist $r<s_\a^c$ and $t \in \RZ$ such that $\psi_\a^c(re^{2\pi i t})=f_\a^m(-c(\a))$.
  It follows that the smooth internal arc $\psi_\a^c([0,s_\a^c(t))e^{2\pi i t})$ pass through $f_\a^m(-c(\a))$.
  Let $\Gamma$ be the connected component of $f_\a^{-m} ( \psi_\a^c([0,r]e^{2\pi i t}) )$ containing $-c(\a)$.
  Then $\Gamma \setminus \{-c(\a) \}$ contains two components $\Gamma_1,\Gamma_2$, and $f_\a^m : \Gamma_i \to \psi_\a^c([0,r]e^{2\pi i t}) $ is a homeomorphism for $i=1,2$.
  Parameterized by the exponential Green function, $\Gamma_i$ satisfies (P1) and (P2).
  Therefore, we see that $\Gamma_1$ and $\Gamma_2$ must be two smooth internal arc with respect to some $(x,t_\a),(y,t_\a')\in \O \times \RZ$ which both terminate at $-c(\a)$.
  It follows that \eqref{critical} holds.
  By \eqref{diaeq1}, since $\Gamma_1$ and $\Gamma_2$ have the same image, we have $\g(x,t_\a)=\g(y,t'_\a)$.
  
  We claim that there are no more smooth internal arc terminating at $-c(\a)$.
  If there is another $\Gamma_3$ terminating at $-c(\a)$, then $f_\a^m(\Gamma_3)$ must be a part of the internal arc passing through $f_\a^m(-c(\a))$.
  It follows that $f_\a^m(\Gamma_3)= \psi_\a^c([0,r]e^{2\pi i t})$.
  This contradicts the fact that $-c(\a)$ is a simple critical point.

  We claim that at least one of $\Gamma_1$ and $\Gamma_2$ does not contain the forward orbit of $-c(\a)$.
  Suppose that $\Gamma_1=\psi_\a^x([0,r_\a]e^{2\pi i t_\a})$ contains some iteration of $-c(\a)$, then $(x,t_\a)$ must be periodic under $\g$.
  Since $\g(y,t_\a')=g(x,t_\a)$, it follows that the forward orbit of $(x,t_\a)$ does not contain $(y,t_\a)$.
  Hence $\Gamma_2=\psi_\a^y([0,r_\a']e^{2\pi i t_\a'})$ does not contain the forward orbit of $-c(\a)$.   
  Let $\omega$ be an $n$-th preimage of $-c(\a)$.
  Assume that $\Gamma_2$ does not contain the forward orbit of $-c(\a)$.    
  Let $\Gamma'$ be the connected component of $f_\a^{-n}(\Gamma)$ which contains $\omega$.
  Since $\Gamma_2$ does not contain the forward orbit of $-c(\a)$, $f_\a^n: \Gamma' \to  \Gamma_2$ is a homeomorphism.
  Similar to the argument above, $\Gamma'$ must be a smooth internal arcs which terminates at $\omega$ with respect to some $(v,t) \in \O \times \RZ$ and $\g^n(v,t)=(y,t_\a')$.
  Hence \eqref{precritical} holds.

  For any $(v,t) \in \O\times \RZ$, if $s_\a^v(t)<1$, by Lemma \ref{lem31}, $\psi_\a^v(s_\a^v(t)e^{2\pi i t})$ is an iterative preimage of $-c(\a)$.
  Suppose that $f_\a^n( \psi_\a^v(s_\a^v(t)e^{2\pi i t}) )=-c(\a)$.
  Then the smooth internal arc of $\g^n(v,t)$ terminates at $-c(\a)$.
  It follows that $ \g^n(v,t) $ coincides with $(x,t_\a)$ or $(y,t_\a')$.
  Conversely, if $ \{ \g^n(v,t) : n \geq 0 \} $ contains $(x,t_\a)$ or $(y,t_\a')$.
  If $s_\a^v(t)=1$, then by Proposition \ref{reduced}, for any $n$, we have $s_\a^{v_n}(t_n)=1$ where $(v_n,t_n):=\g^n(v,t)$.
  Thus either $s_\a^x(t_\a)=1$ or $s_\a^y(t_\a')=1$.
  This is a contradiction.

\end{proof}

\begin{rmk}
    In fact, based on the type of hyperbolic component containing $\a$ , we can find $x,r_\a$ and $y,r'_\a$ explicitly.
There are $3$ cases.

\begin{itemize}
    \item[(A)] If $\a$ belongs to a hyperbolic component of type \A, i.e. $-c(\a) \in B_\a(c)$.
In this case, one of $x,y$ must be $c(\a)$.
We choose $x$ to be $c(\a)$, then $r_\a=s_\a^c$.
Then $y$ must be outside of the $p$ periodic orbit $O(\a)$, and $r_\a'=(s_\a^c)^2$.

\item[(B)] If $\a$ belongs to a hyperbolic component of type \B, i.e. $-c(\a) \in B_\a(c_k)$ for some $1\leq k \leq p-1$, then one of $x,y$ must be $c_k(\a)$. $x=c_k(\a)$.
Then $y$ must be outside of $O(\a)$, and $r_\a=r_\a'=(s_\a^c)^2$.

\item[(C)] Suppose that $-c(\a) \in \tilde{B}_\a \setminus B_\a$, i.e. $\a$ is in a hyperbolic component of type \CCC.
In this case $s_\a^c=1$.
Then $x,y$ are both outside of $O(\a)$.

\end{itemize}

\end{rmk}

There is another useful way to classify ABC-hyperbolic maps, which will play an important role in considering the boundary maps of hyperbolic components.

\begin{defi}[periodic \& non-periodic]
    Let $\a \in \S$ be an ABC-hyperbolic map, and $x,y \in \O$, $t_\a,t'_\a$ be given in Lemma \ref{lem32}.
    We say $f_\a$ has \emph{periodic combinatorics} if $(x,t_\a)$ is periodic under $\g$.
\end{defi}

In fact, $f_\a$ has periodic combinatorics if and only if $\a$ is in a hyperbolic component of type \A~or \B, and $t_\a$ is periodic under the doubling map.
By Lemma \ref{lem32}, an essential difference between maps with periodic combinatorics and non-periodic combinatorics is that if $\a$ has non-periodic combinatorics, then $\g(x,t_\a)\notin \M'_\a$.

\begin{rmk}\label{single}
    In the periodic case, since $\g(x,t_\a)=\g(y,t'_\a)$, by Lemma \ref{lem32}, every $(v,t) \in \M'_\a$ is eventually mapped to $(x,t_\a)$.
\end{rmk}

\subsection{Left and Right Extensions}

In this part, we continue the extension of $\Psi_\a$ to $\M$ for every ABC-hyperbolic map $\a$.
That is, for each $(v,t) \in \O \times \RZ$, continue to extend $\Psi_\a(v,\cdot)=\psi_\a^v$ along the radial line $[0,1)e^{2\pi i t}$ such that the diagram \eqref{diapsi} still holds.
Notice that the original $\psi_\a^v$ is ended at $s_\a^v(t)e^{2\pi i t}$ where $\psi_\a^v(s_\a^v(t)e^{2\pi i t})$ is an iterative preimage of $-c(\a)$.
The local behavior of $f_\a$ allows us to continue the extension of $\Psi_\a(v,\cdot)$ in two opposite directions (left and right).
This extension is the key to find the laminations of boundary maps of hyperbolic components.

A simple planar curve $\gamma: I \to \C$ is called \emph{semi-differentiable} if $\gamma'(t)$ exists and non-vanishing for every $t$ in the interior of $I$ except for a discrete set $T \subset I$, and for any $t\in T $, $\gamma'_-(t)$ and $\gamma'_+(t)$ exist and are non-vanishing.
$T$ is called the \emph{turning set} of $\gamma$.
In particular, if $\gamma$ is differentiable, then it is semi-differentiable.

A turning point $t \in T$ is called a \emph{left-turning point} or a $L$-turning point of $\gamma$ if  $\arg \gamma_+'(t)-\arg \gamma_-'(t)=\frac{\pi}{2}$.
It is called a \emph{right-turning point} or a $R$-turning point of $\gamma$ if  $\arg \gamma_-'(t)-\arg \gamma_+'(t)=\frac{\pi}{2}$.
A semi-differentiable curve $\gamma: I \to \C$ is \emph{left-turning} if every turning point is a left turning point.
Similarly, $\gamma$ is called \emph{right-turning} if every turning point is a right turning point.

For generic $\a \in \S$, and $(v,t) \in \O\times \RZ$, let $S_L(v,t)$ be the set of $0<r<1$ such that there exists a left-turning curve $\gamma: [0,r) \to B_\a(v)$ such that dynamical compatibility property (P2) holds.
Similarly, we define $S_R(v,t)$ be the corresponding set concerning right-turning curves.

\begin{lem}\label{lem34}
    For generic $\a \in \S$, and $(v,t) \in \O\times \RZ$, we have 
    $$\sup S_L(v,t) =  \sup S_R(v,t) =1 .$$
\end{lem}

\begin{proof}
    We only prove that $ \sup S_L(v,t) =1 $.
    The proof is very similar to the proof of Lemma \ref{lem31}.
    We will skip the similar parts quickly and focus on the difference.
    Denote $r_0: =\sup  S_L(v,t) $.
    Since $ S(v,t) \subset  S_L(v,t) $, we have $r_0 \geq s_\a^v(t) $.
    In the following we prove $r_0=1$ by contradiction.

    By the same argument in the proof of Lemma \ref{lem31}, we also obtain that $r_0 \in S_L(v,t) $ and a left-turning curve $\gamma: (0,r_0) \to B_\a(v)$.
    There ate two cases.

    Suppose that $\omega:= \gamma(r_0) $ is not an iterative preimage of $-c(\a)$.
    By the same argument of the proof of Lemma \ref{lem31}, we can find $r’>r_0$ and an extension $\tilde\gamma: [0,r'] \to B_\a(v)$ of $\gamma$ such that (P2) holds.
    Moreover, since $\omega$ is not an iterative preimage of $-c(\a)$, $\tilde\gamma|_{[r_0,r']}$ is smooth.
    Hence $r' \in S_L(v,t) $.
    It contradicts the fact that $r'>r_0= \sup S_L(v,t)$.

    Now assume the opposite, i.e. $\omega:= \gamma(r_0) $ is an $m$-th iterative preimage of $-c(\a)$.
    Pick $n\geq m$ large with $g(v,r_0e^{2\pi i t}):=(c,r_ne^{2\pi i t_n}) \in \{ c\} \times D(0,s_\a^c)$, and a small neighborhood $U$ of $f_\a^n(\omega)$ such that $\omega$ is the unique critical point of $f_\a^n$ in $V$ where $V$ is the connected component of $f_\a^{-n}(U)$ containing $\omega$. 
    Pick $r^\pm$ such that $r_0 \in (r^-,r^+)$ and $\Gamma_n:=\Psi_\a(g^n(v,(r^-,r^+)e^{2\pi i t})) \subset  U$.
    Let $\Gamma$ be the connected of $f_\a^{-n}(\Gamma_n) $ containing $\omega$.
    Then $\Gamma\setminus \{ \omega\}$ contains two components $\Gamma_L$ and $\Gamma_R$.
    Let $\Gamma_L$ be the component be the one such that $\Gamma_L$, $\gamma([r^-,r_0])$ and $\Gamma_R$ are in positive cyclic orders.
    Extend $\gamma$ to a curve $\gamma_L: [0,r^+] \to B_\a(v)$ such that 
    $$
        \gamma_L(r):= (f_n^\a|_{\Gamma_L})^{-1}( \Psi_\a(g^n(v,re^{2\pi i t})) ),\qquad \forall r \in [r_0,r^+].
    $$ 
    Since $\Gamma_n$ is smooth and $\omega$ is a simple critical point of $f_\a^n$, by the choice of $\Gamma_L$, it is not hard to verify that $\gamma_L$ is a left-turning curve and $r_0$ is the unique new-added turning point.
    It follows that $r^+ \in S_L(v,t)$.
    This contradicts the fact that $r^+>r_0= \sup S_L(v,t)$.

    In both cases, we get a contradiction.
    Hence, we conclude that $ \sup S_L(v,t) =1 $.
    

\end{proof}

For generic $\a \in \S$,
by Lemma \ref{lem34}, for $(v,t) \in \O \times \RZ$, since $\sup S_L(v,t) =1 $, there exists a monotone increasing sequence $\{r_n\}$ converging to $1$ such that $r_n \in S_L(v,t) $.
It follows that there exist a sequence $\gamma_{L,n}: [0,r_n] \to B_\a(v)$ such that (P2) holds.
By (P2), we can glue them together to form a left-turning curve $\gamma_L: [0,1) \to B_\a(v)$ such that (P2) holds.
Hence, we can define a map $\Psi_\a^L: \M \to \tilde{B}_\a$.
It is called the \emph{left extension} of $\Psi_\a$.
For $(v,w)\in \M$ with $w=re^{2\pi i t}$, define $\Psi_\a^L(v,w):=\gamma_L(r)$.
Similarly, we can also find a right-turning curve $\gamma_R: [0,1) \to B_\a(v) $ and the \emph{right extension} $\Psi_\a^R: \M \to \tilde{B}_\a$.
Combine our discussion above, we conclude the following.

\begin{prop}[left and right extensions of $\Psi_\a$]\label{nonsmoothray}
    The map $\Psi_\a$ has two extensions: the right extension $\Psi_\a^R : \M \to \tilde{B}_\a$, and the left extension $\Psi_\a^L: \M \to \tilde{B}_\a$ of $\Psi_\a$.
    For $\iota \in \{ L,R \}$, the following diagram
    \begin{equation*}
    \xymatrix{\ar @{} [dr]
    \M \ar[d]^{g} \ar[r]^{\Psi^\iota_{\a}} & \tilde{B}_\a \ar[d]^{f_{\a}} \\
    \M \ar[r]^{\Psi^\iota_\a}        & \tilde{B}_\a .    }
\end{equation*}
Moreover, $\Psi_\a^\iota (v,\cdot) $ is homeomorphic on $[0,1) e^{2\pi i t}$ for each fixed $t \in \RZ$.
\end{prop}

For generic $\a \in \S$, and $(v,t)\in \O \times \RZ$ and $\iota \in \{ L,R \}$, we define the \emph{$\iota$-internal ray} $\RRR_{\a}^\iota(v,t)$ and the \emph{left-right internal ray} $\RRR_\a^{L,R}(v,t)$ of $(v,t)$ to be 
$$
    \RRR_{\a}^\iota(v,t):=\Psi_\a^\iota(v,[0,1)e^{2\pi i t}),\quad  \RRR_\a^{L,R}(v,t):= \RRR_{\a}^L(v,t) \cup \RRR_{\a}^R(v,t).
$$
and 
By (P2), it is not hard to verify that $f_\a$ maps $\iota$-internal rays to $\iota$-internal rays, i.e. for nay $(v,t) \in \O \times \RZ$, and $\iota \in \{ L ,R \}$, we have
$
    f_\a(\RRR^\iota_\a(v,t)) = \RRR_\a^\iota(\g(v,t)).
$
If $s_\a^v(t)=1$, then $ \RRR_{\a}^L(v,t),\RRR_{\a}^R(v,t) $ and $ \RRR_\a^{L,R}(v,t) $ all coincide with the smooth internal ray $R_\a^v(t)$.

 \begin{figure}[h]
  \begin{center}
   \vspace{2mm}
   \begin{minipage}{.48\linewidth}
    \includegraphics[width=\linewidth]{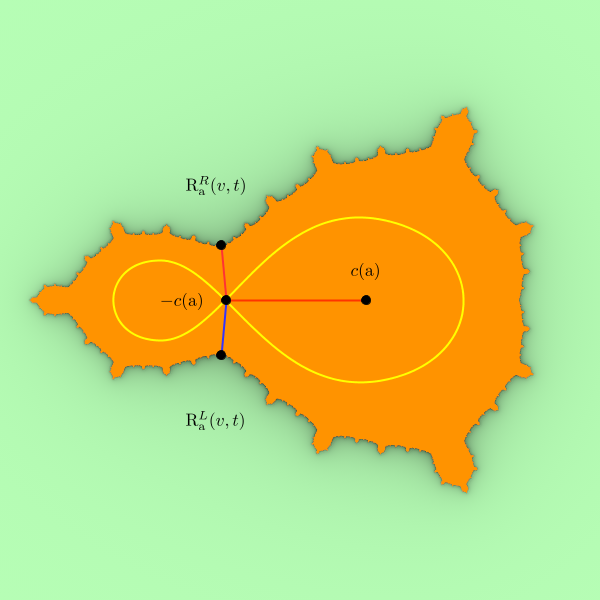}
    \caption{non-periodic}
  \end{minipage}
  \hspace{1mm}
  \begin{minipage}{.48\linewidth}
    \includegraphics[width=\linewidth]{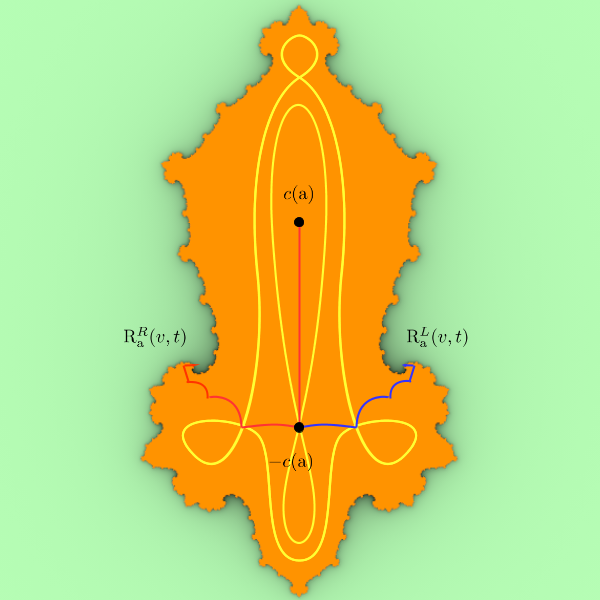}
    \caption{periodic}
  \end{minipage}
 \end{center}
\end{figure}

Let $f_\a$ be an ABC-hyperbolic map.
If $f_\a$ has periodic combinatorics, then every non-smooth left or right internal ray has infinitely many turning points.
If $f_\a$ has non-periodic combinatorics, then every non-smooth left or right internal ray has exactly one turning point.
Another property we observes for non-periodic case is the following.

\begin{rmk}[internal ray with the same end]\label{nonpercom}
    In non-periodic case, if the smooth internal arcs of $(v,t)$ and $(v',t')$ terminate at the same iterative preimage of the critical point $-c(\a)$, then 
    $\RRR_\a^L(v,t)$ and $\RRR_\a^R(v',t')$ have the same end, i.e. they coincide outside $B_\a^x(v)$ for some $0<s<1$.

    Conversely, if two left and right internal ray $\RRR_\a^\iota(v,t)$ and $\RRR_\a^{\iota'}(v',t')$ have the same end, then it must be the non-periodic case,  $(v,t)$ and $(v',t')$ terminate at the same iterative preimage of the critical point $-c(\a)$, and $\iota$ is different from $\iota'$.
\end{rmk}

Smooth internal rays and left or right turning internal rays are all called \emph{generalized internal rays}.
We have the following properties concerning the intersections of generalized internal rays.

\begin{cor}[intersections of generalized internal rays]\label{intersection}
    The following holds.
    \begin{itemize}
        \item[(I1)] Smooth internal ray does not intersect any other smooth internal ray or any left-turning or right-turning internal ray in the interior.   
        \item[(I2)] For any $(v',t') \neq (v,t)$,  $\RRR_\a^{L,R}(v',t')$ belongs to the closure of a connected component of $B_\a(v) \setminus \RRR_\a^{L,R}(v,t)$.
    \end{itemize}
\end{cor}

\begin{proof}
    For (I1), let $R_\a^v(t)$ be a smooth internal ray, and $R(v',t')$ be another generalized internal ray of $(v',t')$ and $z$ is an intersection point of these two rays in the interior.
    Since $R_\a^v(t)$ is a smooth internal ray, $z$ cannot be an iterative preimage of $-c(\a)$. 
    Notice that $v_n=\sigma^n(v)$ and $v_n'=\sigma^n(v')$ always belong to the same Fatou component and are eventually periodic.
    It follows that $v_n=v_n'$ are both periodic for $n$ large.
    Thus, we can find $n$ large with $v_n=c_n'=c$, and an intersection point $z':=f_\a^n(z) \in B_\a^c(s_\a^c)$ of two generalized internal rays $R_\a^{c}(t_n)$ and $\RRR_\a^\iota(c,t'_n)$ where $(c,t_n):=\g^n(v,t)$ and $(c,t'_n):=\g^n(v',t')$.
    Hence $z'$ is an intersection point of two smooth internal arcs starting at the same point.
    This is clearly impossible.
    Hence (I1) holds.

    For (I2), we may assume that $B_\a(v)=B_\a(v')$, otherwise the conclusion trivially holds.
    We prove that it by contradiction.
    Assume that $\RRR_\a^{R,L}(v,t)$ and $\RRR_\a^{R,L}(v',t')$ cross at $\omega$, that means for any small neighborhood $U$ of $\omega$, $(\RRR_\a^{R,L}(v,t) \cap U)\setminus \RRR_\a^{R,L}(v',t')$ is disconnected.
    Using the same argument as (I1), generalized internal rays can only intersect at an iterative preimage of $-c(\a)$.
    Passing to iteration, we see that $R_\a^{R,L}(v_n,t_n)$ and $R_\a^{R,L}(v'_n,t'_n)$ cross at $-c(\a)$.
    By Lemma \ref{lem32}, we may assume that $(v_n,t_n)=(x,t_\a)$ and $(v_n',t_n')=(y,t_\a')$.
    By the construction of left and right internal rays, $\RRR_\a^{R,L}(x,t_\a)$ and $\RRR_\a^{R,L}(y,t_\a')$ can not cross at $-c(\a)$.
    This is a contradiction.
    Thus (I2) holds.

\end{proof}

\subsection{More discussions on Periodic Cases}
We have already seen that if $f_\a$ has periodic combinatorics, every non smooth left and right internal rays have infinitely many turning points.
In this part, we define infinitely many more general turning internal rays in periodic case.

Given a sequence $\underline{\iota}:=(\iota_1,\iota_2,\dots,\iota_n,\dots) \in \{ L,R \}^{\mathbb N}$.
A semi-differentiable curve $\gamma: [0,1) \to \C$ with an infinite turning set $T:=\{ a_1,a_2,\dots,a_n,\dots \}$ is said to be a \emph{$\underline{\iota}$-turning curve} if for every $n$, $a_n$ is a $\iota_n$-turning point of $\gamma $.
A $L$-turning curve curve $\gamma$ with an infinite turning set is a $\underline{\iota}$-turning curve if we pick $\underline{\iota}:=\{ L,L,L,\dots \}$ as the constant.


\begin{prop}[$\iota$-turning internal rays]\label{iotaturning}
    Let $\a \in \S$ be an ABC-hyperbolic map with periodic combinatorics.
    For each each sequence $\underline{\iota}:=(\iota_1,\iota_2,\dots,\iota_n,\dots) \in \{ L,R \}^{\mathbb N}$, and $(v,t) \in \O \times  \RZ$ with $s_\a^v(t)<1$, there exists a $\i$-turning curve $\gamma_{\underline{\iota}}: [0,1) \to B_\a(v)$ such that {\rm (P2)} holds.
    
\end{prop}

\begin{proof}[Sketch of the proof]
    The proof is very similar to Lemma \ref{lem34} as well as Lemma \ref{lem31}.
    So we sketch the proof of this proposition here and focus on the differences.

    Since we are in the periodic case, by Lemma \ref{lem32} and Remark \ref{single}, $(v,t)$ is eventually mapped to $(x,t_\a)$ by $\g$.
    Passing to iterations, we just need to consider the case that $(v,t)=(x,t_\a)$.
    Let $S_{\i}$ be the set of $r$ such that there exists a turning curve $\gamma_r: [0,r) \to B_\a(v)$ with a finite turning set $T_n:=\{ a_1,a_2,\dots,a_n \}$ where $n:=n(r)$ such that $a_k$ is a $\iota_k$-turning point of $\gamma_r$ and (P2) holds.
    Following the idea of the proof of Lemma \ref{lem34}, we can show that $\sup S_{\i}=1$. 
    
    In this case, we further claim that as $r \to 1$, we have $n(r) \to \infty$ since $(x,t_\a)$ is periodic.
    In fact, if the period of $(x,t_\a)$ is $mp$, then the turning points of $\gamma_r$ is precisely $\{ a_n:=\sqrt[2^m]{r_\a} : 1\leq  m \leq n(r)\}$ where $n(r)$ is the largest integer $n$ such that $r^{2^n}>r_\a$.
    Thus, we conclude that $n(r) \to \infty$ as $r\to 1$.
    Therefore, as the proof of Lemma \ref{lem34}, taking a monotone increasing sequence $r_k \to 1$, we can glue together these turning curves $\{ \gamma_{r_k} \}$ and obtain the $\i$-turning curve $\gamma_{\i}$.
    
\end{proof}

By Proposition \ref{iotaturning}, for sequence $\underline{\iota}:=(\iota_1,\iota_2,\dots,\iota_n,\dots) \in \{ L,R \}^{\mathbb N}$, and $(v,t) \in \O \times  \RZ$ with $s_\a^v(t)<1$, define the \emph{$\i$-turning internal ray} of $(v,t)$ to be 
$$
    \RRR_\a^{\i}(v,t):=\gamma_{\iota}([0,1)).
$$
Clearly, for $\iota \in \{ L, R \}$, the previously defined $\iota$-turning curve is a particular $\i$-turning internal ray if we take $\i$ as the constant $(\iota,\iota,\dots)$.
In fact, every kind of generalized internal rays we defined before can be understood as a $\i$-turning internal ray with some $\i$ and corresponding turning set $T_{\i}$.
Smooth internal rays can be understood as $\i$-turning internal ray with $T_{\i}=\emptyset$.
For non-periodic case, a $\iota$-turning internal rays can be understood as $\i$-turning internal ray with finite string $\i=(\iota)$ and singleton $T_{\i}$.

The $\i$-turning internal rays also have nice dynamical properties.
In fact, suppose that $\RRR_\a^{\i}(v,t)$ is a $\i$-turning internal ray, then we already seen that $(v,t)$ is eventually mapped to $(x,t_\a)$ which is periodic under $\g$.
Thus it is eventually mapped to $(c,t_c)$ for some $t_c$.
Then $f_\a$ maps an $\i$-turning internal ray $\RRR_\a^{\i}(v,t)$ to $\RRR_\a^{\i'}(\g(v,t))$ where $\i':=(t_2,t_3,\dots)$ is the left shift of $\i:=(t_1,t_2,\dots)$ if $(v,t)=(c,t_c)$, and $\i'=\i$ if $(v,t)\neq (c,t_c)$.
One may verify that if $(v,t)$ and $\i$ are both periodic, then the $\i$-turning internal ray $\RRR_\a^{\i}(v,t)$ is periodic under $f_\a$.

\subsection{Landing Properties of Generalized Internal Rays}\label{ssec3.4}
In this part, we study the landing properties of generalized internal rays.
We've already seen that every generalized internal ray can be understood as some $\i$-turning internal ray $\RRR_\a^{\i}(v,t)$ with $\i$, $T_{\i}$ and $(v,t)$.
Suppose that $\RRR_\a^{\i}(v,t)$ is parameterized by $\gamma:[0,1) \to B_\a(v)$.
If the limit $\lim_{s\to 1}\gamma(s)=z$ exits, then we say that the $\i$-turning internal ray $\RRR_\a^{\i}(v,t)$ \emph{lands} at $z$.
We will show in all cases, every generalized internal ray land at a point on the boundary of the bounded Fatou component (Proposition \ref{internalland}).

\begin{prop}[every generalized internal ray lands]\label{internalland}
    For generic $\a \in \S$, every generalized internal ray lands at a point on $\partial \tilde{B}_\a$.
\end{prop}

The Proof of Proposition \ref{internalland} starts from here.
If $-c(\a) \notin \tilde{B}_\a$, then by the Carath\'eodory Theorem, the conclusion of Proposition \ref{internalland} certainly holds.
Therefore, we only need to consider the ABC-hyperbolic case.
The proof can be done by imitating the the proof of the classical result on landing properties of external rays and the hyperbolicity of $\a$.
However, here we provide a different approach by using the puzzle technique since it conceals more information.

Assume that $\a \in \mathcal{C}(\S)$ is generic.
The essential part of our puzzle construction is the finite graph $\Gamma(\varTheta,T)$ where $T$ and $\varTheta$ be two finite set of $\RZ$ satisfying the following nice conditions.
\begin{itemize}
    \item[(N1)]  $\varTheta $ is a finite set consisting of $\theta$ which are periodic under $\tau_3$ such that $R_\a(\theta)$ lands at a repelling periodic point.
    \item[(N2)]  $T$ is a finite set consisting of $t$ which is is periodic under $\tau_2$ with $s_\a^c(t)=1$ and the forward orbit of $\psi_\a^c(e^{2\pi i t})$ does not contain $-c(\a)$.
\end{itemize}
For every $t \in T$, pick an external ray $R_\a(\theta(t))$ which lands at $\psi_\a^c(e^{2\pi i t})$ for $t \in T$.
Since $f_\a$ is locally orientational preserving near $\psi_\a^c(e^{2\pi i \theta})$,  there exists $m \geq 1$ such that the curve $\gamma_\a^c(t,\theta):=R_\a(\theta(t)) \cup R_\a^c(t) \cup \{ \psi_\a^c(e^{2\pi i \theta})\}$ is invariant under $f_\a^m$.
Fix $T$ and $\varTheta$, we define an invariant graph under $f_\a$ by
$$\Gamma_\a(\varTheta,T):=\bigcup_{n \geq 0 } f_\a^n \left( \left(\bigcup_{t \in T} \gamma_\a^c(t,\theta)\right) \cup \left( \bigcup_{\theta \in \varTheta}R_\a(\theta) \right)  \right).$$

We shall also include equipotential lines in the graph.
For generic $\a \in \mathcal{C}(\S)$, and $s>1$, $0<s'<1$, denote 
$$
    X_\a:=\C \setminus \overline{\Omega_\a(s) \cup B_\a^c(s') \cup B_\a^{c_1}(s') \cup \cdots \cup B^{c_{p-1}}_\a(s')}.
$$ 
Clearly, we have $f_\a^{-1}(X_\a)\subset X_\a$.
For $n \geq 0$, define $X_n^\a:=f_\a^{-n}(X_\a)$, and $G_n^\a:=f_\a^{-n}(\Gamma_\a(\varTheta,T) \cup \partial X_\a)$.
Connected components of $X_n^\a \setminus G_n^\a$ are called \emph{puzzle pieces} of depth $n$.
For $z \notin G_n^\a$, there exists a unique puzzle piece of depth $n$ containing $z$ denoted by $P_n^\a(z)$.
For $n \geq 1$, the map $f_\a: P_n^\a(z) \to P_{n-1}^\a(f_\a(z))$ is a holomorphic proper map of degree at most $2$.
It has degree $2$ if and only if $-c(\a) \in P_n^\a(z)$.
${\bf P}_\a={\bf P}_\a(\varTheta,T,s,s’)$ is called a \emph{puzzle system}.

Let $P_n^\a \subset P_{n-1}^\a \subset \cdots \subset P_1^\a \subset P_0^\a$ be a nested sequence of puzzle pieces of ${\bf P}_\a$, define the \emph{impression} of this nested sequence to be 
$I:=\bigcap_n \overline{P_n^\a}$.
For $z \in J_\a \setminus \bigcup_nG_n^\a$, there exists a unique nested sequence $\{ P_n^\a(z) \}$ containing $z$, the intersection $I_\a(z):=\bigcap_n \overline{P_n^\a(z)}$ is called the \emph{impression} of $z$.

     \begin{figure}[h]
  \begin{center}
    \includegraphics[width=0.75\linewidth]{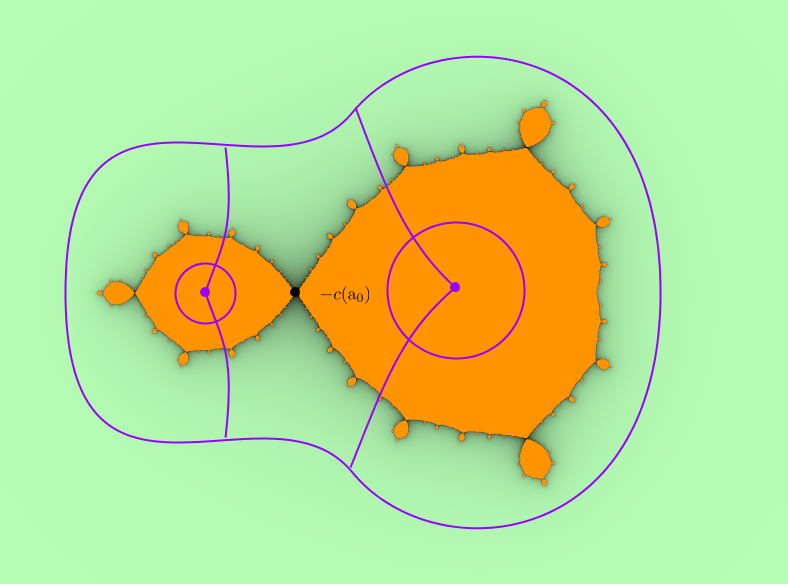}
    \caption{}\label{figring2}
 \end{center}
\end{figure}


\begin{lem}[every impression is a singleton]\label{singleton}
    Let $f_\a$ be an ABC-hyperbolic map.
    Every impression of the puzzle system $\mathrm{P}_\a(\varTheta,T,s,s’)$ is a singleton.
\end{lem}

\begin{proof}
    The proof is a standard application of the thicken technique \cite{milnor1992local}.
    Since $-c(\a)\in \tilde{B}_\a$, then for $n$ large, each puzzle piece of depth $n$ does not contain the forward orbit of $-c(\a)$.
    Thus, without loss of generality, we assume that every puzzle piece of depth $0$ omits the orbit of $-c(\a)$.
    Let $P^{(1)}_\a, P^{(2)}_\a,\dots, P^{(m)}_\a$ be all the puzzle pieces of depth $0$.
    
    For $1\leq k \leq m$, we apply the thicken technique to get a larger domain $\tilde{P}^{(k)}_\a$ containing the closure of $P^{(k)}_\a$
    By the construction of $\Gamma_\a(\varTheta,T)$.
    $\partial P^{(k)}_\a$ consists of internal rays, external rays and equipotential lines.
    Notice that $\partial P^{(k)}_\a$ contains exactly two internal rays with respect to some $v_k \in O(\a)$.
    Assume that $R_\a^{v_k}(t^\pm_k) $ are the two internal rays in $\partial P^{(k)}_\a \cap B_\a(v)$ such that for $t \in (t_k^-,t_k^+)_+$, $R_\a^{v_k}(t)$ intersects $ P^{(k)}_\a $.
    Pick $\tilde{t}^\pm_k \in (t^+_k,t_k^-)_+$ close to $t^\pm_k$ which is periodic under $\tau_2$ with $s_\a^{v_k}(\tilde{t}^\pm_k)=1$.
    Let $R_\a(\theta^\pm_k)$ be the external ray lands at $\psi_\a^{v_k}(e^{2\pi i \tilde{t}^\pm_k})$ respectively, and $\tilde{P}^{(k)}_\a$ be the connected component of $\C \setminus (\overline{R_\a(\theta^-_k)} \cup \overline{R_\a(\theta^+_k)} \cup R_\a^{v_k}(\tilde{t}^-_k) \cup R_\a^{v_k}(\tilde{t}^+_k))$ which intersects $R_\a^{v_k}(t)$ for $t \in (\tilde{t}_k^-,\tilde{t}_k^+)_+$.
    By construction, the closure of of $P^{(k)}_\a$ is contained in $\tilde{P}^{(k)}_\a $.

    Fix $k$, there are $3$ branch of $f_\a^{-1}$ on $\tilde{P}^{(k)}_\a$, all of them are conformal.
    Thus, they decrease the hyperbolic distance.
    Since the the closure of of $P^{(k)}_\a$ is contained in $\tilde{P}^{(k)}_\a $ for all $k$, there exists a uniform contracting constant $0<\kappa<1$ on the union of all puzzle pieces of depth $0$.
    Let $d$ be the maximal diameter of the closure of all puzzle pieces of depth $0$.
    Then for any $P_n$ of depth $n$, we have $\mathrm{diam}(\overline{P_n})\leq d \kappa^n \to 0$.
    Thus, every impression is a singleton.

\end{proof}

As a immediate corollary, we present the proof of Proposition \ref{internalland}

\begin{proof}[Proof of Proposition \ref{internalland}]

    Let $\RRR_\a^{\i}(v,t)$ be a $\i$-turning internal ray, which is parameterized by $\gamma:[0,1) \to B_\a(v)$.
    If $\RRR_\a^{\i}(v,t)$ is a smooth internal ray which belongs to $\bigcup f_\a^{-n}(\Gamma_\a(\varTheta,T))$, then it lands at a point on $\partial \tilde{B}_\a$.
    In the following, we the opposite.
    Let $X$ denote the set of all limit points of $\gamma(s)$ as $s\to 1$.
    By the (I1) of Corollary \ref{intersection}, we observe that for every $n$, $X$ must be contained in some puzzle piece $P_n^\a(X)$ of depth $n$.
    Thus, $X$ is contained in the impression $I(X):=\bigcap_n\overline{P_n^\a(X)}$.
    By Lemma \ref{singleton}, $I(X)=X$ must be a singleton $\{z\}$.
    Hence $\RRR_\a^{\i}(v,t)$ lands at $z$.

\end{proof}


A classical result obtained by the Maximal Modules Principle shows that any two smooth internal rays in a bounded Fatou component cannot land at a common point. 
Using the same idea, we obtain the corresponding result for left or right turning internal rays.

\begin{cor}[distinct landing points]\label{distinctlanding}
    Let $\a \in \S$ be an ABC-hyperbolic map.
    Any two left or right turning internal rays in a common Fatou component which do not have the same end cannot land at a common point.
\end{cor}

\begin{proof}
    We prove it by contradiction.
    Let $R_1:=\RRR_\a^{\iota_1}(v_1,t_1)$ and $R_2:=\RRR_\a^{\iota_2}(v_2,t_2)$ be two left or right internal rays in $B_\a(v_1)=B_\a(v_2)$ with $\iota_1,\iota_2\in \{ L,R \}$.
    Suppose that $R_1$ and $R_2$ lands at a common point $w$.
    If $f_\a^n(R_1)=f_\a^n(R_2)$ holds for some $n \geq 0$, then $w$ is a critical point of $f_\a^n$.
    This is impossible since $f_\a$ is hyperbolic.

    In the following, we consider the case that $f_\a^n(R_1)\neq f_\a^n(R_2)$ for every $n \geq 0$. Notice that the orbits of $v_1,v_2$ are eventually periodic and always belong to Fatou component.
    It follows that there exists $n$ such that $\sigma^n(v_1)=\sigma^n(v_2):=c$.
    Therefore, by passing to iterations, we may assume that $v_1=v_2=c$, and $t_1=t_2=t$.
    Then $R_1$ and $R_2$ must be $\RRR_\a^L(c,t)$ and $\RRR_\a^R(c,t)$.
    Let $U$ be the unique bounded component of $\C \setminus (\overline{R_1} \cup \overline{R_2})$.
    Then either $\overline{U}$ contains $c$ or $c'$ where $c' \in B_\a(c)$ satisfies that $\sigma(c)=\sigma(c')$.
    Without loss of generality, we assume that $c \in U$.
    By the Maximal Modules Principle, we see that $U \subset B_\a(c)$.
    It follows that for any $t\in \RZ$ with $(c,t)\notin \M'$, the smooth internal ray $R_\a^c(t)$ land at a point in $\overline{U} \cap \partial B_\a(c)$.
    But $\overline{U} \cap \partial B_\a(c)=\{ w \}$.    
    Thus, they must  all land at $w$.
    This is impossible since $(c,t) \notin \M'$ can be periodic or non-periodic.
    
\end{proof}

If $s_\a^v(t)<1$, and the left and right turning internal rays $\RRR_\a^L(v,t)$  and $\RRR_\a^R(v,t)$ land at two points $z$ and $w$, then we say that $z$ and $w$ are \emph{joined} by the left-right internal ray $\RRR_\a^{L,R}(v,t)$.

\begin{lem}[joined by left-right internal ray]\label{joined}
    If $z$ and $w$ are joined by a left-right internal ray $\RRR_\a^{L,R}(v,t)$, then we can find $(v',t')$ whose forward orbit under $\g$ contains $(x,t_\a)$ with $B_\a(v)=B_\a(v')$ such that $z$ and $w$ are joined by the left-right internal ray $\RRR_\a^{L,R}(v',t')$.
\end{lem}

\begin{proof}
    Since $s_\a^v(t)<1$, by Lemma \ref{lem32}, $(v,t)$ is eventually mapped to $(x,t_\a)$ or $(y,t_\a')$.
    If we are in the former case, then the conclusion already holds.
    Hence we just need to deal with the case that the forward orbit of $(v,t)$ contains $(y,t_\a')$.
    If we are in the periodic case, then by Remark \ref{single}, we see that its forward orbit will eventually contains $(x,t_\a)$.
    Hence the conclusion follows.
    If we are in the non-periodic case, choose $(v',t')$ such that the smooth internal arc of $(v,t)$ and $(v',t')$ terminate at the same point.
    Suppose that $\RRR_\a^L(v,t)$ land at $z$ and $\RRR_\a^R(v,t)$ land at $w$.
    Then by Remark \ref{nonpercom}, we see that $\RRR_\a^R(v',t')$ land at $z$ and $\RRR_\a^L(v',t')$ land at $w$.
    Hence $z$ and $w$ are joined by the left-right internal ray $\RRR_\a^{L,R}(v',t')$.
    
\end{proof}

\subsection{External Classes and External Angles}\label{ssec3.5}
Given an ABC-hyperbolic map $\a$, the landing point of a generalized internal ray $\RRR_\a^{\i}(v,t)$  is denoted by $z_\a^{\i}(v,t)$.
Since $J_\a$ is connected and locally connected, according to \cite{wanderingorbitportrait}, there are finite many external rays which land at $z_\a^{\i}(v,t)$.
The angle of these external rays are called \emph{external angles} of $\RRR_\a^{\i}(v,t)$ or $z_\a^{\i}(v,t)$.
Let $E_\a^{\i}(v,t)$ be the set of these external angles.
This set is called the \emph{external class} of $ \RRR_\a^{\i}(v,t)$ or $z_\a^{\i}(v,t)$.

For convenience, we pick a preferred element in each external class.
Suppose that  $E_\a^{\i}(v,t)=\{ \theta_1,\theta_2,\dots,\theta_n \}$, and the rays $\RRR_\a^{\i}(v,t),R_\a(\theta_1),R_\a(\theta_2),\dots,R_\a(\theta_n)$ are arranged in positive cyclic order near $z_\a^{\i}(v,t)$.
We define the \emph{preferred external angle} $\vt_\a^{\i}(v,t)$ of $ \RRR_\a^{\i}(v,t)$ or $z_\a^{\i}(v,t)$ to be $\theta_1$.
Notice that an ABC-hyperbolic map $f_\a$ is locally conformal at every point in $J_\a$.
By the orientational preserving property of conformal mappings, one may verify that 
\begin{equation}\label{tau_3}
   \tau_3 ( \vt_\a^{\i}(v,t)) = \vt_\a^{\i} ( \g(v,t) ),\qquad \forall (v,t) \in \O \times \RZ.
\end{equation}

\begin{lem}[external angles and combinatorics]\label{p&np}
    An ABC-hyperbolic map $f_\a$ has periodic combinatorics if and only if $\vt_\a^\iota(x,t_\a)$ is periodic under $\tau_3$ for $\iota \in \{L,R \}$.
\end{lem}

\begin{proof}
       If $f_\a$ has periodic combinatorics, then by Remark \ref{single}, $(x,t_\a)$ is periodic under $\g$.
       By \eqref{tau_3}, we see that $\vt_\a^\iota(x,t_\a)$ is periodic for $\iota \in \{L,R \}$.

       Conversely, if $f_\a$ has non-periodic combinatorics, then $(x,t_\a)$ is non-periodic. In this case, $\vt_\a^\iota(x,t_\a)$ must be non-periodic.
       We prove it by contradiction. 
       Assume that $\theta:=\vt_\a^\iota(x,t_\a)$ has period $q$ under $\tau_3$. 
       Thus $\RRR_\a^\iota(x,t_\a)$ and $R_\a(\theta)$ land at a common periodic point $z$ with period dividing $q$.
       Since $(x,t_\a)$ is non-periodic, then $f_\a^q$ maps $\RRR_\a^{L,R}(x,t_\a)$ to the smooth internal ray $R_\a^{v}(t)$ where $(v,t):=\g^q(x,t_\a)$ which also lands at $z$.
       By Corollary \ref{distinctlanding}, $B_\a(v')\neq B_\a(x)$ is another Fatou component which contains $z$ on its boundary.
       Since $f_\a^q$ is locally orientational preserving and fix $R_\a(\theta)$, it must also fix the smooth internal ray $R_\a^{v}(t)$. Thus, we have $  f_\a^q(B_\a(v)) = f_\a^q(B_\a(x)) =B_\a(v) $.
       Since $z \in \partial B_\a(v) \cap \partial B_\a(x)$, we see that $z$ is a critical point of $f_\a^q$. This contradicts the hyperbolicity of $f_\a$.
      
\end{proof}


\section{Lamination of Boundary Maps}\label{sec5}

\subsection{Parameter Internal Rays}

In this part, we study the dependence of parameter $\a$ for  generalized internal rays, external classes on parameters, and use it to study the lamination of maps on the boundaries of hyperbolic components. 

Let $\H$ be a hyperbolic component of type \A, \B, or \CCC.
Recall that $\ell:=\ell(\H)$ is the smallest integer such that $f_\a^\ell(-c(\a)) \in B_\a(c)$ for $\a \in \H$.
According to \cite{milnor2008cubic}, we have the following parameterization of $\H$.

\begin{prop}[parameterization of hyperbolic components]\label{parameterozation}
    The map $\Phi_\H  : \H \cap \S \to \D$, $\a \mapsto \phi_\a^c(f_\a^\ell(-c(\a)))$ is a  proper of degree 
    $d(\H)=\deg (f_\a^p|_{B_\a(c)} ) - 1.
$
\end{prop}

For $t\in \RZ$, connected components of $\Phi_\H^{-1}((0,1)e^{2\pi i t} )$ are called {\emph{parameter internal rays}} of angle $t$ in $\H \cap \S$.
Therefore, fixing $t \in \RZ$, there are exactly $d(\H)$ parameter internal rays with angle $t$ in $\H$.
Every parameter internal ray $\mathcal{R}_\H(t)$ of angle $t$ can be canonically parameterized by $\Gamma: (0,1) \to \mathcal{R}_\H(t)$ with $\Phi_\H \circ \Gamma=\mathrm{id.}$.

    Fix any $\a \in \H \cap \S$, let $t_\a$, $r_\a$ be given in Lemma \ref{lem32}.
    the parameterization map $\Phi_\H(\a)=\phi_\a^c(f_\a^\ell(-c(\a)))$ is equal to $r_\a e^{2\pi i t_\a}$.
    Thus, for any $\a$ in a parameter internal ray $\mathcal{R}_\H(t_0)$, we have $t_\a=t_0$.
    The following Lemma shows that when the parameter $\a$ moves on the parameter internal ray $\mathcal{R}$, the structure of external rays and generalized internal rays are preserved.



\begin{lem}[invariance on parameter internal rays]\label{surgery}
    Let $\H$ be a hyperbolic component of type \A, \B, or \CCC, and $\a_1 \in \H$ with $\Phi_\H(\a_1)=r_0e^{2\pi i t_0}$.
    Let $\mathcal{R}_\H(t_0)$ be a parameter internal ray in $\H$, then for each $\a \in \mathcal{R}_\H(t_0) $, there exists a quasiconformal map $h_\a:\CC \to \CC$ such that the following diagram commutes.
    \begin{equation*}
    \xymatrix{\ar @{} [dr]
    \CC \ar[d]^{f_{\a_1}} \ar[r]^{h_\a} & \CC \ar[d]^{f_\a} \\
     \CC \ar[r]^{h_\a}        & \CC   }
\end{equation*}
Moreover, $h_\a$ preserves every external rays and generalized internal rays.
\begin{enumerate}
    \item $h_\a |_{\Omega_{\a_1}}=\psi_\a\circ\varphi_{\a_1}$ is a conformal map.
    For $\theta \in \RZ$, $h_\a(R_{\a_1}(\theta))=R_\a(\theta)$.
    \item $\M'_\a=\M'_{\a_1}$. For $(v,t) \in \M'_\a$, we have $h_\a(\RRR_{\a_1}^{\i}(v,t))=\RRR_\a^{\i}(v(\a),t)$ and $E_\a^{\i}(v(\a),t)=E_{\a_1}^{\i}(v,t)$, $\vt_\a^{\i}(v(\a),t)=\vt_{\a_1}^{\i}(v,t)$ where $v(\a):=h_\a(v)$.
    
\end{enumerate}

\end{lem}

\begin{proof}
    We reconstruct the internal ray $\mathcal{R}_\H(t)$ by quasconformal surgeries.
    For $s>0$, let $\ell_s: \C \to \C$ be the linear map given by $\ell_s(z)=z|z|^{s-1}$, and $\sigma_s:=\ell_s^*(\sigma_0)$ where $\sigma_0$ denote the standard complex structure.
    
    Let $\sigma'_{s}$ be the complex structure such that 
    \begin{itemize}
        \item $\sigma'_{s}:=(\phi_{\a_1}^c)^*(\sigma_s)$ on $B_{\a_1}^c( s_{\a_1}^c )$;
        \item $\sigma_s'=\sigma_0$ outside $\tilde{B}_{\a_1}$;
        \item $f_{\a_1}^*(\sigma'_s)=\sigma_s'$.
    \end{itemize}
    This complex structure $\sigma_s'$ has bounded dilatation and depends continuously on $s >0$. 
    By the Measurable Riemann Mapping Theorem, there exists a quasiconformal map $h_s:\CC \to \CC$ such that $h_s^*(\sigma_0)=\sigma_s'$ and $h_s(\infty)=\infty$, $\lim_{z\to \infty} \frac{h_s(z)}{z}=1$, and $h_s(-c(\a_1))+h_s(c(\a_1))=0$.
    It follows that $h_s\circ f_{\a_1}\circ h_s^{-1}$ is a monic and centered cubic polynomial which depends continuously on $s>0$.
    By assigning $h_s(c(\a_1))=c(\a(s))$, we have $\mathcal{R}:=\{ \a(s) : s>0 \}  $ forms a curve in $\hat{\P}$.
    Since $f_{\a(s)}= h_s\circ f_{\a_1}\circ h_s^{-1} $ is hyperbolic, then $\mathcal{R} \subset \H $.

    Since $\sigma_s'=\sigma_0$ on $\Omega_{\a_1}$, then $h_s|_{\Omega_{\a_1}}$ is a conformal conjugacy between $f_{\a_1}$ and $f_{\a(s)}$.
    Thus, $\varphi_{\a(s)}\circ h_s$ conjugates $f_{\a_1}$ and $z\mapsto z^3$.
    Combining the fact that $h_s(\infty)=\infty$, $\lim_{z\to \infty} \frac{h_s(z)}{z}=1$, we conclude that $ \varphi_{\a(s)}\circ h_s $ is the B\"ottcher map on $\Omega_{\a_1}$.
    Hence it must coincide with $\varphi_{\a_1}$.
    It follows that for $\theta \in \RZ$, $h_s(R_{\a_1}(\theta))=R_{\a(s)}(\theta)$.
    For the same reason, one may verify that $\phi_{\a(s)}^c\circ h_s|_{B_{\a_1}^c(s_{\a_1}^c)} =\ell_s \circ \phi_{\a_1}^c$.

    Now we verify that $\mathcal{R}$ coincides with a parameter internal ray of angle $t_0$.
    Notice that
    \begin{align*}
        \Phi_\H(\a(s)) &=\phi_{\a(s)}^c \circ f_{\a(s)}^\ell (-c(\a(s)))  = \phi_{\a(s)}^c \circ f_{\a(s)}^\ell \circ h_s(-c(\a_1)) \\
        &= \phi_{\a(s)}^c \circ h_s  \circ f_{\a_1}^\ell(-c(\a_1)) =\ell_s \circ \phi_{\a_1}^c \circ f_{\a_1}^\ell(-c(\a_1)) \\
        &=\ell_s \circ \Phi_\H(\a_1)=r_0^s e^{2\pi i t_0}.
    \end{align*}
    Since $\{r_0^s: s>0\}=(0,1)$, the conclusion follows.

    For $v \in \tilde{O}(\a_1)$, denote $v(\a(s)):=h_s(v(\a_1))$.
    Given a generalized internal ray $\RRR_{\a_1}^{\i}(v,t)$, since $h_s|_{B_{\a_1}^c(s_0)} \circ \psi_{\a_1}^c=\psi_{\a(s)}^c$ and $h_s$ is an orientational preserving conjugacy, we deduce that $h_s(\RRR_{\a_1}^{\i}(v,t))=\RRR_\a^{\i}(v(\a),t)$.
    By the continuity of $h_s$, we see that $h_{\a(s)}(z_{\a_1}^{\i}(v,t))=z_{\a(s)}^{\i}(v(\a(s)),t)$ and $h_{s}(z_\theta(\a_1))=z_{\a(s)}(\theta)$ for any $\theta\in \RZ$.
    It follows that $E_{\a(s)}^{\i}(v(\a(s)),t)=E_{\a_1}^{\i}(v,t)$.
    
    By changing the notation $h_s$ by $h_\a$ where $\a=\a(s)$, the proof is completed.

\end{proof}


According to \cite{wang2017hyperbolic}, it is proved that $\partial\H \cap \S$ is a Jordan curve.
It follows that every parameter internal ray \emph{lands}, i.e. $\Gamma(s)$ converges to some $\a \in \partial \H \cap \S$ as $s\to 0$.
Conversely, every boundary point in $\partial \H \cap \S$ is the landing point of a unique parameter internal ray.

In the rest of this article, we make the following standing hypothesis.

\begin{asp}
    $\H$ is a bounded hyperbolic component of type {\rm (A), (B) or (C)},
    $\a_0 \in \partial \H \cap \S$, and $\mathcal{R}:=\mathcal{R}_\H(t_0)$ is the parameter internal ray landing at $\a_0$.
\end{asp}

Notice that the original definition of $\O=\O(\a)$, $\M_\a'$, $\vt^\iota_\a(v,t)$, and $g,\g$ are all dependent on $\a$. 
However, by Lemma \ref{surgery}, we that they are independent of $\a$ on the parameter internal ray $ \mathcal{R} $.
Hence, we will continue to use the notation $\O$, $v \in \O$, $g$ and $\g$.
We will also use the notation $\M'$ and $\vt^\iota(v,t)$ for $(v,t) \in \M'$ and $\iota \in \{ L,R \}$.


\subsection{Visual Lamination}

In this part, we start the study of the real lamination $\L_\R(\a_0)$ of $\a_0$.
Our strategy is to study the structure of the left-right internal rays.

From the pictures of Julia sets generated by computers, we see that as $\a \in \mathcal{R}$ tends to $\a_0$, for any $(v,t) \in \M'$, the landing points $z_\a^L(v,t)$ and $z_\a^R(v,t)$ which are joined by the left-right internal ray $\RRR_\a^{L,R}(v,t)$  becomes closer and closer.
It is reasonable to conjecture that for the limit map $\a_0$, the two external ray $R_{\a_0}(\vt^L(v,t))$ and $R_{\a_0}(\vt^R(v,t))$ land at a common point.
For this reason, it is natural to define the following visual lamination $\Lambda_\R(\a_0)$ of $\a_0$.



\begin{defi}[visual lamination]
    Define the relation $\Lambda(\a_0) \subset \RZ \times \RZ$ which contains the diagonal $\{(\theta,\theta):\theta \in \RZ\}$ and for $\theta\neq \theta'$, we have $ (\theta,\theta') \in \Lambda(\a_0) $ if and only if there exist $(v,t) \in \M'$ such that one of $\theta,\theta'$ is $ \vt^L(v,t)$ and the other is $\vt^R(v,t)$.
    
    The \emph{visual lamination} $\Lambda_\R(\a_0)$ of $\a_0$ is defined to be the smallest equivalence relation containing $\L_\R(\H) \cup \Lambda(\a_0)$.
\end{defi}

The main purpose of this section is to show that the visual lamination $\Lambda_\R(\a_0)$ coincides with the real lamination $\L_\R(\a_0)$.

\begin{lem}[characterization of $\Lambda(\a_0)$]\label{Lc}
  The relation $\Lambda(\a_0)$ is a minimal $\tau_3$ invariant equivalence relation with the generator $E^*=\{ \vt^L(x,t_0), \vt^R(x,t_0) \}$.
  All non-singleton equivalence classes of $\Lambda(\a_0)$ are $\{  \vt^L(v,t), \vt^R(v,t) \}$ for all $(v,t)$ whose forward orbit under $\g$ contains $(x,t_0)$.
\end{lem}

\begin{proof}
    First, we show that $\Lambda(\a_0)$ is an equivalence relation.
    We only need to verify the transitivity.
    Suppose that $(\theta_1,\theta_2),(\theta_2,\theta_3) \in \Lambda(\a_0)$ are two elements outside the diagonal line.
    In the following we show that $\theta_1=\theta_3$.
    
    We prove it by contradiction, assume that $\theta_1\neq \theta_3$.
    By the definition of $\Lambda(\a_0)$, there exist $(v,t),(v',t')$ such that $\{ \theta_1,\theta_2 \}=\{ \vt^L(v,t), \vt^R(v,t) \}$ and $\{ \theta_2,\theta_3 \}=\{ \vt^L(v',t'), \vt^R(v',t') \}$.   
    Assume that $\theta_2=\vt^\iota(v,t)=\vt^{\iota'}(v',t')$.
    Since $\theta_1 \neq \theta_3$, we have $(v',t')\neq (v,t)$.    
    We claim that $v$ and $v'$ are not in the same Fatou component.
    If $B_\a(v)=B_\a(v')$, then by Corollary \ref{distinctlanding}, $\RRR_\a^{\iota}(v,t)$ and $\RRR_\a^{\iota'}(v',t')$ must have the same end.
    By Remark \ref{nonpercom}, we are in the non-periodic case and the smooth internal arc of $(v,t)$ and $(v',t')$ must terminate at the same preimage of $-c(\a)$ and $\iota$ is different from $\iota'$.
    Without loss of generality, we assume $\iota=L$, $\iota'=R$.
    Again by Remark \ref{nonpercom}, $\RRR_\a^R(v,t)$ and $\RRR_\a^L(v',t')$ must have the same end, and hence land at a common point.
    By the choice of external angles and the fact $B_\a(v)=B_\a(v')$, we see that $\theta_1=\theta_3$.
    This is a contradiction.
    Therefore, we must have $B_\a(v)\neq B_\a(v')$.
    But this is clearly impossible by the choice of external angles and the fact that $\theta_2=\vt^\iota(v,t)=\vt^{\iota'}(v',t')$.
    Hence, we also reach a contradiction.
    It follows that $\theta_1=\theta_3$.
    Hence the transitivity holds, which means $\Lambda(\a_0)$ is an equivalence relation.

    Next, we show that $\Lambda(\a_0)$ is minimal.
    The above proof also implies that the non-singleton equivalence classes of $\Lambda(\a_0)$ are precisely $\{  \vt^L(v,t), \vt^R(v,t) \}$ for all $(v,t) \in \M'$. 
    Denote $E^*:=\{ \vt^L(x,t_0), \vt^R(x,t_0) \}$.
    Let $E:=\{  \vt^L(v,t), \vt^R(v,t) \}$ be an equivalence class, then the landing points $z_1,z_2$ of the external ray of these two angles are joined by the left-right internal ray $\RRR_\a^{L,R}(v,t)$.
    By Lemma \ref{joined}, we can find $(v',t')$ whose forward orbit containing $(x,t_0)$ with $B_\a(v)=B_\a(v')$ such that $z_1,z_2$ are joined by $\RRR_\a^{L,R}(v',t')$.
    Since $B_\a(v)=B_\a(v')$, we deduce that $E=\{ \vt^L(v',t'), \vt^R(v',t') \}$.
    By \eqref{tau_3}, we see that $E$ is eventually mapped to $E^*$ under $\tau_3$.
    This means that $\Lambda(\a_0)$ is minimal.

    Finally, we show that $\Lambda(\a_0)$ is $\tau_3$-invariant.
    Let $E$ be an equivalence class of $\Lambda(\a_0)$, if $E$ is a singleton, then $\tau_3(E)$ is a singleton equivalence class.
    Assume that $E:=\{  \vt^L(v,t), \vt^R(v,t) \}$ is a non-singleton equivalence class with $(v,t) \in \M'$.
    Denote $(v',t'):=\g(v,t)$.
    By \eqref{tau_3}, we have $\tau_3( \vt^\iota(v,t) ) = \vt^\iota(v',t') $.
    If $(v',t') \in \M'$, then $\tau_3(E)=\{ \vt^L(v',t'), \vt^R(v',t') \}$ is also an equivalence class.
    If $(v',t') \notin \M'$, then $ \tau_3( \vt^L(v,t) ) = \tau_3(\vt^R(v,t)) := \theta $.
    It follows that the internal arc of $(v,t)$ must terminate at $-c(\a)$.
    Hence we may assume that $(v,t)=(x,t_0)$.
    In this case, we need to show that $\{ \theta\}$ forms a singleton equivalence class.
    Otherwise, there exits $(v'',t'') \in \M'$ such that $\theta=\vt^\iota(v'',t'')$ for some $\iota \in \{L,R \}$.
    It follows that there exists $n$ such that $\tau_3^n(\theta)=\vt^\iota(x,t_0)$ and hence $\tau_3^{n+1}(\theta)=\theta$.
    This implies that both $\theta$ and $\vt^\iota(x,t_0)$ are periodic which contradicts Lemma \ref{p&np}.
    This shows that $\Lambda(\a_0)$ is $\tau_3$-invariant.

\end{proof}

The two angles $\vt^L:=\vt^L(x,t_0)$ and $\vt^R:=\vt^R(x,t_0)$ are called the two \emph{characteristic angles} of the visual lamination $\Lambda_\R(\a_0)$.


\begin{lem}[characterization of $\Lambda_\R(\a_0)$]\label{characterization}
     For any equivalence class $E$ of $\Lambda_\R(\a_0)$, there exists an equivalence class $E_\H \subset E$ of $\L_\R(\H)$ such that either
     \begin{itemize}
         \item[(C1)] $E=E_\H$, or
         \item[(C2)] for any $\theta \in E\setminus E_\H$ and $\theta' \in E_\H$, we have $(\theta,\theta') \in \Lambda(\a_0)$.
     \end{itemize}

\end{lem}

\begin{proof}
    Let $E$ be an equivalence class of $\Lambda_\R(\a_0)$.
    Fix $\a \in \mathcal{R}$, consider the set $Z_\a(E):=\{ z_\theta(\a) : \theta \in E \}$.
    If $\# Z_\a(E) \leq 2$, then pick any point $z \in Z_\a(E)$ and let $E_\H$ be the set of angles of external rays landing at $z$.
    By the definition of $\Lambda_\R(\a_0)$, the conclusion certainly holds.

    In the following, we assume that $\# Z_\a(E) \geq 3$.
    By the definition of $\Lambda_\R(\a_0)$, we can find three points $z_0,z_1,z_2 \in Z_\a(E)$ such that $z_0,z_1$ are joined by the left-right internal ray $\RRR_\a^{L,R}(v_1,t_1)$, and $z_1,z_2$ are joined by the left-right internal ray $\RRR_\a^{L,R}(v_2,t_2)$. Let $E_\H$ be the angles of external rays landing at $z_1$.

    First, we claim that it can only happen in periodic case.
    By Lemma \ref{joined}, we may assume that $(v_1,t_1)$ and $(v_2,t_2)$ are both eventually mapped to $(x,t_0)$.
    Thus, we may assume that there exists $n$ and $k$ such that $\g^n(v_1,t_1)=(x,t_0)$, and $\g^k(v'_2,t'_2)=(x,t_0)$ where $(v_2',t_2'):=\g^n(v_2,t_2)$.
    Hence $z'_0:=f_\a^n(z_0)$ and $z_1':=f_\a^n(z_1)$ are joined by $\RRR^{L,R}_\a(x,t_0)$, and $z_1',z_2':=f_\a^n(z_2)$ are joined by $\RRR_\a^{L,R}(v_2',t_2')$.
    Suppose that we are in the non-periodic case, then by Lemma \ref{p&np}, $z'_0$ and $z'_1$ are both non-periodic.
    It follows that $f_\a^k(z'_0)=f_\a^k(z'_1)$, and $f_\a^k(z'_1)=z'_1$ or $f_\a^k(z'_1)=z'_0$.
    In both cases, we obtain that $z'_0$ or $z'_1$ is periodic.
    This is a contradiction.

     \begin{figure}[h]
  \begin{center}
    \includegraphics[width=0.65\linewidth]{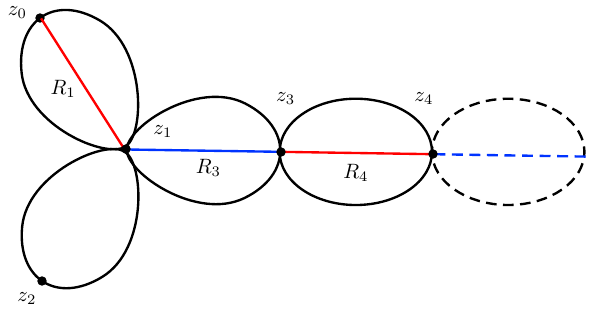}
    \caption{}\label{figring2}
 \end{center}
\end{figure}

    Therefore, we see that it must be the periodic case.
    Now we show (C2) holds by contradiction.
    Otherwise, we can find $\theta \in E \setminus E_\H$, $z_3,z_4$, such that $R_\a(\theta)$ lands at $z_4$, $z_1,z_3$ are joined by $R_3:=\RRR_\a^{L,R}(v_3,t_3)$ and $z_3,z_4$ are joined by $R_4:=\RRR_\a^{L,R}(v_4,t_4)$.
    By iteration, we may assume that $(v_1,t_1)=(x,t_0)$, and denote $R_1:=\RRR_\a^{L,R}(x,t_0)$.     
    In this case, one may verify that $z_0,z_1,z_3,z_4$ and $v_3,v_4$ are all periodic (see Figure \ref{figring2}).
    Let $m$ be the smallest integer such that $\g^m(v_1,t_1)=(v_3,t_3)$, and $\ell$ be the smallest integer such that $\g^\ell(v_3,t_3)=(v_4,t_4)$.
    It follows that $f_\a^m(R_1)=R_3$ and $f_\a^\ell(R_3)=R_4$.
    Then we must have $f_\a^m(\{ z_0 ,z_1\})=\{ z_1, z_3 \}$, and $f_\a^\ell(\{ z_1,z_3 \})=\{ z_3,z_4 \}$.
    By Lemma \ref{ring}, we see that  $f_\a^m(z_1)=z_1$ and $f_\a^m(z_0)=z_3$, $f_\a^\ell(z_3)=z_3$ and $f_\a^\ell(z_1)=z_4$.
    Therefore, we have $f_\a^{m+\ell}(z_0)=z_3$ and $f_\a^{m+\ell}(z_1)=z_4$.
    It follows that $f_\a^{m+\ell}$ maps $\Gamma_1:=R_1\cup R_3 \cup \{z_1\}$ to another curve $\Gamma_2$ with $\Gamma_1 \cap \Gamma_2=\{ z_3 \}$.
    $\Gamma_1 \cup \Gamma_2 \cup \{  z_3\}$ connects $z_0$ and $f_\a^{2(m+\ell)}(z_0)$.
    Now we follow the the same argument in the proof of Lemma \ref{ring}.
    Since $z_0$ is periodic, by applying $f_\a^{m+\ell}$ finitely many times, we obtain a circle $\bigcup_nf_\a^{n(m+\ell)}(\Gamma_1)$ in $K_\a$.   
    Hence this case is also ruled out.
    Thus, we conclude that (C2) holds.

\end{proof}


Although we will finally show that the visual lamination is the real lamination which would imply that it satisfies the axioms of lamination (R1)-(R5), to this end, we are ready to prove this weaker result in a combinatorial view point.

\begin{prop}[visual lamination is a lamination]\label{visuallamination}
    The visual lamination $\Lambda_\R(\a_0)$ is a $\tau_3$-invariant lamination, i.e. it satisfies the axioms of lamination.
\end{prop}

\begin{proof}
    We verify that $\Lambda_\R(\a_0)$ satisfies the the rules (R1) to (R5).
    
   {\bf Proof of (R2)}: 
    Let $E$ be an equivalence class of $\Lambda_\R(\a_0)$.
    Let $E_\H$ be the equivalence class of $\L_\R(\H)$ we obtained in Lemma \ref{characterization}, and $z_0$ be the common landing point of $R_\a(\theta)$ with $\theta \in E_\H$ for some fixed $\a \in \mathcal{R}$.
    If $E=E_\H$ holds, then $E$ is certainly finite.
    If not, by Lemma \ref{characterization}, for any $\theta \in  E\setminus  E_\H$, $z_\theta(\a)$ and $z_0$ are joined by some left-right internal ray.
    Notice that there can be finitely may left-right internal rays which has an end point $z_0$.
    Thus $Z_\a(E)$ is finite.
    Combining the fact that every point in $J_\a$ is the landing point of finitely may external rays, we see that $E$ is finite.
    Hence (R2) holds.

     {\bf Proof of (R3)}: 
    Let $E$ be an equivalence class of $\Lambda_\R(\a_0)$, we need to show that $\tau_3(E)$ is also an equivalence class.
    Since $E_\H$ is an equivalence class of $\L_\R(\H)$ and $\tau_3$ invariant, we see that $E'=\tau_3(E_\H)$ is an equivalence class of $\L_\R(\H)$ and hence contained in an equivalence class $E'$ of $\Lambda_\R(\a_0)$. 
    We will show that $\tau_3(E)=E'$.
    First, we show that $\tau_3(E) \subset E'$.
    In fact, by Lemma \ref{characterization}, for any $\theta \in E \setminus E_\H$ and $\theta' \in E_\H$, $R_\a(\theta)$ and $R_\a(\theta')$ are joined by some left-right internal ray $\RRR_\a^{L,R}(v,t)$.
    There are two cases.
    If $f_\a(\RRR_\a^{L,R}(v,t))$ is a smooth internal ray, then $R_\a(\tau_3(\theta))$ and $R_\a(\tau_3(\theta'))$ land at a common point.
    Hence $\tau_3(\theta) \in \tau_3(E_\H)\subset E'$.
    If $f_\a(\RRR_\a^{L,R}(v,t))$ is also a left-right internal ray, then $R_\a(\tau_3(\theta))$ and $R_\a(\tau_3(\theta'))$ are joined by this left-right internal ray.
    Hence we also have $\tau_3(\theta) \in E'$.
    Therefore, we see that $\tau_3(E) \subset E'$ holds.

    Now we show the converse.
    For any $\theta' \in E' \setminus \tau_3(E_\H)$, for any $\theta'' \in E_\H$, $R_\a(\theta')$ and $R_\a(\tau_3(\theta''))$ are joined by some left-right internal ray $\RRR_\a^{L,R}(v',t')$.
    It follows that there exists $(v,t)$ with $\g(v,t)=(v',t')$, and $\theta$ with $\tau_3(\theta)=\theta'$ such that $R_\a(\theta)$ and $R_\a(\theta'')$ are joined by the left-right internal ray $\RRR_\a^{L,R}(v,t)$.
    It follows that $\theta \in E$ which implies that $\theta' \in \tau_3(E)$.
    Hence $E' \subset \tau_3(E)$.
    Therefore, (R3) holds.

 {\bf Proof of (R4)}: 
    Now we show that (R4) holds, that is, $\tau_3|_E$ is consecutive-preserving for any equivalence class of $\Lambda_\R(\a_0)$.
    Fix $\a \in \mathcal{R}$, denote the landing point of external rays with angles in $E_\H$ by $z_0$.
    Let $(\theta,\theta')_+$ be any complementary arc of $E$.
    There are two cases.
    If $R_\a(\theta)$ and $R_\a(\theta')$ land at a common point $w$, i.e. $(\theta,\theta')_+$ is a complementary arc of some equivalence class of $\L_\R(\H)$, then $(\tau_3(\theta),\tau_3(\theta'))_+$ is a complementary arc of some equivalence class of $\L_\R(\H)$.
    We should further claim that it is a complementary arc of $\tau_3(E)$ of $\Lambda_\R(\a_0)$.
    For otherwise, there exists another $\vt' \in \tau_3(E)\cap (\tau_3(\theta),\tau_3(\theta'))_+$.
    Then there exists $(v',t')\in \M'$ such that $R_\a(\vt')$ and $R_\a(\tau_3(\theta))$ are joined by $\RRR_\a^{L,R}(v',t')$.
    Since $f_\a$ is locally orientational preserving, we can find $(v,t) \in \M'$ and $\vt \in (\theta,\theta')_+ $ with $\g(v,t)=(v',t')$ and $\tau_3(\vt)=\vt'$ such that $R_\a(\theta)$ and $R_\a(\vt)$ are joined by $\RRR_\a^{L,R}(v,t)$. 
    Hence $\vt \in (\theta,\theta')_+ \cap E$.
    which contradicts the fact that $(\theta,\theta')_+$ is a complementary arc of $E$.

    Now we assume that $R_\a(\theta)$ and $R_\a(\theta')$ land at distinct points. 
    We claim that one of them must be $z_0$.
    For otherwise, by Lemma \ref{characterization}, we may assume that $R_\a(\theta)$ and $z_0$ are joined by $\RRR_\a^{L,R}(v,t)$, and $R_\a(\theta')$ and $z_0$ are joined by $\RRR_\a^{L,R}(v',t')$.
    Then we can find $\vt \in (\theta,\theta')_+ $ such that $R_\a(\vt)$ lands at $z_0$.
    Therefore, we have $\vt \in E$.
    This contradicts the fact that $(\theta,\theta')_+$ is a complementary arc of $E$.
    Hence we may assume that $R_\a(\theta)$ land at $z_0$, and $R_\a(\theta')$ and $z_0$ are joined by $\RRR_\a^{L,R}(v',t')$.
    There are two cases.
    If $\tau_3(\theta)\neq \tau_3(\theta')$, then $R_\a(\tau_3(\theta))$ land at $f_\a(z_0)$, and $R_\a(\tau_3(\theta'))$ and $f_\a(z_0)$ are joined by $\RRR_\a^{L,R}(\g(v',t'))$.
    If $(\tau_3(\theta),\tau_3(\theta'))_+$ is not a complementary arc of $\tau_3(E)$, then there exists $\vt' \in \tau_3(E) \cap (\tau_3(\theta),\tau_3(\theta'))_+$.
    Then $R_\a(\vt')$ land at $f_\a(z_0)$ or is joined to $f(z_0)$ by some left-right internal ray.
    Then by a similar argument above, there exists $\vt \in (\theta,\theta')_+$ such that $R_\a(\vt)$ land at $z_0$ or is joined to $z_0$ by some left-right internal ray.
    It follows that $\vt \in E$ which contradicts the fact that $(\theta,\theta')_+$ is a complementary arc of $E$.  
    Now we deal with the case that $\tau_3(\theta) = \tau_3(\theta')$.
    In this case, $R_\a(\theta)$ and $R_\a(\theta')$ must be joined by $\RRR_\a^{L,R}(x,t_0)$, and it must be the non-periodic case.
    One may verify that $\tau_3(E)=\{ \tau_3(\theta) \}$ is a singleton.
    Hence $(\tau_3(\theta),\tau_3(\theta))_+$ is the complementary arc of $\tau_3(E)$.

 {\bf Proof of (R5)}: 
    Let $E$ be an equivalence class of $\Lambda_\R(\a_0)$, and $\Gamma_\a(E)$ be the union of all external rays landing at $Z_\a(E)$ and all left-right internal rays joining them.
    This is a connected set.
    Let $E'$ be any other equivalence class of $\Lambda_\R(\a_0)$.
    Since $Z_\a(E) \cap Z_\a(E')=\emptyset$, the only possible intersections between $\Gamma_\a(E)$ and $\Gamma_\a(E')$ are the intersection of left and right internal rays.
    By Corollary (I2) of Corollary \ref{intersection}, $\Gamma_\a(E')$ must belong to the closure of some connected component of $\C \setminus \Gamma_\a(E)$.
    Therefore, $E'$ must belong to a connected component of $\RZ \setminus E$.
    Thus, (R5) holds.

 {\bf Proof of (R1)}: 
    It remains to show that (R1) holds, that is, $\Lambda_\R(\a_0)$ is a closed set of $\RZ \times \RZ$.
    Let $\{ (\theta_n,\theta'_n) \}$ be a sequence in $\Lambda_\R(\a_0)$ with $(\theta_n,\theta_n')$ tending to $(\theta,\theta')$ outside the diagonal line as $n\to \infty$.
    We need to show that $(\theta,\theta') \in \Lambda_\R(\a_0)$.
    Let $E^n$ be the equivalence class containing $\theta_n,\theta_n'$.
    We may assume that $E^n$'s are all distinct and hence $\theta_n$'s and $\theta'_n$'s are all different.
    If there is a subsequence of $\{ (\theta_n,\theta'_n) \} \subset \L_\R(\H)$, then $(\theta,\theta') \in \L_\R(\H)$ since $\L_\R(\H)$ is closed.  
    In the following, we may assume that $(\theta_n,\theta'_n) \notin \L_\R(\H)$ for all $n$.
    By the unlinked property (R5), for any $n$, $\theta_{n+1},\theta_{n+1}'$ either belong to $(\theta_n,\theta_n')_+$ or $(\theta_n',\theta_n)_+$.
    Without loss of generality, we may assume that $(\theta_{n+1},\theta'_{n+1})_+ \subset (\theta_{n},\theta'_{n})_+$ for all $n$.
    
    We may further assume that for every $n$, $R_\a(\theta_n)$ and $R_\a(\theta'_n)$ can be joined by a single left-right internal ray.
    In fact, if for some $n$, $R_\a(\theta_n)$ and $R_\a'(\theta_n')$ are joined by two left-right internal rays, i.e. there exists $\theta_n'' \in E^n_\H$, and left-right internal rays $R_1,R_2$ such that $\theta_n, \theta_n''$ are joined by $R_1$, and $\theta_n',\theta_n''$ are joined by $R_2$. 
    By (R5), $(\theta_{n+1},\theta_{n+1}')_+$ is either contained in $(\theta_n,\theta_n'')_+$ or $(\theta_n'',\theta_n')_+$.
    Therefore, by replacing one of $\theta_n,\theta_n'$ by $\theta_n''$ for every $n$, we obtain a new nest sequence $(\theta_n,\theta_n')_+$ which satisfies that $R_\a(\theta_n)$ and $R_\a(\theta'_n)$ are joined by a single left-right internal ray $\RRR_\a^{L,R}(v_n,t_n)$ for all $n$.

    Let $B_n:=B_\a(v_n)$.
    There are two cases.
    If $B_n=B_\a(v)$ for $n$ large, then we show that $R_\a(\theta)$ and $R_\a(\theta')$ must land on a common point on $\partial B_\a(v)$.
    We assume that $B_n=B_\a(v)$ for all $n$.
    Let $U_n$ be the connected component of $\C \setminus (\RRR_\a^{L,R}(v_n,t_n) \cup \overline{R_\a(\theta)} \cup \overline{R_\a(\theta')})$ which contains external rays with angles in $(\theta_n,\theta_n')_+$.
    Let $X_n:=\overline {B_n} \cap U_n$.
    Since $(\theta_{n+1},\theta_{n+1}')_+\subset (\theta_n,\theta_n')_+$, we know that $R_\a(\theta_{n+1})$ and $R_\a(\theta_{n+1}')$ are contained in $U_n$.
    It follows that $U_{n+1} \subset U_n$ and $X_{n+1}$ is a proper subset of $X_n$.
    Hence $R_\a(\theta)$ and $R_\a(\theta')$ must land at points in $X:=\bigcap_n X_n$.
    We claim that $X$ must be contained in some impression of the puzzle system.
    In fact, for $n$ large $X_n$ does not contain any $m$-th preimage of $c(\a)$ for $m\leq n$.
    Therefore, for each $m\geq 1$, we can find $n$ large such that $X_n \cap G_m^\a = \emptyset$.
    Hence $X_n$ is contained in some puzzle piece $P_m^\a$ of depth $m$.
    It follows that $X \subset \bigcap_m P_m^\a$.
    By Lemma \ref{singleton}, we see that $X$ is a singleton.
    Therefore, $R_\a(\theta)$ and $R_\a(\theta')$ land at a common point, which implies that $(\theta,\theta') \in \L_\R(\H) \subset \Lambda_\R(\a_0)$.    
    
    Now, we consider the complementary case.
    We may assume that $\{ B_n \}$ are all different.
    Fix $n$, by Theorem \ref{roeschyin}, $B_{n+1}$ belongs to a limb with respect to $B_n$.
    It follows that there exists a pair of external rays $R_\a(\vt_n)$, $R_\a(\vt_n')$ together with their common landing point separate $B_{n+1}$ away from $B_n$.
    Thus they separate $R_\a(\theta_{n+1})\cup R_\a(\theta_{n+1}')$ away from $R_\a(\theta_n) \cup R_\a(\theta_n')$.
    It follows that $(\theta_{n+1},\theta'_{n+1})_+\subset (\vt_n,\vt_n')_+ \subset (\theta_n,\theta_n')_+$.
    By the construction of $\vt_n,\vt_n'$, we have $(\vt_n,\vt_n') \to (\theta,\theta')$  as $n\to \infty$.
    Therefore, we have $(\theta,\theta') \in \L_\R(\H)$ since $(\vt_n,\vt_n') \in \L_\R(\H)$.
    This proves that (R1) holds.
    


\end{proof}

\subsection{Puzzles and Impressions for Boundary Maps}\label{ssec4.3}


In this part, we apply the puzzle system $\mathrm{P}_{\a_0}(\varTheta,T,s,s’)$ to the boundary map $\a_0$.
It would reveal the connection between the lamination of $\a_0$ and the lamination of $\a \in \mathcal{R}=\mathcal{R}_\H(t_0)$.

The following result is a reformulation of \cite[Theorem 3.3]{wang2023}.

\begin{thm}\label{wang}
     There exists $T=T(\a_0)$ depending on $\a_0$ such that the puzzle system $\mathrm{P}_{\a_0}(\Theta,T,s,s’)$ with $\Theta=\emptyset$ has the following cases.

  {\bf non-renormaliziable case:} If every impression is a singleton, then $(x,t_{0})$ is non-periodic under $\g$.

     {\bf renormaliziable case:} If there are non-singleton impressions, then there exists a renormaliziable period $q$ such that restriction of $f_{\a_0}^q$ on $I_{\a_0}(-c(\a_0))$ is conjugated to a parabolic quadratic polynomial $p_c$.
     An impression is not a singleton if and only if it is eventually mapped to $ I_{\a_0}(-c(\a_0)) $.
     There are following two cases.
     \begin{itemize}
         \item {\bf parabolic case:}  $(x,t_{0})$ is periodic under $\g$ and $\psi_{\a_0}^x(e^{2\pi i t_{0}})$ is a parabolic periodic point. i.e. $p_c$ is parabolic.
         \item {\bf PCF case:} $(x,t_{0})$ is non-periodic and $-c({\a_0})$ is eventually mapped to a repelling periodic point, i.e. $p_c$ is post critically finite.
    \end{itemize}

\end{thm}

In the following, we chose $\varTheta$ carefully so that the  puzzle system $\mathrm{P}_{\a_0}(\varTheta,T,s,s')$ have the following nicer properties.

\begin{prop}[nice puzzle system]\label{wang2}
    There exist $\varTheta:=\varTheta(\a_0)$ such that the puzzle system $\mathrm{P}_{\a_0}(\varTheta,T,s,s')$ satisfies the following.

    \begin{itemize}
        \item  {\bf periodic case:} If $(x,t_{0})$ is periodic under $\g$,
        then there exists $k$ such that the restriction of $f_{\a_0}^k$ on $I_{\a_0}(-c(\a_0))$ is conjugated to a quadratic polynomial with a parabolic fixed point 
        $\psi_{\a_0}^x(e^{2\pi i t_{0}})$.
       \item   {\bf non-periodic case:} If $(x,t_{0})$ is not periodic under $\g$, then every impression is a singleton.
    \end{itemize}

\end{prop}

\begin{proof}
    We only need to consider the renormaliziable case.
    Suppose that $\a_0$ is PCF renormaliziable.
    Since the renormalized quadratic polynomial $p_c$ is post-critically finite, according to the study of quadratic polynomials, there exist $\theta,\theta'$ such that $R_c(\theta_1)$ and $R_c(\theta_2)$ lands at a common repelling periodic point which is not on the critical orbit.
    Moreover, $\bigcup_np_c^{-n}(\overline{R_c(\theta_1)} \cup \overline{R_c(\theta_2)}) $ divides the Julia set of $p_c$ into singletons.
    For cubic polynomial $f_{\a_0}$, since $I_{\a_0}(-c(\a_0))$ is homeomorphic to the filled-in Julia set of $p_c$,  we can find $\theta'_1,\theta'_2$ such that $R_{\a_0}(\theta'_1)$ and $R_{\a_0}(\theta_2')$ landing at a common repelling periodic point which is not on the orbit of $-c(\a_0)$, and $I_{\a_0}(-c(\a_0))$ can be divided into singleton sub-impressions by $\bigcup_n f_{\a_0}^{-n}(\overline{ R_{\a_0}(\theta'_1)} \cup \overline{R_{\a_0}(\theta'_2) })$.
    Let $\varTheta:=\{ \theta_1',\theta_2' \}$.
    It follows that the impression of every point in $J_{\a_0}$ is a singleton with respect to  $\mathrm{P}_{\a_0}(\varTheta,T,s,s')$.

    Now we consider the parabolic renormaliable case.
    By Theorem \ref{wang}, we see that there is a renormalization period $q$ and the renormalized quadratic polynomial $p_c$ has a unique parabolic orbit.
    By the orbit separation Lemma (\cite[Lemma 3.7]{Schleicher1997RationalPR}, see also \cite[Lemma 3.6]{wang2023}), there exists $\theta,\theta'$ such that $R_c(\theta_1)$ and $R_c(\theta_2)$ lands at a common repelling periodic point such that the parabolic points are in distinct components of $\bigcup_np_c^{-n}(\overline{R_c(\theta_1)} \cup \overline{R_c(\theta_2)}) $.
    For the same reason above, we can find $\theta'_1,\theta'_2$ such that $R_{\a_0}(\theta'_1)$ and $R_{\a_0}(\theta_2')$ landing at a common repelling periodic point such that the parabolic orbit of $f_{\a_0}$ are separated by $\bigcup_n f_{\a_0}^{-n}(\overline{ R_{\a_0}(\theta'_1)} \cup \overline{R_{\a_0}(\theta'_2) })$.
    Hence when consider the puzzle system  $\mathrm{P}_{\a_0}(\varTheta,T,s,s')$ where $\varTheta:=\{ \theta_1',\theta_2' \}$, every impression must contain at most one parabolic point.
    Therefore, there exists $k$ which is a multiple of $q$ such that $f_{\a_0}^k$ on the new impression of $-c(\a_0)$ is conjugated to a quadratic polynomial with a parabolic fixed point.
    
\end{proof}

From now on, we fix the puzzle system $\mathrm{P}_{\a_0}(\varTheta,T,s,s')$ which satisfies Proposition \ref{wang2}.
The following lemma is a variation of \cite[Lemma 3.1]{wang2023} which shows that the graphs $G_n^\a$ and $\Gamma_n^\a$ of $\mathrm{P}_{\a_0}(\varTheta,T,s,s')$ admit a holomorphic motions.

\begin{lem}[stability of graphs]\label{hmpuzzle}
    Let $\a_0 \in \S\cap \mathcal{C}$ with $-c(\a_0)\notin B_{\a_0}$.
    Fix $\varTheta,T,s,s'$, let $\{ G_n^\a \}$ be the sequence of graphs of the puzzle system $\mathrm{P}_{\a_0}(\varTheta,T,s,s’)$.
    Fix $n \geq 0$, there exists a sufficiently small neighborhood $ \mathcal{P}_n$ of $\a_0$, and a holomorphic motion $h_n: \mathcal{P}_n \times G_n^{\a_0} \to \C$ such that the following holds.
    \begin{enumerate}
        \item For $\theta$ with $R_{\a_0}(\theta)\cap X_n^{\a_0}\subset G_n^{\a_0}$,  we have $h_n(\a,R_{\a_0}(\theta)\cap X_n^{\a_0}) = R_\a(\theta) \cap X_n^\a$.
        \item For $(v,t) \in \O \times \RZ$ with $s_{\a_0}^v(t)=1$ and $R_{\a_0}^v(t) \cap X_n^{\a_0} \subset G_n^{\a_0}$, we have $s_\a^v(t)=1$ and $h_n(\a,R_{\a_0}^v(t)\cap X_n^{\a_0}) = R_\a^v(t) \cap X_n^\a$.
    \end{enumerate}

\end{lem}

Using Lemma \ref{hmpuzzle}, we deduce the following important lemma which asserts that the landing point of external rays with angles in an equivalence class of $\Lambda_\R(\a_0)$ belong to the same impression of the puzzle system $\mathrm{P}_{\a_0}(\varTheta,T,s,s')$.

\begin{lem}[common impression]\label{impression}
    For $(\theta,\theta') \in \Lambda_\R(\a_0)$, $z_\theta(\a_0)$ and $z_{\theta'}(\a_0)$ are in the same impression of puzzle system $\mathrm{P}_{\a_0}(\varTheta,T,s,s')$.
    
    If a $\i$-turning internal ray $\RRR_\a^{\i}(v,t)$ contains an iterative preimage $\omega(\a)$ of $-c(\a)$ such that $z_\a^{\i}(v,t)=z_\theta(\a)$ for $\a \in \mathcal{R}$, then we have 
    \begin{equation}\label{puzzle=}
        I_{\a_0}(\psi_{\a_0}^v(e^{2\pi i t}))=I_{\a_0}(z_\theta(\a_0))=I_{\a_0}(\omega(\a_0))
    \end{equation}
     where $\omega(\a_0)=\lim_{\a \to \a_0}\omega(\a)$ is also an iterative preimage of $-c(\a_0)$. 
\end{lem}

\begin{proof} 
    Suppose that $\RRR_\a^{\i}(v,t)$ contains an iterative preimage $\omega(\a)$ of $-c(\a)$ such that $z_\a^{\i}(v,t)=z_\theta(\a)$ for $\a \in \mathcal{R}$.
    Fix any $n$, for $\a \in \mathcal{R}$, since $G_\a^n$ consists of equipotential lines and smooth internal rays and external rays, by (I1) of Lemma \ref{intersection}, the three points $z_\a^{\i}(v,t)$ and $w(\a)$ cannot be separated by $G_n^\a$. 
    Hence for any $\a \in \mathcal{R}$,
    \begin{equation}\label{puzzle==}
            P_n^\a ( z_\a^{\i}(v,t) ) = P_n^\a(z_\theta(\a)) = P_n^\a(\omega(\a)), \quad \forall n \geq 1.
     \end{equation}

    Now, we show that $I_{\a_0}(\psi_{\a_0}^v(e^{2\pi i t}))=I_{\a_0}(z_\theta(\a_0))$ by contradiction.
    We assume the contrary that $P_n^{\a_0}(\psi_{\a_0}^v(e^{2\pi i t})) \neq  P_n^{\a_0}(z_{\theta}(\a_0)) $ for some $n\geq 1$.  
    By the continuity of the B\"ottcher map $\psi_{\a_0}$ and $\psi_{\a_0}^v$, for $r>1$ and $r'<1$ sufficiently close to $1$, we have $P_n^{\a_0}(\psi_{\a_0}^v(r'e^{2\pi i t})) \neq  P_n^{\a_0}(\psi_{\a_0}(re^{2\pi i \theta})) $.    
    By Lemma \ref{hmpuzzle}, and the continuity of B\"ottcher map of parameter, we see that for $\a $ close to $\a_0$, $P_n^\a(\psi_\a^v(r'e^{2\pi i t})) \neq P_n^\a (\psi_\a(re^{2\pi i \theta})).$
    It follows that $ P_n^\a(z_\a^{\i}(v,t)) \neq P_n^\a (z_{\theta}(\a))$.
    This contradicts \eqref{puzzle==}.
    Using a similar argument, we see that $I_{\a_0}(z_\theta(\a_0)) =  I_{\a_0}(\omega(\a_0)) $.
    Hence \eqref{puzzle=} holds.

    Suppose that $(\theta,\theta') \in \Lambda(\a_0)$, we need to show that $I_{\a_0}(z_\theta(\a_0))=I_{\a_0}(z_{\theta'}(\a_0))$.
    By Lemma \ref{characterization}, it suffices to consider two case that $(\theta,\theta') \in \L_\R(\H)$ and $(\theta,\theta')\in \Lambda_\R(\a_0)$.
    If we are in the former case, then we have $P_n^\a(z_\theta(\a))=P_n^\a(z_{\theta'}(\a))$ for every $n$ and $\a \in \mathcal{R}$.
    Then use the same argument above, we deduce that $I_{\a_0}(z_\theta(\a_0))=I_{\a_0}(z_{\theta'}(\a_0))$.
    If we are in the later case, then for $\a \in \mathcal{R}$, $z_\theta(\a)$ and $z_{\theta'}(\a)$ are joined by some left-right internal ray which contains an iterative preimage $\omega(\a)$ of $-c(\a)$.
    Then by the above discussion, we see that $I_{\a_0}(z_\theta(\a_0))=I_{\a_0}(\omega(\a_0))=I_{\a_0}(z_{\theta'}(\a_0))$.

\end{proof}

Combining the non-periodic case in Proposition \ref{wang2} and Lemma \ref{impression}, we deduce the following.

\begin{cor}[non-periodic case]\label{nonper}
    If $(x,t_0)$ is non-periodic under $\g$, then $\Lambda_\R(\a_0) \subset \L_\R(\a_0)$.
    In particular, we have $\L_\R(\H) \subset \L_\R(\a_0)$.
\end{cor}

For periodic case, since the impression of any iterative preimage of the critical point is not a singleton, we need some more precise characterization of the landing location of the external rays to deduce the counterpart of Corollary \ref{nonper}.

\begin{lem}[landing on the parabolic component]\label{periodicimpression}
    For any $\i \in \{ L,R \}^{\mathbb N}$, the external ray $R_{\a_0}(\vt^{\i}(x,t_0))$ land at a point on the boundary of the parabolic component containing $-c(\a_0)$.
\end{lem}

\begin{proof}
    From Lemma \ref{impression}, since the $\i$-turning internal ray $\RRR_\a^{\i}(x,t_0)$ contains the critical point $-c(\a)$, we know that $R_{\a_0}(\vt^{\i}(x,t_0))$ land at a point in $I_{\a_0}(-c(\a_0))$.
    
    By Proposition \ref{wang2}, $\psi_{\a_0}^x(e^{2\pi i t_0})$ is a parabolic fixed point of $f_\a^k$.
    Without loss of generality, we assume that $k=1$, i.e. $\psi_{\a_0}^x(e^{2\pi i t_0})$ is a parabolic fixed point of $f_{\a_0}$.
    It follows that $\psi_{\a_0}^x(e^{2\pi i t_0})$ must be the common point of $p$ super-attracting Fatou component and $p$ parabolic component.
    By Lemma \ref{stableparabolic}, for $\a \in \H$ and sufficiently close to $\a_0$, the parabolic point $\psi_{\a_0}^x(e^{2\pi i t_0})$ breaks into a repelling fixed point and a $p$ repelling periodic cycle.
    Let $z(\a)$ denote the fixed point, and $R_\a(\theta_1)$ be an external ray which lands at $z(\a)$.
    By Proposition \ref{keylemma}, $R_{\a_0}(\theta_1)$ lands at $ \psi_{\a_0}^x(e^{2\pi i t_0}) $ since $z(\a)$ tends to $ \psi_{\a_0}^x(e^{2\pi i t_0}) $ as $\a $ tends to $\a_0$.
    It follows that $\theta_1$ must have period $p$ under $\tau_3$.
    Let $\{ \theta_1,\theta_2,\dots,\theta_p \}$ be its orbit.
    It follows that for $\a \in \H$ and near $\a_0$, $R_\a(\theta_j)$ lands at $z(\a)$, $R_{\a_0}(\theta_j)$ lands at $\psi_{\a_0}^x(e^{2\pi i t_0})$ for $1\leq j\leq p$.
    They separates the $p$ super-attracting periodic components.
    We assume that $B_\a(x)$ are in the sector $S_\a(\theta_1,\theta_2)_+$. 
    For $\a \in \mathcal{R}$, since the orbit of $\RRR_\a^{\i}(x,t_0)$ under $f_\a^p$ always land at a point on $\partial B_\a(x)$.
    If follows that the orbit of $\vt^{\i}(x,t_0)$ under $\tau_3^p$ stays in $(\theta_1,\theta_2)_+$. 
    Then the orbit of $R_{\a_0}(\vt^{\i}(x,t_0))$ under $f_{\a_0}^p$ must stay in $S_{\a_0}(\theta_1,\theta_2)_+$.
    Notice that the points in $I_{\a_0}(-c(\a_0))$ whose orbit under $f_{\a_0}^p$ stays in $S_{\a_0}(\theta_1,\theta_2)_+$ belong to a periodic parabolic component.
    This component must be the one containing $-c(\a_0)$.    
    Therefore, we conclude that $R_{\a_0}(\vt^{\i}(x,t_0))$ must land at a point on the boundary of the parabolic component containing $-c(\a_0)$.

\end{proof}

\begin{cor}[characteristic angles]\label{coland}
   The external rays $R_{\a_0}(\vt^L)$ and $R_{\a_0}(\vt^R)$ of the two characteristic angles $\vt^L,\vt^R$ both land at the point $\psi_{\a_0}^x(e^{2\pi i t_0})$.    
\end{cor}

\begin{proof}
    By Lemma \ref{periodicimpression}, we know that $R_{\a_0}(\vt^L)$ and $R_{\a_0}(\vt^R)$ both land at points on the boundary of the parabolic component containing $-c(\a_0)$.
    On the other hand, since $\RRR_\a^L(x,t_0)$ and $\RRR_\a^R(x,t_0)$ are both fixed under $f_\a^p$ for $\a \in \mathcal{R}$, we see that $R_{\a_0}(\vt^L)$ and $R_{\a_0}(\vt^R)$ must both land at a fixed point of $f_{\a_0}^p$.
    Notice that $\psi_{\a_0}^x(e^{2\pi i t_0})$ is the unique fixed point of $f_{\a_0}^p$ on the boundary of the parabolic component containing $-c(\a_0)$.
    It follows that $R_{\a_0}(\vt^L)$ and $R_{\a_0}(\vt^R)$ both land at  $\psi_{\a_0}^x(e^{2\pi i t_0})$.    
    
\end{proof}

\begin{rmk}\label{parabolicperturb}
    From Lemma \ref{periodicimpression}, for the $\a \in \mathcal{R}$ and close to $\a_0$, the parabolic fixed point of $f_{\a_0}$ breaks into a repelling fixed point $z(\a)$ of $f_\a$, and a $p$ periodic orbit.
    One of $R_\a(\vt^L)$ and $R_\a(\vt^R)$ must land at $z(\a)$ of $f_\a$, and the other lands at a point in the $p$ periodic orbit.
    Let $E$ be the equivalence class of $\Lambda_\R(\a_0)$ containing $\{ \vt^L,\vt^R \}$.
    Then $E_\H$ is the set of angles of all external rays landing at $z(\a)$.
\end{rmk}

\subsection{Relations of Laminations of Boundary and Interior}

In this part, we prove that the real lamination $\L_\R(\H)$ is contained in the real lamination $\L_\R(\a_0)$ of the boundary map $\a_0$.
The non-periodic case is implied by Corollary \ref{nonper}.

In the following, we focus on the periodic case.
To do this, since $f_{\a_0}$ have no critical point with infinite orbit lying on the boundary of any Fatou component, by Theorem \ref{kiwicontain}, we only need to show that $\L_\Q(\H) \subset \L_\Q(\a_0)$.

Recall that $\Q_p$ denote the set of all $\theta \in \RZ$ which is periodic under $\tau_3$.
According to Corollary \ref{Qp}, we deduce that $\L_{\Q_p}(\H) \subset \L_{\Q_p}(\a_0)$.
Thus, we only need to deal with the rational lamination of non-periodic angles.
We introduce the \emph{critical leaf}.


\begin{lem}[critical leaf]\label{fact}
     If we are in the periodic case, then there exist $L^\pm \in \RZ \times \RZ$ such that the following holds.
     \begin{enumerate}
         \item For $\a \in \H \cup \{ \a_0\}$ and $\theta \in L^\pm$, $R_\a(\theta)$ lands at a repelling preperiodic point on the boundary of the Fatou component containing $\pm c(\a)$.
         \item For any $\theta \in L^-\cup L^+$, $\{ \theta \}$ forms a singleton equivalence class of $\lambda_\R(\a_0)$.
         \item $\Delta L^-=\Delta L^+= 1/3$.
    \end{enumerate}
\end{lem}

\begin{proof}
    Pick $t_1 \in \Q/\Z$ such that the internal ray $R_{\a_0}^c(t_1)$ lands at a repelling periodic point $z_1$ which is the landing point of a unique external ray $R_{\a_0}(\theta_1^+)$.
    Hence $R_{\a_0}^c(t_1+1/2)$ also lands at a repelling preperiodic point $z_2$ which is the landing point of a unique external ray $R_{\a_0}(\theta_2^+)$.      
    By Lemma \ref{stablerepelling}, for $\a \in \H$, $R_\a(\theta_j^+)$ also lands at a point on $\partial B_\a(c)$.
    Define $L^+:=\{ \theta_1^+,\theta_2^+ \}$, then we have $\Delta L^+=1/3$.

    Pick a sequence $\i=(\iota_1,\iota_2,\dots) \in \{ L,R \}^{\mathbb N}$ which is periodic under the shift map, and $\i'=(\iota'_1,\iota_2',\dots)$ be a sequence such that $\iota_1\neq \iota_1'$ and $\iota_n=\iota'_n$ for every $n \geq 2$.
    Define $\theta_1^-:=\vt^{\i}(x,t_0)$ and $\theta_2^-:=\vt^{\i'}(x,t_0)$.
    Since $\tau_3(\theta_1^-)=\tau_3(\theta_2^-)$, (3) holds.
    By Lemma \ref{periodicimpression}, $R_{\a_0}(\theta_1^-)$ land at a repelling periodic point on the boundary of the parabolic component containing $-c(\a_0)$, and $R_{\a_0}(\theta_2^-)$ land at a repelling preperiodic point on the boundary.     
    By picking $\i$ with sufficiently large periodic, we may assume that (2) holds.

\end{proof}

For each $\a \in \mathcal{R}\cup \{ \a_0 \}$, we define the \emph{cut ray} of the critical leaf $L^+$ by 
$$
    C_\a(L^+):=\overline{R_\a(\theta_1^+)} \cup \overline{R_\a(\theta_2^+)} \cup \overline{R_\a^c(t_1)} \cup \overline{R_\a^c(t_1+1/2)}.
$$
For $\a \in \mathcal{R}$, define the cut ray of the critical leaf $L^+$ by 
$$
    C_\a(L^-):= \overline{R_\a(\theta_1^-)} \cup \overline{R_\a(\theta_2^-)} \cup \overline{\RRR_\a^{\i}(x,t_0)} \cup \overline{\RRR_\a^{\i'}(x,t_0)}.
$$
For $\a_0$, we define
$$
    C_{\a_0}(L^-):= \overline{R_{\a_0}(\theta_1^-)}\cup \overline{R_{\a_0}(\theta_2^-)} \cup \overline{R_{\a_0}(\theta_1^+)} \cup \overline{R_{\a_0}(\theta_2^+)} \cup B_{\a_0}(c) \cup U_{\a_0}(-c)
$$
where $U_{\a_0}(-c)$ denote the parabolic basin which contains $-c(\a_0)$.



\begin{thm}\label{laminationsubset}
    Let $\H$ be a bounded hyperbolic component of type \A, \B, or \CCC, and $\a_0 \in \partial \H \cap \S$, then we have 
    $\L_\R(\H) \subsetneq \L_\R(\a_0)$.
\end{thm}

\begin{proof}
    As we already noticed, it suffices to show that every $(\theta,\theta') \in \L_\Q(\H)\setminus \L_{\Q_p}(\H)$ is contained in $\L_\Q(\a_0)$ for the periodic case.

     Notice that $\theta$ and $\theta'$ must both be non-periodic.
     Let $n$ be an integer such that $\theta_n:=\tau_3^n(\theta)$ and $\theta_n':=\tau_3^n(\theta')$ are both periodic.
     Since $(\theta_n,\theta_n') \in \L_\R(\H) $, by the above discussion, $R_{\a_0}(\theta_n)$ and $R_{\a_0}(\theta'_n)$ land at common periodic point $z_0$.
     There are two cases.
     The case that $z_0$ is repelling is easier.
     We solve it by contradiction.
     Suppose that $R_{\a_0}(\theta)$ and $R_{\a_0}(\theta')$ does not land at a common point.
     Since $R_{\a_0}(\theta_n)$ and $R_{\a_0}(\theta'_n)$ both land at $z_0$.
     Thus there exists $\theta''$ such that $R_{\a_0}(\theta)$ and $R_{\a_0}(\theta'')$ land at a common repelling preperiodic point with $\tau_3^n(\theta'')=\theta'_n$.
     By Lemma \ref{stablerepelling}, for $\a$ sufficiently close to $\a_0$, the external rays $R_\a(\theta)$ and $R_\a(\theta')$ also land at a common point.
     Therefore, for $\a \in \H$ which is sufficiently to $\a_0$, $R_\a(\theta')$ and $R_\a(\theta'')$ both land at a common point.
     Notice that these two external rays are both mapped to $R_\a(\theta_n)$ under $f_\a^n$.
     This contradicts the fact that $f_\a^n$ is conformal near any point in the Julia set.
     Hence we conclude that $(\theta,\theta') \in \L_\Q(\a_0)$.

     The remaining case is that $z_0$ is a parabolic periodic point.
     By the result we obtained for periodic case and induction, we may assume that $(\tau_3(\theta),\tau_3(\theta')) \in \L_\Q(\a_0) $.
     By (1) and (2) of Lemma \ref{fact}, $\overline{R_\a(\theta)}$ and $\overline{R_\a(\theta')}$ must belong to the same connected component of $\C \setminus ( C_\a(L^-) \cup C_\a(L^+) )$.
     Therefore, $\{\theta,\theta'\}$ must belong to an connected component of $\RZ \setminus (L^-\cup L^+)$.
     Since $\Delta L^\pm=1/3$, $\theta+1/3$ and $\theta-1/3$ must belong to the other two connected component of $\RZ \setminus (L^-\cup L^+)$.
     If $(\theta,\theta') \notin \L_\Q(\a_0)$, either $(\theta+1/3,\theta') \in \L_\Q(\a_0)$ or either $(\theta-1/3,\theta') \in \L_\Q(\a_0)$.
     Without loss of generality, we assume that $(\theta+1/3,\theta') \in \L_\Q(\a_0)$.
     By Lemma \ref{fact}, $\overline{R_{\a_0}(\theta+1/3)}$ and $\overline{R_{\a_0}(\theta')}$ must belong to the same connected component of $\C \setminus (C_{\a_0}(L^-) \cup C_{\a_0}(L^+) )$.
     This is impossible since $\theta+1/3$ and $\theta'$ are in distinct connected components of $\RZ \setminus (L^-\cup L^+)$.
     Therefore, we conclude that $(\theta,\theta') \in \L_\Q(\a_0)$.

    \end{proof}

    \subsection{Visual Laminations Coincides with Real Laminations}

     In this part, we show that the visual lamination $\Lambda_\R(\a_0)$ coincides with the real lamination $\L_\R(\a_0)$ in all cases.
     First, we deal with the periodic case.

\begin{prop}[periodic case]\label{periodicmain}
    If $(x,t_0)$ is periodic, then $\Lambda_\R(\a_0) = \L_\R(\a_0)$.
\end{prop}

\begin{proof}
    First, we show that $\Lambda_\R(\a_0)\subset \L_\R(\a_0)$.
    Since $\Lambda_\R(\a_0)$ is defined to be the smallest equivalence relation in $\RZ \times \RZ$ which contains $\L_\R(\H)\cup \Lambda(\a_0)$, and $\L_\R(\H) \subset \L_\R(\a_0)$, it suffices to show that $\Lambda(\a_0) \subset \L_\R(\a_0)$.

    Let $(\theta,\theta') \in \Lambda(\a_0)$.
    By Lemma \ref{Lc}, there exists a minimal integer $n$ such that $\{\tau_3^n(\theta),\tau_3^n(\theta')\} =\{ \vt^L,\vt^R \}$.
    According to Corollary \ref{coland}, we have $(\vt^L,\vt^R) \in \L_\R(\a_0)$.
    Thus, by induction, we may assume that $(\tau_3(\theta),\tau_3(\theta')) \in \L_\R(\a_0)$.
    By choosing $\i$, we may assume that $R_\a(\theta)$ and $R_\a(\theta')$ does not land at $z_{\theta''}(\a)$ for $\a \in \mathcal{R}\cup \{ \a_0 \}$ and $\theta'' \in L^- \cup L^+$.
    Fix $\a \in \mathcal{R}$, assume that $R_\a(\theta)$ and $R_\a(\theta')$ are joined by the left-right internal ray $\RRR_\a^{L,R}(v,t)$.
    By (I2) of Corollary \ref{intersection}, we know that $\RRR_\a^{L,R}(v,t)$ is contained in a connected component of $\C \setminus (C_\a(L^-)\cup C_\a(L^+))$.
    Let this connected component be $U_\a$.
    It follows that $R_\a(\theta)$ and $R_\a(\theta')$ belong to $U_\a$ for $\a \in \mathcal{R}$.
    Therefore, $\{ \theta,\theta'\}$ are in single connected component of $\RZ \setminus (L^-\cup L^+)$.
    Assume that $(\theta,\theta')\notin \L_\R(\a_0)$, then either $(\theta+1/3,\theta') \in \L_\Q(\a_0)$ or either $(\theta-1/3,\theta') \in \L_\Q(\a_0)$.
    Without loss of generality, we assume that $(\theta+1/3,\theta') \in \L_\Q(\a_0)$.
    By Lemma \ref{fact}, $R_{\a_0}(\theta+1/3)$ and $R_{\a_0}(\theta')$ must belong to the same connected component of $\C \setminus ( C_{\a_0}(L^-) \cup C_{\a_0}(L^+) )$.
    This is impossible since $\theta+1/3$ and $\theta'$ are in distinct connected components of $\RZ \setminus (L^-\cup L^+)$.
    Therefore, $(\theta,\theta') \in \L_\Q(\a_0)$.
    This shows that $\Lambda(\a_0) \subset \L_\R(\a_0)$.

    In the following, we show that $\Lambda_\R(\a_0) \supset \L_\R(\a_0)$.
    By Theorem \ref{kiwicontain} and Proposition \ref{visuallamination}, it suffices to show that $\Lambda_\R(\a_0) \supset \L_\Q(\a_0)$.
    Suppose that $(\theta,\theta') \in \L_\Q(\a_0)$, If $R_{\a_0}(\theta)$ and $R_{\a_0}(\theta')$ land at a repelling preperiodic point, then by Lemma \ref{stablerepelling}, we have  $(\theta,\theta') \in \L_\Q(\H) \subset \Lambda_\R(\a_0)$.
    Now we assume that $R_{\a_0}(\theta)$ and $R_{\a_0}(\theta')$ land at a parabolic preperiodic point $z_0$.
    First, we assume that $z_0$ is a parabolic periodic point. 
    As before, we assume it is a parabolic fixed point.
    As we notice in Remark \ref{parabolicperturb}, when the parameter $\a$ enters the hyperbolic component $\H$,  $z_0$ perturbs into a repelling fixed point, and a repelling $p$ periodic orbit.
    It follows that $R_\a(\theta)$ and $R_\a(\theta')$ land at these $p+1$ points.
    According to Remark \ref{parabolicperturb}, we may assume that the fixed point is the landing point of $R_\a(\vt^L)$, and $R_\a(\vt^R)$ lands on the $p$ periodic cycle.
    Let $E$ be the equivalence class of $\Lambda_\R(\a_0)$ containing $\vt^L,\vt^R$.
    By Proposition \ref{visuallamination}, we have $\tau_3^k(\vt^\iota) \in E$ for $0\leq k \leq p-1$ and $\iota \in \{ L,R \}$.
    For each point among these $p+1$ points, there exists $\theta'' \in E$ such that $R_\a(\theta'')$ lands at it.
    Therefore, we conclude that $\theta,\theta' \in E$, i.e. $(\theta,\theta') \in \Lambda_\R(\a_0)$. 
   
    Now we consider the case that $z_0$ is a parabolic strictly preperiodic point.
    Let $n$ be the smallest integer such that $f_{\a_0}^n(z_0)=\psi_{\a_0}^x(e^{2\pi i t_0})$.
    By the discussion of the periodic case, $\tau_3^n(\theta),\tau_3^n(\theta')$ belong to an equivalence class $E$ of $\Lambda_\R(\a_0)$.
    Moreover, $R_\a(\tau_3^n(\theta))$ and $R_\a(\tau_3^n(\theta'))$ land at the $p+1$ points which converges to $ \psi_{\a_0}^x(e^{2\pi i t_0}) $ for $\a \in \H$.
    It follows that $R_\a(\theta)$ and $R_\a(\theta')$ land at $n$-th preimages of these $p+1$ points.
    Let $E'$ be the equivalence class of $\Lambda_\R(\a_0)$ containing $\theta$.
    We prove $\theta' \in E'$ by contradiction.
    Suppose that $\theta' \notin E'$, then by Proposition \ref{visuallamination}, there exists $\theta''\in E$ such that $\tau_3^n(\theta'')= \tau_3^n(\theta') $.
    Since $\Lambda_\R(\a_0) \subset \L_\R(\a_0)$,  $R_{\a_0}(\theta) $ and $R_{\a_0}(\theta'')$ land at a common point.
    But by our assumption $(\theta,\theta') \in \L_\Q(\a_0) $, we conclude that $R_{\a_0}(\theta')$ and $R_{\a_0}(\theta'')$ land at a common point.
    Thus, the common landing point is a preimage of a critical point of $f_{\a_0}^n$.
    This contradicts the fact that we are in the parabolic case.
    This finishes the proof.

\end{proof}

\begin{prop}[non-periodic case]\label{nonperiodicmain}
    If $(x,t_0)$ is non-periodic, then we have $\L_\R(\a_0)=\Lambda_\R(\a_0)$.
\end{prop}

\begin{proof}
    Corollary \ref{nonper} indicates that $ \Lambda_\R(\a_0) \subset \L_\R(\a_0) $.
    In the following, we show that $ \Lambda_\R(\a_0) \supset \L_\R(\a_0) $.
    We distinguish the following two cases.

    {\bf post-critically finite case:}
    Suppose that $f_{\a_0}$ is post-critically finite, again by Theorem \ref{kiwicontain} and Proposition \ref{visuallamination}, it suffices to show that $\L_\Q(\a_0) \subset \Lambda_\R(\a_0) $.
    Suppose that $(\theta,\theta') \in \L_\Q(\a_0)$, and $R_{\a_0}(\theta)$ lands at $z_0$.
    If the forward orbit of $z_0$ does not contain the critical point $-c(\a_0)$, then by Lemma \ref{stablerepelling}, we see that for $\a$ near $\a_0$, the external rays $R_\a(\theta)$ and $R_\a(\theta')$ both land at the perturbed pre-periodic point $z_0(\a)$.
    Therefore, we see that $(\theta,\theta') \in \L_\R(\H) \subset \Lambda_\R(\a_0) $ holds.
    Now we consider the case that $z_0=\omega(\a_0)$ is preimage of $-c(\a_0)$.
    Let $n$ be the unique integer such that $f_{\a_0}^n(z_0)=-c(\a_0)$.
    Since the forward orbit of $u(\a_0):=f_{\a_0}(-c(\a_0))$ does not contain $-c(\a_0)$, then by the discussion of the former case, for $\a$ near $\a_0$, $R_\a(\tau_3^{n+1}(\theta))$ and $R_\a(\tau_3^{n+1}(\theta'))$ both land at $u(\a)$ which is the perturbed pre-periodic point of $u(\a_0)$.
    Similarly, since the smooth internal ray $R_{\a_0}^{x_1}(t_1)$ lands at $u(\a_0)$, for $\a$ near $\a_0$, $R_{\a}^{x_1}(t_1)$ also lands at $u(\a)$ where $(x_1,t_1):=\g(x,t_0)$.
    Pick $m$ large so that the map $f_{\a_0}^{n+1}: P_{m+n+1}^{\a_0}(z_0) \to P_m^{\a_0}(u(\a_0))$ has degree $2$.
    For $\a \in \mathcal{R}$ which is close to $\a_0$, denote $P_\a:=h_{m+n+1}(\a,P_{m+n+1}^{\a_0}(z_0) )$ and  $P'_\a:=h_m(\a,P_{m}^{\a_0}(u(\a_0)) )$ where $h_m$ and $h_{m+n+1}$ are holomorphic motions given in Lemma \ref{hmpuzzle}.
    By Lemma \ref{hmpuzzle}, for $\a \in \mathcal{R}$ which is close to $\a_0$, we see that $u(\a) \in P_\a'$, $z_\theta(\a), z_{\theta'}(\a) \in P_\a$, and the map $f_\a^{n+1} : P_\a \to P'_\a$ also has degree $2$.
    For $\a \in \mathcal{R}$, $\omega(\a_0)$ perturbs into an iterative pre-critical point $\omega(\a)$ in $P_\a$ connecting a pair of left and right turning internal ray $\RRR_\a^{L,R}(v,t)$.
    It follows that the two end points $z^L_\a(v,t)$ and $z^R_\a(v,t)$ of $\RRR_\a^{L,R}(v,t)$  both belong to $P_\a$ and must be mapped to $u(\a)$ by $f_\a^{n+1}$.
    Hence $z^L_\a(v,t)$ and $z^R_\a(v,t)$ are the only two $n+1$-th preimage of $u(\a)$ in $P_{n+1}^\a$.
    Therefore, we have $\{ z_\theta(\a), z_{\theta'}(\a) \}=\{ z^L_\a(v,t), z^R_\a(v,t) \}$.
    This implies that $(\theta,\theta') \in \Lambda_\R(\a_0)$.

    {\bf post-critically infinite case:}
    Now we consider the case that the forward orbit of $-c(\a_0)$ is infinite.
    We claim that $\L_\Q(\a_0)\subset \L_\Q(\H)$.
    In fact, for any $(\theta,\theta') \in \L_\Q(\a_0) $, $R_{\a_0}(\theta)$ and $R_{\a_0}(\theta')$ land at a repelling preperiodic point whose forward orbit omits $-c(\a_0)$.
    By Lemma \ref{stablerepelling}, for $\a$ near $\a_0$, the external rays $R_\a(\theta)$ and $R_\a(\theta')$ land at a common point.
    Thus $(\theta,\theta') \in \L_\Q(\H)$.
    
    Therefore, we only need to consider the case that $(\theta,\theta') \in \L_\R(\a_0)$ with $\theta\neq \theta'$, and $R_{\a_0}(\theta)$ lands at $z_0$ which has infinite forward orbit.

    \begin{clm}\label{clm2}
        Let $B$ be a Fatou component of $f_{\a_0}$ and $z \in \partial B$ is not an iterative preimage of $-c(\a_0)$.
        Suppose that two distinct external rays $R_{\a_0}(\theta_1)$ and $R_{\a_0}(\theta_2)$ land at $z$, then $\theta_1,\theta_2$ are both rational.
    \end{clm}

    \begin{proof}[Proof of Claim \ref{clm2}]
        We prove it by contradiction, if $\theta_1,\theta_2$ are not rational, then $z$ has infinite forward orbit.
        Since $z$ is not an iterative preimage of $-c(\a_0)$, for $n \geq 0$, external rays $R_{\a_0}(\tau_3^n(\theta_1))$ and $R_{\a_0}(\tau_3^n(\theta_2))$ are two distinct external rays landing at $f_{\a_0}^n(z)$.
        Thus, without loss of generality, we assume that $z \in \partial B_{\a_0}(x)$.
        By Theorem \ref{roeschyin}, $K_{\a_0}$ can be decomposed into $B_{\a_0}(x)$ and limbs $L_w$ for $w \in \partial B_{\a_0}(x)$.
        Since $f_{\a_0}^p$ has finitely many critical points, $L_{f_{\a_0}^n(z)}=\{ z \}$ for $n$ large.
        This contradicts the fact that there are at least two external rays landing at $f_{\a_0}^n(z)$.    
    \end{proof}

    We continue the proof in the post-critically infinite case.
    There are two sub-cases.    
    Suppose that $z_0$ is not on the boundary of any bounded Fatou component, in this case, we show that $(\theta,\theta') \in \L_\R(\H)$.
     By Lemma \ref{singleton}, it suffices to show that
    $$P_n^\a(z_\theta(\a))=P_n^\a(z_{\theta'}(\a)),\quad \forall n \geq 1, \a \in \mathcal{R}.$$ 
    Fix $n$, define a modified puzzle $Q_n$ of $z_0$.
    If the boundary of $P_n^{\a_0}(z_0)$ does not contain any equipotential line in the bounded Fatou component, then set $Q_n=P_n^{\a_0}(z_0)$.
    On the contrary, we assume that $\partial P_n^{\a_0}(z_0) $ contains equipotential lines in Fatou components $B_{\a_0}(v_1), B_{\a_0}(v_2),\dots,B_{\a_0}(v_q)$.

    \begin{clm}\label{clm3}
        For each $1\leq k \leq q$, there exist $\theta_k,\theta_k' \in \Q$ such that $R_{\a_0}(\theta_k)$ and $R_{\a_0}(\theta_k')$ lands at a common point and separates $z_0$ away from $B_{\a_0}(v_k)$.
    \end{clm}

    \begin{proof}[Proof of Claim \ref{clm3}]
        Since $z_0 \notin \partial B_{\a_0}(v_k)$, by Theorem \ref{roeschyin}, it belongs to a unique non-singleton limb $L_w$ for some $w \in \partial B_{\a_0}(v_k)$.
        Therefore, there exists two external rays $R_{\a_0}(\theta_1)$ and $R_{\a_0}(\theta_2)$ landing at $w$ and separating $z_0$ away from $B_{\a_0}(v_k)$.
        If $w$ is not an iterative preimage of $-c(\a)$, by Claim \ref{clm2}, $\theta^-,\theta^+$ are both rational.
        In this case, set $\theta_k,\theta_k'$ to be $\theta^-,\theta^+$.

        Now suppose that $w$ is an $n$-th preimage of $-c(\a_0)$.
        It follows that there exists $B_{\a_0}(u_1)$ such that $\partial B_{\a_0}(u_1) \cap \partial  B_{\a_0}(v_k)=\{ w \}$.
        Since $z_0 \notin \partial B_{\a_0}(u_1)$, again by Theorem \ref{roeschyin}, there exists a pair of external rays $R_{\a_0}(\theta_1^-),R_{\a_0}(\theta_1^+)$ landing at $w_1 \in \partial B_{\a_0}(u_1)$ and separating $z_0$ away from $B_{\a_0}(u_1)$.
        If $w_1$ is not an iterative preimage of $-c(\a_0)$, then set $\theta_k,\theta_k'$ to be $\theta^-_1,\theta^+_2$.
        If not, $w_1$ must be the common boundary of $B_{\a_0}(u_1)$ and some other Fatou component $B_{\a_0}(u_2)$.
        Repeating the procedure, either we find rational $\theta_k,\theta_k'$ we desired in finite many steps, or we obtain a sequence of Fatou components $\{B_{\a_0}(u_n)\}$ with $\partial B_{\a_0}(u_n) \cap \partial B_{\a_0}(u_{n+1})=\{ w_{n+1} \}$, irrational pairs $(\theta^-_n,\theta_n^+) \in \L_{\R}(\a_0)$ such that $R_{\a_0}(\theta_n^-)$ and $R_{\a_0}(\theta_n^+)$ land at $w_n$.
        Without loss of generality, we assume that $(\theta,\theta')_+\subset (\theta_n^-,\theta_n^+)_+ \subset (\theta_{n-1}^-,\theta_{n-1}^+)_+$ for all $n \geq  2$.
        By Theorem \ref{kiwicontain}, since $\L_\R(\a_0)$ is closed, there exists $(\vt^-,\vt^+) \in \L_\R(\a_0)$ and  $(\theta,\theta')_+\subset (\vt^-,\vt^+)_+ \subset (\theta_{n}^-,\theta_{n}^+)_+$.
        In the following, we will show that $(\vt^-,\vt^+)$ must be rational.

     \begin{figure}[h]
  \begin{center}
    \includegraphics[width=0.75\linewidth]{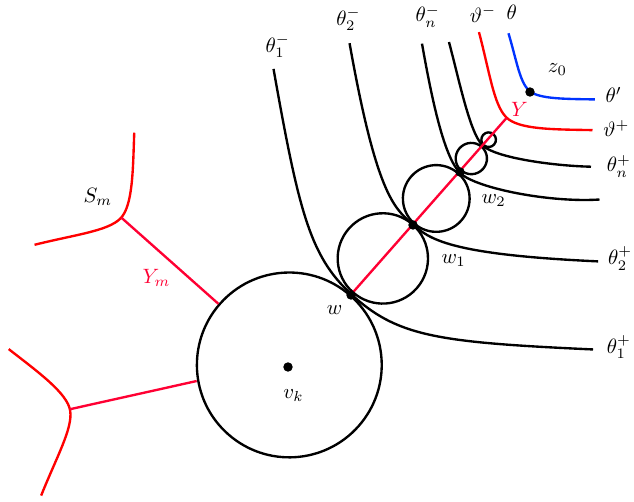}
    \caption{}
 \end{center}
\end{figure}  
        
        Let $R_n$ be the union of two internal rays inside $B_{\a_0}(u_n)$ joining $w_n$ and $w_{n+1}$.
        Define the connected set $Y:=\left(\bigcup_n R_n\right) \cup \overline{R_{\a_0}(\vt^-)}\cup \overline{R_{\a_0}(\vt^+)}$.
        Passing to iteration, we may assume that $Y$ is ``standing'' on the boundary of $B_{\a_0}(c)$, i.e. $v_k=c$ and $z_0=-c(\a_0)$.
        Let $Y_m:=f_{\a_0}^{mp}(Y)$ for $m \geq 0$.
        One may verify that $Y_m$ is also standing on $B_{\a_0}(c)$ and passing through infinitely internal rays and reaching $R_{\a_0}(\vt^-_m),R_{\a_0}(\vt_m^+)$.
        We may assume that $S_m:=S_{\a_0}(\vt^-_m,\vt^+_m)_+$ does not contain $B_{\a_0}(c)$.
        If $(\vt^-,\vt^+)$ is not a rational pair, then $S_m,m\geq 1$ are disjoint.
        Since $f_{\a_0}^p$ has finitely many critical points, it follows that for $m$ large, $S_m$ does not contain any critical point of $f_{\a_0}^p$. 
        This means the angular width of $\Delta(\vt^-_{m+1},\vt_{m+1}^+)_+=3\Delta(\vt^-_m,\vt^+_m)_+$ for $m$ large.
        This is clearly impossible.
        Thus $(\vt^-,\vt^+)$ must be rational.
        Set $\theta_k:=\vartheta^-$ and $\theta_k':=\vartheta^+$ for this case.
    \end{proof}

    By Claim \ref{clm3} and Lemma \ref{stablerepelling}, for $\a$ sufficiently close to $\a_0$, the external rays $R_\a(\theta_k)$ and $R_\a(\theta'_k)$ consistently land at a common point and separates $\{ z_\theta(\a),z_{\theta'}(\a)\}$ away from $v_k(\a)$.
    Define the modified puzzle piece $Q_n$ to be the connected component of $P_n^{\a_0}(z_0) \setminus  \bigcup_k (\overline{ R_{\a_0}(\theta_k) } \cup \overline{ R_{\a_0}(\theta_k')})$ containing $z_0$
    It follows that $Q_n \subset P_n^{\a_0}(z_0)$ does not contain any equipotential line on its boundary, i.e. it can only contains co-landing external rays with $\{(\theta_k,\theta_k'):1\leq l \leq \ell \}$ for some $\ell \geq q$, and equipotential lines outside the filled-in Julia set $K_{\a_0}$.
    As long as $K_\a$ is connected and $(\theta_k,\theta_k') \in \L_\R(\a)$ for $1\leq l \leq \ell$, the modified puzzle piece is always defined and denoted by $Q_n^\a$.
    Moreover, we know that $\{ z_\theta(\a), z_{\theta'}(\a) \} \subset Q_n^\a$ holds by the construction of $Q_n^\a$.
    By Lemma \ref{stablerepelling}, we know that for $\a_1 \in \mathcal{R}$ and sufficiently close to $\a_0$, $(\theta_k,\theta_k') \in \L_\R(\a_1)$.
    By Lemma \ref{surgery}, for any $\a \in \mathcal{R}$,  we also have $(\theta_k,\theta_k') \in \L_\R(\a)$. 
    It follows that $z_\theta(\a), z_{\theta'}(\a) \in Q_n^\a $ are contained in the same (unmodified) puzzle piece of depth $n$ for all $\a \in \mathcal{R}$.
    This complete the proof in this case.

    It remains to discuss the case that $z_0 \in \partial B_{\a_0}(v)$ which has infinite forward orbit.
    By Claim \ref{clm2}, $z_0=\omega(\a_0)$ must be an iterative preimage of $-c(\a_0)$ since it receives two external rays $R_{\a_0}(\theta)\neq R_{\a_0}(\theta')$ , and $u(\a_0):=f_{\a_0}(-c(\a_0))$ is the landing point of a unique external ray $R_{\a_0}(\vt)$.
    Let $\omega(\a)$ denote the iterative preimage of $-c(\a)$ for $\a$ near $\a_0$.
    This implies that there are only two external rays landing at $z_0$, they must be $R_{\a_0}(\theta)$ and $R_{\a_0}(\theta')$.
    Notice that for $\a \in \mathcal{R}$ near $\a_0$, the $\omega(\a)$ connects a pair of left and right turning internal ray $\RRR_\a^{L,R}(v',t)$.
    By Lemma \ref{impression}, $R_{\a_0}(\vt^L(v',t))$ and $R_{\a_0}(\vt^R(v',t))$ must land at $z_0$.
    Hence $\{ \theta,\theta' \}=\{ \vt^L(v',t), \vt^R(v',t) \}$.
    This implies that $(\theta,\theta') \in \Lambda_\R(\a_0)$.

\end{proof}

Combining Proposition \ref{periodicmain} and Proposition \ref{nonperiodicmain}, we conclude that the visual lamination indeed coincides with the real lamination.

\begin{thm}[boundary lamination in $\S$]\label{main}
     Let $\H$ be a bounded hyperbolic component of type \A, \B, or \CCC, and $\a_0 \in \partial \H \cap \S$, then 
    $\Lambda_\R(\a_0) = \L_\R(\a_0).$
\end{thm}

\subsection{Proof of the Main Theorem}\label{ssec4.6}

In this part, we provide the proof of Theorem \ref{thmmain2}.
We need to transfer the main result Theorem \ref{main} in $\partial \H \cap \S$ to the tame boundary $\partial_{\rm t}\H$ in two complex dimension.
This is a standard application of quasiconformal surgeries which change the multiplier of the attracting cycle.
The following is a variation of \cite[Lemma 3.12]{wang2023}.

\begin{prop}[changing the multiplier of attracting cycle]\label{multiplier}
    Let $\H$ be a bounded hyperbolic component of type \A, \B, or \CCC.
    For every $\a_1 \in \p \H$, there exist $p \geq 1$ and $\a_0 \in \partial \H \cap \S$ such that $\L_\R(\a_0)=\L_\R(\a_1)$.
\end{prop}

Using Proposition \ref{multiplier} and Theorem \ref{main}, we provide the proof of Theorem \ref{thmmain2}.

\begin{proof}[Proof of Theorem \ref{thmmain2}]
    For every $\a \in \p \H$, by Proposition \ref{multiplier}, there exists $\a_0 \in \partial \H \cap \S$ such that $\L_\R(\a_0)=\L_\R(\a)$.
    According to Theorem \ref{main}, we have $\L_\R(\a_0)=\Lambda_\R(\a_0)$.
    By the definition of $\Lambda_\R(\a_0)$, we know that $\Lambda_\R(\a_0)= \langle \Lambda(\a_0) \cup \L_\R(\mathcal{H}) \rangle $.
    By Proposition \ref{Lc}, $\Lambda(\a_0)$ is a minimal $\tau_3$-invariant equivalence class.
    By changing the notation $\Lambda(\a_0)$ to $\Lambda(\a)$, the proof is done.

\end{proof}

Corollary \ref{semiconjugacy} is a direct corollary of Theorem \ref{thmmain2}.

    \begin{proof}[Proof of Corollary \ref{semiconjugacy}]
        For any $z \in J_{\a^*}$, $z$ is the landing point of at least one external ray $R_{\a^*}(\theta)$, define $\pi_\a(z)$ to be the landing point of $R_\a(\theta)$.
        Theorem \ref{laminationsubset} ensures that the map $\pi_\a$ is well-defined.
        The continuity of $\pi_\a$ follows from the continuity of $\psi_\a$ and $\psi_{\a^*}$.
        
    \end{proof}

Finally, we present the proof of Theorem \ref{degree3} as an application.

\begin{proof}[Proof of Theorem \ref{degree3}]
    Let $\H$ be a hyperbolic component of type (A), (B) or (C), and $\a_1 \in \H$.
    Pick any $t_0$ which is irrational.
    Let $\mathcal{R}:=\mathcal{R}_{\H}(t_0)$ be a parameter internal ray with angle $t_0$.
    Let $\a \in \S \cap \partial \H$ be the landing point of $\mathcal{R}$.
    By Theorem \ref{laminationsubset}, we have $\L_\R(\a) \neq \L_\R(\a_1)$.
    By Theorem \ref{main}, $\lambda_\R(\a)=\langle \Lambda(\a) \cup \L_\R(\H) \rangle$.
    By Lemma \ref{Lc}, non-singleton equivalence classes of $\Lambda(\a)$ are all eventually mapped to $\{ \vt^L(x,t_0), \vt^R(x,t_0) \}$.
    These two characteristic angles $ \vt^L(x,t_0), \vt^R(x,t_0)$ are irrational provided $t_0$ is irrational.
    It follows that $\lambda_\Q(\a)=\lambda_\Q(\a_1)$.
    Thus, $f_{\a_1}$ is not combinatorially rigid.

\end{proof}

\bibliographystyle{plain}
\bibliography{reference} 
\end{document}